\documentclass[11pt]{article}
\usepackage[utf8]{inputenc}
\usepackage{lmodern}
\usepackage[normalem]{ulem} %cross out a word
\usepackage{float}%for algorithm's position
\usepackage[margin=0.8in]{geometry}
\usepackage{enumitem}

\usepackage{graphicx, titlesec}
\usepackage{amscd,amsmath,amsthm,amssymb}
\usepackage{verbatim}
\usepackage[all]{xy}
\usepackage{booktabs}
\usepackage{tikz-cd}

\tikzset{roundnode/.style={circle,draw=black!50,fill=black!100,inner sep=1.2pt}}

\usepackage{ytableau}
\usepackage[algonl,boxed,norelsize,lined]{algorithm2e}
\usepackage[bookmarksnumbered=true]{hyperref} 
\hypersetup{
     colorlinks = true,
     linkcolor = blue,
     anchorcolor = blue,
     citecolor = teal,
     filecolor = blue,
     urlcolor = blue
     }

\usepackage{wasysym}
\usepackage{graphicx}
\usepackage{hyperref}
\usepackage{bbold}
\usepackage{color}
\usepackage{setspace}
\usepackage{comment}
\usepackage{multirow}
\usepackage{ulem}
\usepackage{bm}
\usepackage{kbordermatrix}

\DeclareMathOperator{\rank}{rank}
\newcommand{\kk}{\mathbb{F}}
\newcommand{\J}{{\rm J}}
\newcommand{\V}{{\mathcal{V}}}

\newtheorem{theorem}{Theorem}[section]

\theoremstyle{definition}
\newtheorem{definition}[theorem]{Definition}
\newtheorem{remark}[theorem]{Remark}
\newtheorem{example}[theorem]{Example}

\theoremstyle{plain}
\newtheorem{corollary}[theorem]{Corollary}
\newtheorem{lemma}[theorem]{Lemma}
\newtheorem{prop}[theorem]{Proposition}
\newtheorem{conjecture}[theorem]{Conjecture}

\newtheorem{claim}{Claim}
\newtheorem{setup}[theorem]{Setup}

\newcommand{\MV}{V}

\newcommand{\cX}{{\cal X}}
\newcommand{\cS}{{\cal S}}
\newcommand{\bp}{{\bm p}}

\newcommand{\blfootnote}[1]{{%
  \let\thempfn\relax% Remove footnote number printing mechanism
  \footnotetext[0]{#1}% Print footnote text
}}

\DeclareMathOperator{\rk}{rank} 
\DeclareMathOperator{\perm}{perm} 
\DeclareMathOperator{\Span}{span}
\definecolor{mediumtealblue}{rgb}{0.0, 0.33, 0.71}
\definecolor{caribbeangreen}{rgb}{0.0, 0.8, 0.6}
\definecolor{capri}{rgb}{0.0, 0.75, 1.0}

\title{Identifiability of Points and Rigidity of Hypergraphs\\ under Algebraic Constraints\blfootnote{2020 Mathematics Subject Classification: Primary 52C25, 05E14; Secondary 05C65, 68R05, 14N07}.}
\author{James Cruickshank, Fatemeh Mohammadi, Anthony Nixon, and Shin-ichi Tanigawa}
\date{ }

\begin{document}

\maketitle

%\Tony{include word 'identifiability' in title? I like 'Identifiability of points under algebraic constraints' best so far...}\Fatemeh{How about ``Algebraic identifiability of points and rigidity of hypergraphs?" or ``Identifiability of points and rigidity of hypergraphs under algebraic constraints"?}

\begin{abstract}
The identifiability problem arises naturally in a number of contexts in mathematics and computer science. Specific instances include local or global rigidity of graphs and unique completability of partially-filled tensors subject to rank conditions. The identifiability of points on secant varieties has also been a topic of much research in algebraic geometry. It is often formulated as the problem of identifying a set of points satisfying a given set of algebraic relations. A key question then is to prove sufficient conditions for relations to guarantee the identifiability of the points.
 
This paper proposes a new general framework for capturing the identifiability problem when a set of algebraic relations has a combinatorial structure and develops tools to analyse the impact of the underlying combinatorics on the local or global identifiability of points. Our framework is built on the language of graph rigidity, where the measurements are Euclidean distances between two points, but applicable in the generality of hypergraphs with arbitrary algebraic measurements. We establish necessary and sufficient (hyper)graph theoretical conditions for identifiability by exploiting techniques from graph rigidity theory and algebraic geometry of secant varieties. In particular our work analyses combinatorially the effect of non-generic projections of secant varieties.
\end{abstract}

\section{Introduction}

Suppose one is given a set of points in $\mathbb{R}^d$  whose positions are unknown and a measuring device which provides relations among those points. 
The fundamental question arises: Can the locations of the points be uniquely identified from the measurements?
This identifiability problem is ubiquitous  across various applications in data science and engineering. Furthermore, it appears in several context of mathematics such as the identifiability of secant varieties in algebraic geometry. 
A key question in the identifiability problem is to determine a sufficient condition for the observations to guarantee the unique identification of the points. Such a question is often challenging if the observations are not sufficiently generic.

In this paper, we introduce a general framework for capturing the identifiability problem when a set of algebraic relations has a combinatorial structure as follows.
Suppose there are $n$ unknown points $p_1, p_2, \dots, p_n$ in $\mathbb{R}^d$.
Let $g$ be a polynomial map, which is a measurement function representing a measurement device,
and suppose that the value of $g$ is determined for each tuple of $k$ points in $\mathbb{R}^d$.

To represent a set of possible observations, we utilise a $k$-uniform hypergraph denoted as $G=(V,E)$. In this case, $V$ corresponds to the set $\{1, \dots, n\}$ representing the points. Essentially, the observer can obtain measurements $g(p_{v_1},\dots, p_{v_k})$ for all $\{v_1, \dots, v_k\}\in E$.
Then the identifiability problem revolves around determining whether the polynomial system given by:
\begin{equation}\label{eq:intro}
g(x_{v_1},\dots, x_{v_k})=g(p_{v_1},\dots, p_{v_k}) \qquad (\{v_1,\dots, v_k\}\in E)
\end{equation}
has a unique solution (up to certain symmetry) in the variables $x_1, x_2, \dots, x_n$ in  $\mathbb{R}^d$. 
This formulation is based on graph rigidity theory, which specifically addresses the case when $g$ represents the Euclidean distance between two points. Our framework extends beyond this specific case, providing a general approach for the identifiability problem in cases involving algebraic relations with combinatorial structures.

Graph rigidity theory has a long history with its roots in mathematics - arising from Cauchy and Euler's investigations of Euclidean polyhedra~\cite{Cauchy} - and engineering - Maxwell's analysis of the stiffness of frames~\cite{maxwell}.
The question is related to various branches of mathematics such as graph theory, matroid theory, and algebraic geometry, and it also appears in various modern engineering topics such as molecular conformation, network localisation, and multi-agent formation control, see, e.g.,~\cite{whiteley1996some,connelly_guest_2022}.
The two central notions in rigidity theory are {\em global rigidity} and {\em local rigidity}, which concern whether the system (\ref{eq:intro}) has a unique solution or a finite number of solutions (up to Euclidean isometry), respectively, when $G$ is an ordinary graph and $g$ is the Euclidean distance.
Accordingly, one can define the notion of {\em local rigidity} and {\em global rigidity of hypergraphs} for each model of measurement map $g$ using the system (\ref{eq:intro}).

A key aim of the paper is to present a generalisation of graph rigidity theory that covers a wide range of algebraic topics by looking at general measurement maps $g$.
While the concept of graph rigidity is intuitive, several variants of the problem have already been explored. Notable examples include rigidity in different metric spaces such as spherical space or $\ell_p$-space, as well as rigidity concerning other geometric constraints (see, for instance, \cite{borcea, bulavka2022volume, Dewar23, SW07, singer2010uniqueness, kiraly2015algebraic}).
While it is indeed possible to consider more general algebraic systems, doing so may result in a loss of the combinatorial perspective, which serves as the core of rigidity theory.
In this paper, we demonstrate that this new rigidity model of hypergraphs is general enough to address the identifiability problem across various %in several 
applications and it  also provides new mathematical challenges from the graph rigidity viewpoint.

%To demonstrate that our generalised theory has natural applications we will apply it to several examples. 
One prominent example of our rigidity model is the uniqueness problem of low-rank tensor completions. Identifiability in tensor decompositions arises naturally in a panorama of application areas, including phylogenetics \cite{phylogenetics}, quantum information \cite{grover} and signal processing \cite{signal}. Our model can further deal with situations when some entries are missing, thus enhancing its applicability.
Our rigidity-based approach to the unique low-rank tensor completion problem can be viewed as an extension of the rigidity-based analysis of low-rank matrix completions pioneered by 
Singer-Cucuringu~\cite{singer2010uniqueness} and Kir{\'a}ly-Theran-Tomioka~\cite{kiraly2015algebraic}. However, as we will demonstrate in Section~\ref{sec:examples2},
the tensor case gives rise to substantially more challenging mathematical questions.

The key idea of combinatorial rigidity theory is the genericity assumption on point configurations.
In the case of local rigidity, under this genericity assumption, the core question becomes equivalent to determining the dimension of the image of a measurement map. Notably, these images typically correspond to projections of
well-studied algebraic varieties such as the secant variety of a Veronese variety or of a Grassmaniann variety. Consequently, our rigidity problem can be formulated as an algebraic sensing problem, 
similar to the approach presented 
in  \cite{breiding2021algebraic}. While Noether's normalization lemma shows how the dimension changes under a generic projection~\cite{breiding2021algebraic}, this general theory does not provide a solution to our specific problem. In our case, the projections are non-generic as they are consistently performed along a coordinate axis of the ambient affine space.

A particularly interesting case is when the measurement map is the sum of simpler measurement maps. 
In such a case, the image of the measurement map is a secant of a simpler variety.  
The computation of the dimension of secant varieties and the identifiability of secant points have been studied extensively in algebraic geometry (see, e.g.,~\cite{Zak,bernardi}), and results in that context give rise to rigidity theorems for complete hypergraphs in our language. 
For example, the Alexander-Hirschowitz theorem \cite{alexander1995polynomial} on the dimension of secant varieties of Veronese varieties provides a characterisation for complete hypergraphs to be locally rigid, specifically when the measurement map $g$ is the sum of copies of the product map over coordinates. These results highlight the connection between algebraic geometry and our notion of rigidity, offering new insights into the identifiability and rigidity properties of complete hypergraphs.
Subsequent results~\cite{chiantini2017generic,galuppi2019identifiability} on identifiability even give a characterisation of global rigidity for complete hypergraphs.
(More details will be discussed in Section~\ref{sec:examples2}.) 
However, there seems to be few results on the projections of secant varieties.
Our rigidity problem for general hypergraphs deals with two structures: the geometry of secant varieties and the combinatorics of coordinate projections.
The goal is to establish a link between these two structures, shedding light on their connections and implications. 

%Exploiting 
By leveraging techniques from both rigidity theory and secant varieties, we will derive several necessary and sufficient conditions for local and global rigidity.
%Among other things, we give a
Our contributions include a graph-packing-type sufficient condition for local rigidity and a stress-matrix-type sufficient condition for global rigidity.   
 We believe that our new rigidity model and the results of this article are the start of a larger effort to exploit the link between combinatorial rigidity and the geometry of secant varieties. This connection is likely to be instrumental in solving further rigidity and identifiability problems in the future. To justify this we provide two concrete applications of our work. The first concerns random projections of secant varieties along coordinate axes. 
 %Motivated by sampling complexity for low rank tensor completions, we establish the threshold
 %, in the Erdős-Rényi subgraph model, for a random hypergraph 
 %for which the dimension of a random axis-parallel projection of a secant variety is full dimensional.
 Our rigidity framework enables us to reduce this to the problem of analysing a property of the Erdős-Rényi random hypergraph, and we establish  the threshold dimension for a random axis-parallel projection of a secant variety to be full dimension with high probability.
%The second provides an analogue of the Cayley-Menger variety for $\ell_p$-norms and explains how our analysis leads to a characterisation of global rigidity in the $\ell_p$-plane. (This extends a result of \cite{Dewar23} which applies when $p$ is even.) 
The second analyses the identifiability of the $p$-Cayley-Menger variety (an analogue of the Cayley-Menger variety for $\ell_p$-norms) and global rigidity in the $\ell_p$-space.
A combination of tools from algebraic geometry such as tangential weak defectiveness and graph-theoretical tools from rigidity theory leads to a precise combinatorial characterisation of the 2-identifiability of the axis-parallel projections of the $p$-Cayley-Menger variety. 
This result is also important in rigidity theory since it affirmatively answers the question of whether or not global rigidity in the $\ell_p$-plane is a generic property of the graph when $p$ is even. 
 Both applications provide novel connections between our graph-theoretical approach of rigidity theory to the non-degeneracy and the identifiability of secant varieties in algebraic geometry.
\medskip 

\noindent{\bf Outline.} In Section \ref{sec:rigid}, we recall the basics of Euclidean rigidity theory and provide a new generalisation to the setting of arbitrary algebraic constraints on hypergraphs. Section \ref{sec:examples} presents a range of examples that 
are modelled by our generalisation,
including those arising from rigidity theory, matrix and tensor product completions and Chow decompositions. In Section \ref{sec:tools}, we provide an analysis of the basic theorems from rigidity theory in our general context providing, for example, a precise analogue of Asimow and Roth's theorem. This section also presents our generalised rigidity theory in the algebraic geometry context, in particular explaining the notions of defectivity and identifiability. Further, we define the generic $g$-rigidity matroid, compute the combinatorially expected rank of this matroid and pose the combinatorial characterisation problem.

Section \ref{sec:examples2} provides detailed information and references for what is known in each of the concrete instances from Section \ref{sec:examples}, particularly illustrating the abstract theory of Section \ref{sec:tools}. In section \ref{sec:suff}, we extend combinatorial techniques from graph rigidity theory to $k$-uniform hyper-frameworks. In particular, we present a technique to construct a larger rigid hyperframework from a given one, and, in one of our main results, we provide a 
%packing sufficient condition 
packing condition that is sufficient
for rigidity when the constraints arise from multilinear $k$-forms (Theorem \ref{thm:packing}). We then move on to analyse global rigidity for the generalised rigidity model in Section \ref{sec:global}, noting first that global rigidity is a generic property in the complex case but not in the real case. Our main results here are Theorems \ref{thm:global_determinant} and \ref{thm:global_tensor} which provide sufficient conditions for global rigidity in the determinant and tensor product cases, analogous to a well known sufficient condition in the Euclidean rigidity case due to Connelly \cite{Con05}. 

We then consider some variants of our generalisation in Section \ref{sec:variants}. Section \ref{sec:apps} contains two  applications of our work. Theorem \ref{thm:random} which is motivated by sampling complexity for low rank tensor completions shows the threshold, in the Erdős-Rényi subgraph model, for a random hypergraph for which the dimension of a random axis-parallel projection of a secant variety is full dimensional.
The second analyses global rigidity in the $\ell_p$-plane.
We finish with some comments for further exploration (Section \ref{sec:concluding}) including a potential application in algebraic statistics.

\medskip

\noindent{\bf Notation.} 
For a positive integer $n$, let $[n]=\{1,2, \dots, n\}$. 
Throughout the paper, let $\mathbb{F}$ be either the real field $\mathbb{R}$ or the complex field $\mathbb{C}$. 
For a finite set $V$, a map $p:V\rightarrow \mathbb{F}^d$ is called a {\em point-configuration} in $\mathbb{F}^d$.
We use $(\mathbb{F}^d)^{V}$ to denote the set of all point-configurations in $\mathbb{F}^d$. Obviously $(\mathbb{F}^d)^V$ is isomorphic to $\mathbb{F}^{d|V|}$, and hence we sometimes regard  each $p\in (\mathbb{F}^d)^V$ as a $d|V|$-dimensional vector.
If $\MV=[k]$, then $(\mathbb{F}^d)^{[k]}$ is simply denoted by $(\mathbb{F}^d)^k$.

We will need the following terminology related to hypergraphs. Let $\genfrac{\{}{\}}{0pt}{}{X}{k}$ be the set of all multisets of $k$ elements of a finite set $X$ and $X\choose k$ be the subset of $\genfrac{\{}{\}}{0pt}{}{X}{k}$ consisting of sets having no repeated elements. Throughout the paper, a {\em $k$-uniform hypergraph} $G$ is defined as a pair $G=(V, E)$ of a finite set $V$ and $E\subseteq \genfrac{\{}{\}}{0pt}{}{V}{k}$, and $G$ is said to be {\em simple} if $E\subseteq {V\choose k}$.
Note that $|\genfrac{\{}{\}}{0pt}{}{|V|}{k}|={|V|+k-1\choose k}$.
We use  $K_n^k$ to denote the complete $k$-uniform hypergraph with $n$ vertices, that is, 
$K_n^k=([n], \genfrac{\{}{\}}{0pt}{}{[n]}{k})$.
Similarly,  let $\tilde{K}_n^k=([n], {[n]\choose k})$ be the simple complete $k$-uniform hypergraph on $n$ vertices.
For a hyperedge $e\in  \genfrac{\{}{\}}{0pt}{}{V}{k}$ and $u\in e$, 
we use $e-u$ to denote the multiset obtained from $e$ by reducing the multiplicity of $u$ by one.
Similarly, for $v\in V$, let $e+v$ be the multiset obtained from $e$ by increasing the multiplicity of $v$ by one.
For $e\in  \genfrac{\{}{\}}{0pt}{}{V}{k}$ and $v\in V$,
the multiplicity of $v$ in $e$ is denoted by $m_e(v)$.

\section{Rigidity of Hypergraphs}
\label{sec:rigid}
\subsection{Euclidean rigidity}
In this subsection, we review basic terminologies of the Euclidean rigidity.
The ordinary Euclidean rigidity concerns the rigidity of graphs drawn in Euclidean space.
Such a realisation of a graph in Euclidean space is called a {\em (bar-and-joint) framework}.
Conventionally a framework is  defined as a pair $(G,p)$ of a finite graph $G$ and a point-configuration $p:V\rightarrow \mathbb{R}^d$. 
The framework is \emph{rigid} if the only edge-length preserving continuous motions of its vertices arise from isometries of $\mathbb{R}^d$, and otherwise it is called {\em flexible}. (See Figure \ref{fig:flex} for basic examples.) The study of the rigidity of frameworks has its origins in the work of Cauchy and Euler on Euclidean polyhedra \cite{Cauchy} and Maxwell \cite{maxwell} on frames. 

\begin{figure}[h]
\begin{center}
\begin{tikzpicture}[scale=0.8]
\filldraw (1,6) circle (3pt);
\filldraw (-1,6) circle (3pt);
\filldraw (1,8) circle (3pt);
\filldraw (-1,8) circle (3pt);

\draw (-.7,7.9) circle (3pt);
\draw (1.3,7.9) circle (3pt);

 \draw[black,thick]
  (1,6) -- (-1,6) -- (-1,8) -- (1,8) -- (1,6);

  \draw[dashed]
  (-1,6) -- (-.7,7.9) -- (1.3,7.9) -- (1,6);

%%%

\filldraw (5,6) circle (3pt);
\filldraw (7,6) circle (3pt);
\filldraw (5,8) circle (3pt);
\filldraw (7,8) circle (3pt);

 \draw[black,thick]
  (7,6) -- (5,6) -- (5,8) -- (7,8) -- (7,6) -- (5,8);
  
\end{tikzpicture}
\end{center}
\caption{Simple examples of frameworks in the Euclidean plane: (left) a flexible framework with motion indicated; and (right) a rigid framework.}
\label{fig:flex}
\end{figure}
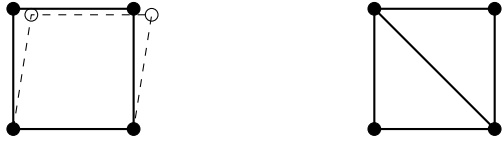

Abbot \cite{abbot} showed that it is  NP-hard to determine whether a given $d$-dimensional framework is rigid whenever $d\geq 2$. The problem becomes more tractable for generic frameworks $(G,p)$ since we can linearize the problem and consider `infinitesimal rigidity' instead. In this setting, the \emph{measurement map} $f_{G}:(\mathbb{R}^d)^{V}\rightarrow \mathbb{R}^{E}$ is defined by putting $f_{G}(p)=(\dots, \|p(i)-p(j)\|^2,\dots)_{ij\in E}$, where 
$\lVert.\lVert$ denotes the Euclidean metric.
The \emph{rigidity matrix} $R_d(G,p)=\frac{1}{2}df_{G}$ is the $|E|\times d|V|$ matrix in which, for $e=ij\in E$, the entries in row $e$ and columns $i$ and $j$ are $p(i)-p(j)$ and $p(j)-p(i)$, respectively, and all other entries in row $e$ are zero. Here, $g$ is the usual Euclidean distance.
% We say that $(G,p)$ is \emph{infinitesimally flexible} if there exists a function $\dot{p}:V\rightarrow \mathbb{R}^d$ such that 
A map $\dot{p}:V\rightarrow \mathbb{R}^d$ is called an {\em infinitesimal motion} of $(G,p)$ if
$$ \langle p(i)-p(j),\dot{p}(i)-\dot{p}(j)\rangle =0 \mbox{ for all } ij\in E,$$
where $\langle\cdot,\cdot\rangle$ denotes the Euclidean inner product.
An infinitesimal motion $\dot{p}$ is \emph{trivial} if there exists a skew-symmetric matrix $S$ and a vector $t$ such that $\dot{p}(i)=Sp(i)+t$ for all $i\in V$. A framework $(G,p)$ is \emph{infinitesimally rigid} if every infinitesimal flex of $(G,p)$ is trivial. Equivalently, $(G,p)$ is infinitesimally rigid if $G$ is complete on at most $d+1$ vertices or $|V|\geq d+2$ and $\rank\, R_d(G,p)=d|V|-{d+1 \choose 2}$.
Asimow and Roth \cite{asimow} showed that infinitesimal rigidity is equivalent to rigidity for generic frameworks (and hence that generic rigidity depends only on the underlying graph of the framework). Figure \ref{fig:inf} illustrates these concepts in dimension 2.

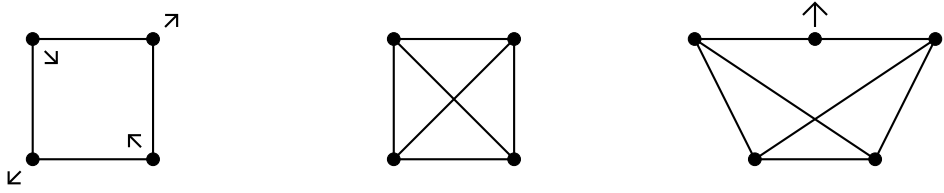
\begin{figure}[h]
\begin{center}
\begin{tikzpicture}[scale=0.8]
\filldraw (1,6) circle (3pt);
\filldraw (-1,6) circle (3pt);
\filldraw (1,8) circle (3pt);
\filldraw (-1,8) circle (3pt);

 \draw[black,thick]
  (1,6) -- (-1,6) -- (-1,8) -- (1,8) -- (1,6);

      \draw[black,thick]
    (-.8,7.8) -- (-.6,7.6);

  \draw[black,thick]
   (-.8,7.6) -- (-.6,7.6) -- (-.6,7.8);

    \draw[black,thick]
(.8,6.2) -- (.6,6.4);

 \draw[black,thick]
 (.6,6.2) -- (.6,6.4) -- (.8,6.4);

  \draw[black,thick]
  (-1.2,5.8) -- (-1.4,5.6);

   \draw[black,thick]
   (-1.2,5.6) -- (-1.4,5.6) -- (-1.4,5.8);

    \draw[black,thick]
    (1.2,8.2) -- (1.4,8.4);

     \draw[black,thick]
    (1.2,8.4) -- (1.4,8.4) -- (1.4,8.2);

\filldraw (7,6) circle (3pt);
\filldraw (5,6) circle (3pt);
\filldraw (7,8) circle (3pt);
\filldraw (5,8) circle (3pt);

 \draw[black,thick]
  (7,6) -- (5,6) -- (5,8) -- (7,8) -- (7,6) -- (5,8);

 \draw[black,thick]
 (5,6) -- (7,8);

%%%

\filldraw (13,6) circle (3pt);
\filldraw (11,6) circle (3pt);
\filldraw (10,8) circle (3pt);
\filldraw (12,8) circle (3pt);
\filldraw (14,8) circle (3pt);

 \draw[black,thick]
 (11,6) -- (13,6) -- (14,8) -- (11,6); 
 
  \draw[black,thick]
  (13,6) -- (10,8) -- (11,6);

   \draw[black,thick]
   (10,8) -- (12,8) -- (14,8);

    \draw[black,thick]
    (12,8.2) -- (12,8.6);

     \draw[black,thick]
   (11.8,8.4) -- (12,8.6) -- (12.2,8.4);

\end{tikzpicture}
\end{center}
\caption{(Left) An infinitesimal flex of the flexible framework in Figure \ref{fig:flex}. (Middle) A generically rigid graph whose edge set is dependent in the generic 2-dimensional rigidity matroid since $|E|=6>2|V|-3$. (Right) A generically rigid graph realised as a framework that has an infinitesimal flex (indicated).}
\label{fig:inf}
\end{figure}

The {\em generic $d$-dimensional rigidity matroid} of a graph $G=(V,E)$ is the matroid $\mathcal{R}_d(G)$ on $E$ in which a set of edges $F\subseteq E$ is independent if the corresponding rows of $R(G,p)$ are linearly independent, for some (or, equivalently, for every) generic $p$. 

A framework $(G,p)$ in $\mathbb{R}^d$ is \emph{globally rigid} if every $q\in f_G^{-1}(f_G(p))$ satisfies $t\cdot q=p$ for some isometry $t$ of $\mathbb{R}^d$ (i.e. some composition of translations and rotations).
See Figure \ref{fig:global} for an example.

\begin{figure}[h]
\begin{center}
\begin{tikzpicture}[scale=0.8]
\filldraw (0,0) circle (3pt);
\filldraw (2,0) circle (3pt);
\filldraw (-1,1.5) circle (3pt);
\filldraw (1,2) circle (3pt);
\filldraw (3,1.5) circle (3pt);
\filldraw (1,3) circle (3pt);

 \draw[black,thick]
  (0,0) -- (2,0) -- (3,1.5) -- (1,2) -- (-1,1.5) -- (0,0) -- (1,2) -- (2,0);

 \draw[black,thick]
 (-1,1.5) -- (1,3) -- (3,1.5);

 %%%

 \filldraw (7,0) circle (3pt);
\filldraw (9,0) circle (3pt);
\filldraw (6,1.5) circle (3pt);
\filldraw (8,2) circle (3pt);
\filldraw (10,1.5) circle (3pt);
\filldraw (8,0.5) circle (3pt);

 \draw[black,thick]
  (7,0) -- (9,0) -- (10,1.5) -- (8,2) -- (6,1.5) -- (7,0) -- (8,2) -- (9,0);

 \draw[black,thick]
 (6,1.5) -- (8,.5) -- (10,1.5);

 %%%%

 \filldraw (14,0) circle (3pt);
\filldraw (16,0) circle (3pt);
\filldraw (13,1.5) circle (3pt);
\filldraw (15,2) circle (3pt);
\filldraw (17,1.5) circle (3pt);
\filldraw (15,3) circle (3pt);

 \draw[black,thick]
  (14,0) -- (16,0) -- (17,1.5) -- (15,2) -- (13,1.5) -- (14,0) -- (15,2) -- (16,0);

 \draw[black,thick]
 (13,1.5) -- (15,3) -- (17,1.5);

 \draw[black,thick]
 (15,2) -- (15,3);
  
\end{tikzpicture}
\end{center}
\caption{(Left and middle) rigid but not globally rigid frameworks in the plane. (Right) a globally rigid framework in the plane.}
\label{fig:global}
\end{figure}
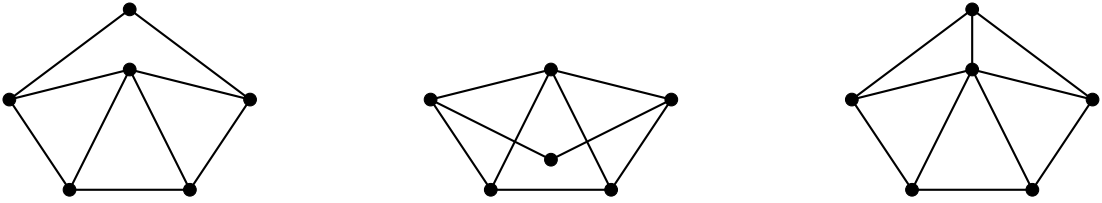

\subsection{Rigidity under algebraic constraints}
The concept of rigidity is flexible in the sense that all notations can be extended  by replacing the Euclidean distance function with a general smooth function. Since a measurement may involve more than two vertices, we will use hypergraphs to capture the underlying combinatorics.

Extending the central object from rigidity,  a pair $(G,p)$ of a hypergraph $G$ and a point-configuration $p\in (\mathbb{F}^d)^{V}$ is called a {\em $d$-dimensional hyper-framework} or a {\em hyper-framework in $\mathbb{F}^d$}.

\smallskip
Throughout the paper, we use the following setup.
\begin{setup}\label{setup}
Let $G=(V,E)$.
Suppose we are given a $k$-uniform hyper-framework $(G,p)$ and a polynomial map $g:(\mathbb{F}^{d})^k\rightarrow \mathbb{F}$.
The {\em $g$-measurement map} of $G$ is defined as a polynomial map $f_{g,G}:(\mathbb{F}^d)^{V}\rightarrow \mathbb{F}^E$ that sends $p$ to the list of the $g$-values of the tuples $(p(v_1), p(v_2), \dots, p(v_k))$ over the hyperedges $\{v_1,\ldots,v_k\}$ in $E$, i.e., 
\[
f_{g,G}(p):=\left(g(p(v_1), p(v_2), \dots, p(v_k)): e=\{v_1,\dots, v_k\}\in E\right).
\]
\end{setup}

In terms of the discussion in the introduction, $g$ represents a measurement device and $f_{g,G}(p)$ is the actual list of measurements that the observer can obtain. The observer is asked to recover $p$ from $f_{g,G}(p)$. The $g$-rigidity defined below captures the identifiability of $p$.

Note that, in order to make $f_{g,G}$ well-defined, $g$ must be either {\em symmetric} or {\em anti-symmetric} with respect to the ordering of the points, and if $g$ is anti-symmetric 
we always assume that $v_1,\dots, v_k$ are aligned in increasing order in $g(p(v_1),\dots, p(v_k))$, assuming a (fixed) total order on the vertices of $G$.

In most practical applications, there is a nontrivial group action that does not change the value of the $g$-measurement map, and rigidity has to be defined by taking care of the degree of freedom caused by such actions.
Suppose the general affine group ${\rm Aff}(d,\mathbb{F})$ acts on $\mathbb{F}^d$ by $\gamma\cdot x= Ax+t$ for $x\in \mathbb{F}^d$ and each pair $\gamma=(A,t)$ of $A\in {\rm GL}(d,\mathbb{F})$ and $t\in \mathbb{F}^d$.
The action of ${\rm Aff}(d,\mathbb{F})$ on $(\mathbb{F}^d)^{V}$ is also defined by $(\gamma \cdot p)(v)=Ap(v)+t\ (v\in V)$ for any $\gamma=(A,t)\in {\rm Aff}(d,\mathbb{F})$ and $p\in (\mathbb{F}^d)^{V}$.
Then the induced action on a polynomial map $g:(\mathbb{F}^d)^k\rightarrow \mathbb{F}$ is given by 
$\gamma\cdot g(x_1,\dots, x_k)=g(\gamma^{-1}\cdot x_1,\dots, \gamma^{-1} \cdot x_k)$ for $x_1,\dots, x_k\in \mathbb{F}^d$ and $\gamma\in {\rm Aff}(d,\mathbb{F})$.
We say that $\gamma$ {\em stabilises}  $g$ if $g$ is invariant by the action of $\gamma$. 
The set of pairs $(A,t)$ that stabilise $g$ forms a subgroup of ${\rm Aff}(d,\mathbb{F})$, called the {\em stabiliser} $\Gamma_g$  of $g$.
Since the action is smooth, $\Gamma_g$ is a closed subgroup of ${\rm Aff}(d,\mathbb{F})$, so $\Gamma_g$ is also a Lie group by the closed-subgroup theorem.

\smallskip
We are now ready to give a formal definition of $g$-rigidity.
\begin{definition}
    We say that $(G,p)$ is {\em globally $g$-rigid}  if 
for any $q\in f^{-1}_{g,G}(f_{g,G}(p))$ there is $\gamma \in \Gamma_g$ such that 
$q=\gamma \cdot p$.
We say that $(G,p)$ is {\em locally $g$-rigid} if there is an open neighbourhood $N$ of $p$ in $(\mathbb{F}^d)^{V}$ (in the Euclidean topology) such that
for any $q\in f^{-1}_{g,G}(f_{g,G}(p))\cap N$ there is $\gamma \in \Gamma_g$ such that $q=\gamma \cdot p$.
\end{definition}

\begin{remark}
In ordinary Euclidean rigidity, the following fact is implicitly used to capture rigidity in several different but equivalent forms.

\begin{prop}\label{prop:Euclidean_isometry}
Suppose $V$ is a finite set and $p, q\in (\mathbb{R}^d)^{V}$.
Then 
\[
\|p(u)-p(v)\|=\|q(u)-q(v)\|\qquad (u, v\in V)
\]
holds if and only if 
there is an Euclidean isometry $(A,t)$ such that $q(u)=Ap(u)+t\ (u\in V)$.
\end{prop}

Let $|V|=n$ and $K_{n}^2$ denote the complete graph on the vertex set $V.$
Proposition~\ref{prop:Euclidean_isometry} implies that the point configuration is uniquely determined up to Euclidean isometry by measuring inter-point distances for all edges in 
$K_{n}^2$,
and hence the framework of a complete graph is always locally/globally rigid in the ordinary Euclidean rigidity. 
By this fact, we may also define the local/global rigidity of a graph $G$ by comparing $f^{-1}_{g,G}(f_{g,G}(p))$ with $f^{-1}_{g,K_n^2}(f_{g,K_n^2}(p))$.
However, for a general measurement $g$, the analogue of Proposition~\ref{prop:Euclidean_isometry} is no longer true, that is if we replace $\lVert.\lVert$ with a general measurement $g$, the statement does not hold in general.
A notable example is the case of symmetric tensor completion,
where the Alexander-Hirschowitz theorem (cf.~Theorem~\ref{thm:ah}) shows that Proposition~\ref{prop:Euclidean_isometry} does not hold for the corresponding polynomial map $g$. A more basic example is provided by rigidity theory under a smooth $\ell_p$ norm; as described in \cite[Figure 2]{Dewar22gen}, the proposition fails in the $\ell_4$-plane where two realisations of the complete graph $K_4$ can have the same edge lengths but there is no isometry mapping one onto the other.
In general, the validity of the counterpart statement to Proposition~\ref{prop:Euclidean_isometry} also depends on the number of vertices $|V|$ and the size of each hyperedge $k$. It can be either true or false in different cases. However, our definition of $g$-rigidity for hyper-frameworks, as stated earlier, is formulated to solely depend on the measurement map $g$ itself, remaining independent of factors such as $|V|$ and $k$.
\end{remark}

\section{Examples of $g$-rigidity Models}\label{sec:examples}
We give a list of primary examples of $g$-rigidity. 

\paragraph{Ordinary Euclidean rigidity.}
A fundamental example is the case when 
$G$ is 2-uniform (i.e., a graph) and
$g(x,y)=(\lVert x-y\lVert_2)^2=\sum_{i=1}^d (x_i-y_i)^2$ for $x,y\in \mathbb{R}^d$.
Then the Euclidean group $E(d)$ is the stabilizer of $g$,
and its Lie algebra is the set of pairs $(S,t)$ of skew-symmetric matrices $S$ and $t\in \mathbb{R}^d$.
In this case, $g$-rigidity is nothing but the standard rigidity of bar-and-joint frameworks.

\paragraph{Rigidity in pseudo-Euclidean space.} 
%A closely related example was first considered by Pogorelov, who in \cite[Chapter V]{Pogorelov} observed that the space of infinitesimal motions of a bar-joint framework on a semi-sphere is isomorphic to those of the framework obtained by a central projection to Euclidean space.
%More generally, changing the underlying metric of the space to a pseudo-Euclidean metric, such as Minkowski space, does not change the infinitesimal rigidity properties of the framework; this is a consequence of the fact that infinitesimal rigidity is preserved by projective transformations. 
In the pseudo-Euclidean rigidity context
$G$ is 2-uniform (i.e., a graph) and 
$g(x,y)=\sum_{i=1}^{d_1} (x_i-y_i)^2-\sum_{i=d_1+1}^{d} (x_i-y_i)^2$ for $x,y\in \mathbb{R}^d$ and $d_1\leq d$. 

\paragraph{$\ell_p$-norm rigidity.}
An alternative generalisation of Euclidean rigidity is to allow the distance function to be replaced by distance under another norm \cite{KP14}. Specifically $G$ is 2-uniform (i.e., a graph) and 
$g(x,y)=(\lVert x-y\lVert_p)^p=\sum_{i=1}^d |x_i-y_i|^p$ for $x,y\in \mathbb{R}^d$ and $1<p<\infty$. In the case where $p$ is a positive even integer, $g$ is a polynomial. 
%In the remainder of the paper, whenever we refer to the $\ell_p$-norm we are assuming that $p$ is a positive even integer.

\paragraph{Volume-constrained rigidity.}
Given a $d$-dimensional pure simplicial complex realised in $\mathbb{R}^d$,
the notion of volume-constrained rigidity concerns whether there is a motion of vertices keeping the (signed) volume of each $d$-simplex \cite{bulavka2022volume}. % \textcolor{brown}{Reference?}
A $d$-dimensional pure simplicial complex can be identified with a $(d+1)$-uniform hyper-framework $(G,p)$ with a simple hypergraph $G$.
Hence, the volume-constrained rigidity can be captured as the $g$-rigidity of a simple $(d+1)$-uniform hyper-framework $(G,p)$ with $g:(\mathbb{R}^d)^{d+1}\rightarrow \mathbb{R}$ defined by 
\[
g(x_1, x_2, \dots, x_{d+1})=\det \begin{pmatrix} x_1 & x_2 & \dots & x_{d+1} \\
1 & 1 & \dots & 1\end{pmatrix}.
\]
A rigidity-based analysis was initiated by Borcea and Streinu~\cite{borcea}. Recent studies can be found in \cite{bulavka2022volume,Southgate}. 

\paragraph{Positive semidefinite symmetric matrix completion.} 

Let $T$ be a positive semidefinite symmetric matrix of size $n$ over $\mathbb{R}$.
If the rank of $T$ is $d$, then the spectral decomposition implies that 
\begin{equation}\label{eq:symmetric_matrix}
T= \sum_{i=1}^{d} x_i x_i^{\top}
\end{equation}
for some  vectors $x_1, x_2, \dots, x_d\in \mathbb{R}^n$.
Let $p$ be a $d\times n$-matrix obtained by aligning $x_i$ as the $i$-th row vector.
Then $p(i)\cdot p(j)$ is equal to the $(i,j)$-th entry $t_{ij}$ of $T$.

In the positive semidefinite symmetric matrix completion problem, we are given a partially-filled positive semidefinite symmetric matrix $T$ and asked to recover the positive semidefinite symmetric matrix by filling missing entries.
If we use a graph $G=([n],E)$ (which may contain self-loops) to denote a set of indices $e=\{i,j\}$ of known entries $t_e$ of $T$, then the problem is to find $p\in (\mathbb{R}^d)^{n}$ satisfying 
\begin{equation}\label{eq:matrix_completion}
\langle p(i), p(j)\rangle=t_e\qquad (e=\{i,j\}\in E),
\end{equation}
where $\langle\cdot, \cdot\rangle$ denotes the Euclidean inner product.
We are, in particular, interested in the uniqueness of the completions rather than finding a completion. 
In the unique completion problem, we are given a solution $p\in (\mathbb{R}^d)^{V}$ of Equation (\ref{eq:matrix_completion}) and are asked if it is the unique solution of (\ref{eq:matrix_completion}).
This uniqueness question is equivalent to asking the $g$-rigidity of  a framework $(G,p)$ for an appropriate choice of $g$. 

Specifically, consider a 2-uniform hypergraph (i.e., a graph) $G$ and 
$g(x,y)=\langle x, y\rangle$ for $x,y\in \mathbb{R}^d$.
Then the orthogonal group $O(d)$ is the stabiliser of $g$,
and the global $g$-rigidity decides if Equation (\ref{eq:matrix_completion}) has the unique solution up to the action of $O(d)$, or equivalently a completion is unique.

This rigidity-based formulation of matrix completion coincides with that introduced by Singer and Cucuringu~\cite{singer2010uniqueness} whose paper also describes a number of detailed examples of matrix completions from this viewpoint.

\paragraph{Symmetric tensor completions.} 
The idea of converting the unique low-rank matrix completion problem to the $g$-rigidity of frameworks can be extended naturally to tensors as follows.

For a vector space $V$ of dimension $n$ over $\mathbb{C}$, let $V^{\otimes k}$ be the $k$-fold tensor product of $V$.
We fix a basis of $V$, and assume that each $T\in V^{\otimes k}$ is represented by a $k$-dimensional array over $\mathbb{C}$. 
A tensor $T\in V^{\otimes k}$ is said to be {\it symmetric} if for any permutation $\sigma$ on $[k]$
we have $T_{i_1,i_2,\dots, i_k}=T_{\sigma(i_1), \sigma(i_2),\dots, \sigma(i_k)}$. The set of symmetric tensors in $V^{\otimes k}$ is denoted by 
$S^k(V)$.
%\end{definition}
It is known that any symmetric tensor can be written as 
\begin{equation}\label{eq:symmetric_tensor}
T= \sum_{i=1}^{d}  x_i^{\otimes k}:=\sum_{i=1}^d  x_i\otimes x_i\otimes \dots \otimes x_i
\end{equation}
for some  vectors $x_1, x_2, \dots, x_d\in V$.
For $T\in S^k(V)$, the smallest possible $d$ for which we can write $T$ in the form of Equation (\ref{eq:symmetric_tensor}) is called the {\it symmetric rank} of $T$. 
%The subset of $S^k(V)$ consisting of tensors with symmetric rank at most $d$ is denoted by $S^k_d(V)$.

Once we introduce a notion of rank, the corresponding low-rank completion problem can be defined automatically.
In the symmetric tensor completion problem,  
given a partially-filled tensor of order $k$ and size $n$, 
we are asked to fill the remaining entries to obtain a symmetric tensor of symmetric rank at most $d$. 
Recall that $\genfrac{\{}{\}}{0pt}{}{X}{k}$ denotes the set of multisets of $k$ elements of a finite set $X$. 
%Hence, for $V=\mathbb{C}^n$ we can use a subset $E$ of $\genfrac{\{}{\}}{0pt}{}{[n]}{k}$ to represent the known entries in the symmetric tensor completion problem. 
Due to the symmetry condition, each entry of a symmetric tensor can be indexed by an element in $\genfrac{\{}{\}}{0pt}{}{[n]}{k}$.
In this manner, we encode the underlying combinatorics of each instance of the completion problem using a $k$-uniform hypergraph $([n], E)$.

We can also convert the decomposition in Equation (\ref{eq:symmetric_tensor}) to a form of an algebraic relation among points in $\mathbb{C}^d$.
For this, let $p$ be the $d\times n$ matrix with the $i$-th row equal to $x_i$.
Then, for each $e\in \genfrac{\{}{\}}{0pt}{}{[n]}{k}$, the $e$-th entry of Equation (\ref{eq:symmetric_tensor}) is equal to $\mathbf{1}\cdot \bigodot_{v\in e} p(v)$,
where $\mathbf{1}$ denotes the all-one vector and $\bigodot$ denotes the Hadamard product of vectors, that is the component-wise product.
Hence, the symmetric tensor completion problem can be reformulated as a hypergraph realisation problem as follows:
Given a $k$-uniform hypergraph  $G=([n],E)$ and  $a_e\in \mathbb{C}$ for each $e\in E$, find $p\in (\mathbb{C}^d)^{V}$ such that 
\begin{equation}\label{eq:system1}
\mathbf{1}\cdot \bigodot_{v\in e} p(v)=a_e \qquad \text{for $e\in E$},
\end{equation}
The corresponding unique completion problem is captured by $g$-rigidity for $g: (\mathbb{C}^d)^k \rightarrow \mathbb{C}$ with 
\begin{equation}\label{eq:sym_tensor_g}
g(y_1,\dots,y_k) = \mathbf{1}\cdot \bigodot_{i \in \{1,\dots,k\}} y_i.
\end{equation}
The stabiliser $\Gamma_g$ of $g$ is the set of matrices  of the form $\Sigma D$ with a permutation matrix $\Sigma$ and a diagonal matrix $D$ over $\mathbb{C}$ whose diagonal entries are  $k$-th roots of unity.  

Let $h_{\rm prod}:\mathbb{F}^k\rightarrow \mathbb{F}$ be the product map that takes $k$ values and returns their product.
Then the function $g$ defined in Equation (\ref{eq:sym_tensor_g}) is written as the sum of $d$ copies of $h_{\rm prod}$.
This structure will be important in the analysis of $g$-rigidity in the subsequent discussion.

\paragraph{Skew-symmetric tensor completions.}

A tensor $T\in V^{\otimes k}$ is said to be {\it skew-symmetric} if for any permutation $\sigma$ on $[k]$
we have $T_{i_1,i_2,\dots, i_k}={\rm sign}(\sigma)T_{\sigma(i_1), \sigma(i_2),\dots, \sigma(i_k)}$. 
The set of skew-symmetric tensors in $V^{\otimes k}$ is denoted by 
$A^k(V)$.
It is known that $A^k(V)$ is linearly isomorphic to the $k$-th component $\bigwedge^k V$ of the exterior algebra $\bigwedge V$
and  any skew-symmetric tensor can be written as 
\begin{equation}\label{eq:skew_symmetric_tensor}
T=\sum_{i=1}^r  x_1^i\wedge x_2^i\wedge \dots \wedge x_k^i
\end{equation}
for some  $kr$ vectors $x_j^i$ ($1\leq i\leq r$, $1\leq j\leq k$), %$x_1^1, \dots, x_1^r,\ldots,x_k^1,\ldots, x_k^r\in V$,
where $\wedge$ denotes the tensor product of vectors. 
The smallest possible $r$ for which we can write $T$ in the form of Equation (\ref{eq:skew_symmetric_tensor}) is called the {\it skew-symmetric rank} of $T$. 
%The subset of $A^k(V)$ consisting of tensors with skew-symmetric rank at most $r$ is denoted by $A^k_r(V)$.
When $V=\mathbb{F}^n$, each element in $A^k(V)$ is called a {\it skew-symmetric tensor of order $k$ of size $n$} (over $\mathbb{F}$).

Recall that ${X \choose k}$ denotes the family of sets of $k$ elements of a finite set $X$. 
Due to the skew-symmetry condition, each entry of a skew-symmetric tensor can be indexed by an element in 
${X \choose k}$. 
Hence,  we encode the underlying combinatorics of each instance of the skew-symmetric completion problem using a simple $k$-uniform hypergraph $([n], E)$ (where each hyperedge is a set unlike the symmetric case).

In the same manner as the correspondence between Equations (\ref{eq:symmetric_tensor}) and (\ref{eq:system1}) in the symmetric tensor case, we shall convert Equation (\ref{eq:skew_symmetric_tensor}) to a realisation problem of the underlying hypergraph.
This can be done by picking column vectors in the matrix obtained by stacking $x_j^i$ as row vectors. Specifically, consider the $k\times n$ matrix $Q_i$ obtained by stacking $x_j^i$ for $1\leq j\leq k$ for each $1\leq i\leq r$ as row vectors,
and let $P$ be the $kr\times n$ matrix obtained by stacking $Q_i$'s.
Let $q_i\in (\mathbb{C}^k)^n$ (and respectively $p\in (\mathbb{C}^{kr})^n$) be the point configuration such that $q_i(v)$ (resp., $p(v)$) is the $v$-th column of $Q_i$ (resp., $P$).
Then, for each $e=\{v_1,\dots, v_k\}$, the $e$-th entry of $T$ in Equation (\ref{eq:skew_symmetric_tensor}) is equal to 
$\sum_{i=1}^r \det \begin{pmatrix} q_i(v_1) \dots q_i(v_k)\end{pmatrix}$.
This gives the following equivalent formulation of the skew-symmetric tensor completion problem:
Given a simple $k$-uniform hypergraph  $G=([n],E)$ and  $a_e\in \mathbb{C}$ for each $e\in E$, find an $r$-tuple $p=(q_1,\dots, q_r)$ of  $q_i\in (\mathbb{C}^k)^{V}$ such that 
\begin{equation}\label{eq:system_skew2}
\sum_{i=1}^r \det \begin{pmatrix} q_i(v_1) \dots q_i(v_k)\end{pmatrix}
=a_e \qquad \text{for  $e=\{v_1,\dots, v_k\}\in E$}.
\end{equation}

Let $h_{\rm det}:(\mathbb{C}^k)^k\rightarrow \mathbb{C}$ be the determinant as a multilinear $k$-form over $\mathbb{C}^k$.
Then the corresponding unique completion problem is captured by $g$-rigidity for $g:(\mathbb{C}^{kr})^k \rightarrow \mathbb{C}$ defined as the sum of $r$ copies of $h_{\rm det}$. 
Since the stabiliser of $h_{\rm det}$ is ${\rm SL}(k,\mathbb{C})$,
the stabiliser of $g$ is 
$$\{(A_1\oplus \dots \oplus A_r)(\Sigma\otimes I_k)\mid  A_i\in {\rm SL}(k,\mathbb{C}), \Sigma: \text{ a permutation matrix of size $r$}\}.$$

\paragraph{Chow decompositions.}
Replacing the determinant with the permanent in the above discussion, we obtain the corresponding problem for Chow decompositions. 
Suppose $f$ is a homogeneous polynomial in $n$ variables of degree $k$.
A {\em Chow decomposition} of $f$ is a representation of $f$ as the sum of $r$ polynomials written as a product of $k$ linear forms, i.e., 
\begin{equation}\label{eq:chow}
f(x_1,\dots, x_n)=\sum_{i=1}^r \prod_{j=1}^k(a_{i,j,1}x_1+\dots + a_{i,j,n}x_n)
\end{equation}
for some $a_{i,j,1},\dots, a_{i,j,n}\in \mathbb{F}$  ($1\leq i\leq r$ and $ 1\leq j\leq d$).
The smallest possible $r$ for which $f$ has an expression of the form given in Equation (\ref{eq:chow}) is called the {\em Chow rank} of $f$.
A low Chow-rank decomposition problem asks to compute a Chow decomposition of a given polynomial $f$ that achieves its Chow rank \cite{torrance2017generic}. 
In this paper, we consider the completion version, where we are given a partial list of the coefficients of a polynomial $f$ of Chow rank at most $r$ and we are asked to recover $f$.

Now we reformulate this decomposition or completion problem in terms of a hypergraph realisation problem.
As usual, the list of indices of given coefficients of $f$ is represented by a $k$-uniform hypergraph on $[n]$.
Consider the right hand side of Equation (\ref{eq:chow}).
Let $Q_i$ be the $k\times n$ matrix whose $j$-th row is 
$(a_{i,j,1}, a_{i,j,2},\dots, a_{i,j,n})\in \mathbb{F}^n$
for each $1\leq i\leq r$,
and let $P$ be the $kr\times n$ matrix obtained by stacking $Q_i$'s.
Let $q_i\in (\mathbb{F}^k)^n$ (and respectively $p\in (\mathbb{F}^{kr})^n$) be the point configuration such that $q_i(v)$ (resp., $p(v)$) is the $v$-th column of $Q_i$ (resp., $P$).
Then, the coefficient of the squarefree monomial $x_{v_1}x_{v_2}\dots x_{v_k}$ 
in the right hand side of Equation (\ref{eq:chow}) is equal to 
\begin{equation}\label{eq:perm}
\sum_{i=1}^r \perm \begin{pmatrix} q_i(v_1) \dots q_i(v_k)\end{pmatrix},
\end{equation}
where ${\rm perm}$ denotes the permanent.
Even if $x_{v_1}x_{v_2}\dots x_{v_k}$ is not squarefree, the coefficient of $x_{v_1}\dots x_{v_k}$ is still a scalar multiple of Equation (\ref{eq:perm}).
Thus, (by scaling coordinates along each axis appropriately) the low Chow-rank completion problem can be reduced to the following  hypergraph embedding problem:
Given a $k$-uniform hypergraph  $G=([n],E)$ and  $a_e\in \mathbb{F}$ for each $e\in E$, find an $r$-tuple $p=(q_1,\dots, q_r)$ of $q_i\in (\mathbb{F}^k)^{V}$ such that 
\begin{equation}\label{eq:system_skew1}
\sum_{i=1}^r \perm \begin{pmatrix} q_i(v_1) \dots q_i(v_k)\end{pmatrix}
=a_e \qquad \text{for  $e=\{v_1,\dots, v_k\}\in E$}.
\end{equation}

The corresponding unique completion problem checks if $f$ can be uniquely recovered from a partial list of coefficients  using the information that the Chow rank of $f$ is at most $r$.
Let $h_{\rm perm}:(\mathbb{F}^k)^r\rightarrow \mathbb{F}$ be the permanent as a multilinear $k$-form.
Then the unique completion problem is captured by $g$-rigidity for $g:(\mathbb{F}^{kr})^k \rightarrow \mathbb{F}$ defined by the sum of $r$ copies of $h_{\rm perm}$. 

By \cite{botta}, the stabilizer $\Gamma_{h_{\rm perm}}$ of $h_{\rm perm}$ is 
the set of matrices of the form $\Sigma D$
with a permutation matrix $\Sigma$ and a diagonal matrix $D$ with $\det D=1$.
The stabilizer of $g$ is $\{(A_1\oplus \dots \oplus A_r)(\Sigma\otimes I_k): A_i\in \Gamma_{h_{\rm perm}}, \Sigma: \text{ a permutation matrix of size $r$}\}$.

\section{Tools for Analysing $g$-rigidity}\label{sec:tools}

In this section, we give several fundamental tools for analysing local/global $g$-rigidity.
Applications of these tools to specific problems are explained in Section~\ref{sec:examples2}. 

\subsection{Basic facts on polynomial maps}
Determining the local or global $g$-rigidity of a hyper-framework is a NP-hard computational problem. A formal proof of its computational hardness can be found in \cite{abbot, saxe}, specifically addressing ordinary Euclidean rigidity. Therefore, a common approach to studying $g$-rigidity involves analyzing the linear approximation of the concept and leveraging existing tools to fill the remaining gaps. 
In this strategy the following facts are often used in graph rigidity theory.

\begin{definition}[Generic point]
Let $\kk \in \{\mathbb{R},\mathbb{C}\}$.
A  point-configuration $p\in (\mathbb{F}^d)^n$ 
is said to be {\em generic} if 
the set of coordinates of $p$ is algebraically independent over the rationals.
More generally, a point in an algebraic set ${\cal A}$ defined over $\mathbb{Q}$ is generic if its coordinates do not satisfy any $\mathbb{Q}$-polynomial besides those that are satisfied by every point on ${\cal A}$. 
\end{definition}

Let $X$ be a smooth manifold and $f: X \rightarrow \mathbb{F}^m$ be a smooth map. Then $x\in X$ is
said to be a \emph{regular point} of $f$ if the Jacobian matrix $\J f(x)$ of $f$ has maximum rank. Otherwise, $x$ is called a {\em critical point} of $f$. Also $f(x)$ is said to be a \emph{regular value} of $f$ if, for all $y \in f^{-1}(f(x))$, $y$
is a regular point of $f$.

The following proposition summarises the basic geometric tools we shall use. All four parts are known in the rigidity literature~(see, e.g., \cite{Gortler2014,GHT10}) though may not have previously been stated in the full generality of $g$-rigidity.

\begin{prop}\label{prop:matroid}
Let $\kk \in \{\mathbb{R},\mathbb{C}\}$.
\begin{itemize}
    \item[{\rm (i)}]  Let $f: \kk^n\rightarrow \kk^m$ be a polynomial map and $p\in \kk^n$.
If $f(p)$ is a regular value, then $f^{-1}(f(p))$ is a smooth manifold with codimension equal to the $\rank$ of $\J f(p)$, where $\J f(p)$ is the Jacobian derivative of $f$ evaluated at $p$. 

\item[{\rm (ii)}]  
Let $V, W$ be smooth manifolds,
$f:V\rightarrow W$ a surjective $\mathbb{Q}$-polynomial map
and $p\in V$.
If $p$ is generic in $V$, then $f(p)$ is generic in $W$.

\item[\rm{(iii)}] 
Let $f: \kk^n\rightarrow \kk^m$ be a $\mathbb{Q}$-polynomial map
and $p\in\kk^n$ a generic point.
Then, $f(p)$ is generic if and only if $\J f(p)$ is row independent.

\item[{\rm (iv)}] 
Let $f: \mathbb{F}^n\rightarrow \mathbb{F}^m$ be a $\mathbb{Q}$-polynomial map and $p\in \mathbb{F}^n$ a generic point. Then $f(p)$ is a regular value~of~$f$.
\end{itemize}
\end{prop}

\begin{proof}
\begin{itemize}
    \item[{\rm (i)}] This is simply a consequence of the implicit function theorem, see, e.g., ~\cite{milnor1965topology}.
The inverse function theorem states that, if $q\in \kk^n$ is a regular point, then there is a (Euclidean) neighbourhood $N$ of $q$ such that   $f^{-1}(f(q)) \cap N$ is  a smooth manifold of codimension equal to $\rank \J f(q)$.
But, since $f(p)$ is a regular value, this property holds for any $q\in f^{-1}(f(p))$, meaning that $f^{-1}(f(p))$ is a manifold.

\item[{\rm (ii)}] Suppose that there is a polynomial $g$ that is vanishing at $f(p)$. Then $g\circ f: V\rightarrow W$ is a polynomial map. 
Since $p$ is generic in $V$, we have that $g\circ f$ is vanishing on $V$.
Since $f$ is surjective, this implies that $g$ is vanishing on $W$.

\item[{\rm (iii)}] Suppose that  $\J f(p)$ is row independent.
Then $p$ is a regular point of $f$, and hence $f({\mathbb{F}}^n)\cap M$ is an $m$-dimensional manifold in a neighbourhood $M$ of $f(p)$ by the inverse function theorem. This in turn implies that, for a small neighbourhood $V$ of $p$ in $\mathbb{F}^n$, the restriction $f|_{V}$ of $f$ to $V$ is a surjective polynomial map from $V$ to an $m$-dimensional manifold $f(V)$. Applying (ii) to $f|_{V}$, we have that $f(p)$ is generic in $f(V)$. Since $f(V)$ is $m$-dimensional, we conclude that $f(p)$ is also generic in $\kk^m$. 

Conversely, if $\J f(p)$ is row dependent then, since the rank of $\J f$ is locally constant in a neighbourhood of $p$, if follows that $f(p)$ lies on a proper algebraic subset of $\mathbb F^m$. Using either the Tarski-Seidenberg Theorem (when $\mathbb F = \mathbb R$) or Chevalley's Theorem (when $\mathbb F = \mathbb C$) we can show that this algebraic subset is defined over $\mathbb Q$ and so $f(p)$ is not generic. (A short exposition on Chevalley's theorem is given in the proof below.)

\item[{\rm (iv)}]  
Suppose $\mathbb{F}=\mathbb{C}$.
Let $X$ be the set of points $q\in \mathbb{F}^n$ such that $f(q)$ is not regular.
Chevalley's theorem implies that any quantifier can be eliminated in an expression of a set given by boolean operations, quantifiers, and $\mathbb{Q}$-algebraic equations, see, e.g.,~\cite{basu2006algorithms}.
This, in particular, implies that $X$ is a $\mathbb{Q}$-algebraically constructible set.
On the other hand, Bertini's theorem states that there is a dense open set $U\subseteq \overline{{\rm im} f}$ in which every point is regular, see, e.g.,~\cite{shafarevich1994basic}.  Since $\overline{{\rm im} f}$ is irreducible, $U^c$ is a lower dimensional algebraic subset of $\overline{{\rm im} f}$. Since $f(X)=U^c$,  $X$ must be a lower-dimensional $\mathbb{Q}$-algebraically constructible set in $\mathbb{C}^n$, and the genericity of $p$ implies $p\notin X$.

When $\mathbb{F}=\mathbb{R}$, we can apply the same argument by using Tarski-Seidenberg quantifier elimination (see, e.g.,~\cite{basu2006algorithms}) instead of Chevalley's theorem.
\end{itemize}
\end{proof}

\subsection{Infinitesimal $g$-rigidity}\label{subsec:infinitesimal_rigidity}
Let $(G,p)$ be a $k$-uniform hyper-framework 
with $G=(V,E)$ and $g:(\mathbb{F}^{d})^k\rightarrow \mathbb{F}$ a polynomial map.
Then the Jacobian $\J f_{g,G}(p)$ of the measurement map $f_{g,G}:(\mathbb{F}^d)^V\rightarrow \mathbb{F}^E$
is a linear map from $(\mathbb{F}^d)^V$ to $\mathbb{F}^E$.
Hence the right kernel of $\J f_{g,G}(p)$ is a linear subspace $(\mathbb{F}^d)^V$, which is defined as the space of  {\em infinitesimal $g$-motions} of the framework $(G,p)$.

If the stabiliser $\Gamma_g$ of $g$ is positive dimensional, 
every framework $(G,p)$ has a nonzero infinitesimal $g$-motion.
This fact can be seen as follows.
Recall that $\Gamma_g$  is a Lie group,
so let $\mathfrak{g}$ be the Lie algebra of $\Gamma_g$. 
Observe that $\Gamma_g$ acts on $(\mathbb F^d)^V$ via $(\gamma\cdot p)(v)=  \gamma \cdot (p(v))$ for $v \in V$. This induces a map from $\mathfrak{g} \times (\mathbb F^d)^V $ to $(\mathbb F^d)^V$, denoted $(\dot \gamma, p) \mapsto \dot \gamma \cdot p$, whose image lies in the right kernel of $\J f_{g,G}(p)$.
Then $\dot{p}\in (\mathbb{F}^d)^\MV$ defined by 
$\dot{p}=\dot{\gamma} \cdot p$ for $\dot{\gamma} \in \mathfrak{g}$ is called a {\em trivial infinitesimal $g$-motion} of $(G,p)$.
For $p\in (\mathbb{F}^d)^{n}$, define the space of trivial 
infinitesimal motions of $(G,p)$ to be
$${\rm triv}_g(p)=\{\dot{\gamma}\cdot p: \dot{\gamma}\in \mathfrak{g}\}.$$
Note that ${\rm triv}_g(p)\subseteq \ker \J f_{g,G}(p)$.
\begin{definition}\label{def:inf}
Let $(G,p)$ be a $k$-uniform hyper-framework and $g:(\mathbb{F}^{d})^k\rightarrow \mathbb{F}$ a polynomial map. We say that $(G,p)$ is {\em infinitesimally $g$-rigid} if 
 ${\rm triv}_g(p)=\ker \J f_{g,G}(p)$.
%the dimension of
%$\ker \J f_{g,G}(p)$ is equal to the dimension of the space of trivial infinitesimal $g$-motions.
\end{definition}
The following proposition generalises the fundamental theorem of Asimow and Roth \cite{asimow} which provides the central motivation to look at infinitesimally rigidity in graph rigidity theory. 

\begin{prop}\label{prop:inf}
Let $g:(\mathbb{F}^d)^k\rightarrow \mathbb{F}$ be a $\mathbb{Q}$-polynomial map 
and $(G,p)$ be a $k$-uniform  hyper-framework. 
\begin{itemize}
\item If $(G,p)$ is infinitesimally $g$-rigid, then it is locally $g$-rigid.
\item If $p$ is generic, then $(G,p)$ is infinitesimally $g$-rigid if and only if it is locally $g$-rigid.
\end{itemize}
\end{prop}
\begin{proof}
%\textcolor{red}{Jim's argument (needs assumption on $n$?):} 
Let $T_p: \mathfrak{g} \rightarrow (\mathbb{F}^d)^V$ be the linear map $\dot\gamma \mapsto \dot\gamma\cdot p$. Then
$\dim({\rm triv}_g(p))$ = $\rank(T_p)$ is a lower semicontinuous function of $p$. It follows that there is a neighbourhood $U$ of $p$ such that,
for $q \in U$, $\dim({\rm triv}_g(q)) \geq \dim({\rm triv}_g(p))$.
Also for some neighbourhood $W$ of $p$ we have $\rank(\J f_{g,G}(q)) \geq \rank(\J f_{g,G}(p))$ for $q \in W$, again by the semicontinuity of rank.  

If $(G,p)$ is infinitesimally $g$-rigid 
then, by definition, $\rank 
(\J f_{g,G}(p))  = d|V| –  \dim({\rm triv}_g(p))$. 
Therefore, for $q \in U \cap W$, we have $\rank(\J f_{g,G}(q)) \geq \rank(\J f_{g,G}(p))  = d|V|– \dim({\rm triv}_g(p)) \geq d|V| – \dim({\rm triv}_g(q)) \geq \rank (\J f_{g,G}(q))$. 
So we have equality throughout and we can apply the constant rank theorem to deduce that, in a neighbourhood of $p$, $f_{g,G}$ is locally equivalent (i.e. up to a choice of coordinates) to a coordinate projection from $\mathbb{F}^{d|V|} \rightarrow \mathbb{F}^{d|V| -\dim({\rm triv}_g(p))}$. It follows that $(G,p)$ is locally $g$-rigid. This proves the first statement.

The second statement, that generic locally $g$-rigid hyper-frameworks are infinitesimally $g$-rigid, follows from 
%the technique of Asimow and Roth \cite{asimow} using the general 
Proposition~\ref{prop:matroid}(i).
\end{proof}

We now compare the dimension of $\Gamma_g$ and the dimension of the space of all trivial infinitesimal $g$-motions.
Define $d_{\Gamma_g}$ to be the dimension of $\Gamma_g$ as a Lie group. 
%and  for $p\in (\mathbb{F}^d)^{n}$ let  ${\rm triv}_g(p)=\{\dot{\gamma}\cdot p: \dot{\gamma}\in \mathfrak{g}\}$.
Since $\Gamma_g$ is a Lie subgroup of ${\rm Aff}(d,\mathbb{F})$, if $p$ is generic and $n$ is sufficiently large, then $\dim {\rm triv}_g(p)=d_{\Gamma_g}$.
We use $n_{\Gamma_g}$ to denote the minimum integer $n$ such that $\dim {\rm triv}_g(p)=d_{\Gamma_g}$ for some $p\in (\mathbb{F}^d)^{n}$. 

\begin{example}
The following presents two specific instances that exemplify these definitions.
\begin{itemize}
    \item Suppose $g:(\mathbb{R}^d)^2\rightarrow \mathbb{R}$ is the Euclidean inner product in $\mathbb{R}^d$.
    Then $\Gamma_g$ is the orthogonal group $O(d)$, so $d_{\Gamma_g}={d\choose 2}$.
    The Lie algebra $\mathfrak{g}$ is the set of skew-symmetric matrices,
    and $\dim {\rm triv}_g(p)={d\choose 2}-{d-n\choose 2}$
    if $p$ is a generic set of $n$ points with $n\leq d$.
    Hence, $n_{\Gamma_g}=d-1$.
    \item Suppose $g:(\mathbb{C}^d)^k\rightarrow \mathbb{C}$ is as in Equation (\ref{eq:sym_tensor_g}).
    Then $\Gamma_g$ is the set of matrices of the form $\Gamma D$ with a permutation matrix $\Gamma$ and a diagonal matrix $D$ whose diagonal entries are $k$-th roots of unity.
    Then $d_{\Gamma_g}=0$ and the Lie algebra is trivial.
    So $n_{\Gamma_g}=0$.
\end{itemize}
\end{example}

Since ${\rm triv}_g(p)$ is the space of trivial infinitesimal $g$-motions, if $|V(G)|\geq n_{\Gamma_g}$ and $p$ is generic, then Definition~\ref{def:inf} implies that $(G,p)$ is infinitesimally $g$-rigid if and only if $\dim\ker \J f_{g,G}(p) = d_{\Gamma_g}$.
Combining this with Proposition~\ref{prop:inf}, we have the following linear algebraic characterisation of generic local $g$-rigidity.

\begin{prop}\label{prop:inf_rigid}
Assume Setup~\ref{setup}. Let $g:(\mathbb{F}^d)^k\rightarrow \mathbb{F}$ be a polynomial map whose stabilizer is $\Gamma_g$, %$G=(V,E)$ a $k$-uniform  hypergraph, 
and $(G,p)$ a generic $k$-uniform hyper-framework with $G=(V,E)$ with $|\MV|\geq n_{\Gamma_g}$.
Then 
\[
\rank \J f_{g,G}(p)\leq d|\MV|-d_{\Gamma_g},
\]
 and the equality holds if and only if $(G,p)$ is locally $g$-rigid.
\end{prop}

%See Section~\ref{sec:examples2} for concrete examples.

\begin{example}\label{eq:global_tensor0}
    Consider the symmetric tensor completion problem of symmetric rank one and order three, that is, $d=1$, $k=3$ and $g=h_{\rm prod}:(\mathbb{C}^d)^k\rightarrow \mathbb{C}$, where $h_{\rm prod}$ denotes the product map of $k$ variables.
    Then $d_{{\Gamma}_g}=n_{\Gamma_g}=0$ as seen in Example~4.5.
    Consider a framework
     $(G,p)$, where 
    $G=(\{a,b,c\}, \{aaa, aab,abc\})$ and $p: a\mapsto x_a, b\mapsto x_b, c\mapsto x_c$.
    Then,
    \[
\J f_{g,G}(p)=
\kbordermatrix{
 & a & b & c\\
 aaa & 3x_a^2 & 0 & 0 \\
 aab & 2x_ax_b & x_a^2 & 0 \\
 abc & x_bx_c & x_cx_a & x_ax_b
    }.
    \]
    Since $\rank \J f_{g,G_i}(p)=3=d|V(G)|-d_{\Gamma_g}$,  $(G, p)$ is locally $g$-rigid.
\end{example}

\smallskip

The following geometric implication of local $g$-rigidity is also important.
\begin{prop}\label{prop:local_rigidity_finite}
Let $g:(\mathbb{F}^d)^k\rightarrow \mathbb{F}$ be a polynomial map whose stabiliser is $\Gamma_g$.
Suppose $p$ is generic and $G=(V,E)$ is a $k$-uniform hypergraph with $|V|\geq n_{\Gamma_g}+d_{\Gamma_g}$.
Then $(G,p)$ is locally $g$-rigid 
if and only if $f^{-1}_{g,G}(f_{g,G}(p))/\Gamma_g$ is finite.
\end{prop}

\begin{proof}
Suppose $n=|V|$ and $q\in f^{-1}_{g,G}(f_{g,G}(p))$. Recall that ${\rm triv}_g(q)=\{\dot{\gamma}\cdot p: \dot{\gamma}\in \mathfrak{g}\}$.
We first show that 
\begin{equation}\label{eq:local_g_rigidity_finite}
\dim {\rm triv}_g(q)=d_{\Gamma_g}.
\end{equation}
For a vector $x$ over $\mathbb{F}$, let 
$\mathbb{Q}(x)$ be the field extension of $\mathbb{Q}$ generated by the entries of $x$.
By Proposition~\ref{prop:matroid}(iii), 
the transcendence degree of $\mathbb{Q}(f_{g,G}(p))$ over $\mathbb{Q}$ is $dn-d_{\Gamma_g}$.
Since $f_{g,G}(p)=f_{g,G}(q)$,
the transcendence degree of $\mathbb{Q}(f_{g,G}(q))$ over $\mathbb{Q}$ is also $dn-d_{\Gamma_g}$.
Hence the transcendence degree of $\mathbb{Q}(q)$ over $\mathbb{Q}$ is at least $dn-d_{\Gamma_g}$.
Therefore, since $n\geq n_{\Gamma_g}+d_{\Gamma_g}$, 
there is a vertex subset $X\subset V$ such that $|X|\geq n_{\Gamma_g}$ and $q|_{X}$ is generic.\footnote{Formally the existence of $X$ can be seen by the following algorithm.
First, set $X=\emptyset$ and 
repeat the following procedure from $i=1$ through $n$.
Check if the entries of $q(X\cup\{i\})$ are algebraically independent over $\mathbb{Q}$. If yes, update $X$ to $X\cup \{i\}$; otherwise reject $i$.
Since the transcendence degree of $\mathbb{Q}(q)$ over $\mathbb{Q}$ is at least $dn-d_{\Gamma_g}$,
the fact that algebraic independence defines a matroid implies that the number of rejections in the above procedure is at most $d_{\Gamma_g}$.
Since $n\geq n_{\Gamma_g}+d_{\Gamma_g}$, we have 
$|X|\geq n_{\Gamma_g}$ at the end of the procedure.
}
Hence, $\dim {\rm triv}_g(q_{|X})=d_{\Gamma_g}$,
and Equation (\ref{eq:local_g_rigidity_finite}) follows.

Parts (i) and (iv) of Proposition~\ref{prop:matroid} imply that 
$f_{g,G}^{-1}(f_{g,G}(p))$ is a smooth manifold
of dimension equal to $\dim \ker \J f_{g,G}(p)$.
On the other hand, Equation (\ref{eq:local_g_rigidity_finite}) implies that, for any $q\in f^{-1}_{g,G}(f_{g,G}(p))$, 
$\Gamma_g \cdot q=\{ \gamma \cdot q: \gamma \in \Gamma_g \}$ is a $d_{\Gamma_g}$-dimensional smooth submanifold of $f_{g,G}^{-1}(f_{g,G}(p))$.
Since $p$ is generic, there is a neighbor $N$ of $p$ such that $(f_{g,G}^{-1}(f_{g,G}(p))\cap N) /\Gamma_g$ is a manifold of dimension equal to 
$\dim \ker \J f_{g,G}(p)-d_{\Gamma_g}$.
By Proposition~\ref{prop:inf}, the latter value is zero if and only if $(G,p)$ is locally $g$-rigid.
Hence, if $(G,p)$ is not locally $g$-rigid,
then $f_{g,G}^{-1}(f_{g,G}(p))/\Gamma_g$ is not finite.

Conversely, if $(G,p)$ is locally $g$-rigid, then 
the dimension of $f_{g,G}^{-1}(f_{g,G}(p))$ is $d_{\Gamma_g}$. 
Since $\Gamma_g \cdot q$ is a $d_{\Gamma_g}$-dimensional smooth submanifold  for any $q\in f_{g,G}^{-1}(f_{g,G}(p))$, the connected component of $f_{g,G}^{-1}(f_{g,G}(p))$ containing $q$ coincides with the connected component of $\Gamma_g \cdot q$ containing $q$.
Since $f_{g,G}^{-1}(f_{g,G}(p))$ is also an algebraic variety, it has finitely many components and hence
$f_{g,G}^{-1}(f_{g,G}(p))/\Gamma_g$ is finite.
\end{proof}

In the context of tensor completion, $g$ is  defined as in Equation (\ref{eq:sym_tensor_g}) and $d_{\Gamma_g}=n_{\Gamma_g}=0$.
As a consequence, Proposition~\ref{prop:local_rigidity_finite} holds unconditionally, independent of the number of vertices. In general we believe that the assumption $|V|\geq n_{\Gamma_g}+d_{\Gamma_g}$ in Proposition~\ref{prop:local_rigidity_finite} can be eliminated.
Conversely, the requirement of genericity for the point configuration $p$ is crucial. 
To illustrate this significance, an example showcasing non-finite $f_{g,G}^{-1}(f_{g,G}(p))/\Gamma_g$ is provided in Figure~\ref{fig:non_finite} within the realm of ordinary Euclidean rigidity.

\begin{figure}[h]
\begin{center}
\begin{tikzpicture}[scale=0.7]
\filldraw (0,0) circle (3pt);
\filldraw (1.5,0) circle (3pt);
\filldraw (3,0) circle (3pt);
\filldraw (4.5,0) circle (3pt);
\filldraw (6,0) circle (3pt);
\filldraw (3,1.5) circle (3pt);
\filldraw (4.5,3) circle (3pt);
\filldraw (1.5,3) circle (3pt);

 \draw[black,thick]
(0,0) -- (1.5,0) -- (3,0) -- (4.5,0) -- (6,0) -- (3,1.5) -- (1.5,3) -- (4.5,3) -- (6,0);

 \draw[black,thick]
 (1.5,3) -- (0,0) -- (3,1.5) -- (4.5,3);
 %%%

\filldraw (10,0) circle (3pt);
\filldraw (9.7,1.3) circle (3pt);
\filldraw (10.3,2.6) circle (3pt);
\filldraw (10,3.9) circle (3pt);
\filldraw (11.5,4.5) circle (3pt);
\filldraw (11.4,3) circle (3pt);
\filldraw (12.8,1.4) circle (3pt);
\filldraw (14,3.1) circle (3pt);

 \draw[black,thick]
 (10,0) -- (9.7,1.3) -- (10.3,2.6) -- (10,3.9) -- (11.5,4.5) -- (14,3.1) -- (12.8,1.4) -- (11.4,3) -- (14,3.1);

  \draw[black,thick]
 (11.4,3) -- (10,0) -- (12.8,1.4) -- (11.5,4.5);
  
\end{tikzpicture}
\end{center}
\caption{Consider the ordinary Euclidean rigidity in the plane.
(Left) a locally $g$-rigid framework $(G,p)$.
(Right) the framework $(G,q)$  satisfies 
$f_{g,G}(p)=f_{g,G}(q)$ but $(G,q)$ is not locally $g$-rigid
since the left path on five vertices in the framework can move continuously without changing the edge lengths.
So $f_{g,G}^{-1}(f_{g,G}(p))/\Gamma_g$ is not finite but $(G,p)$ is locally $g$-rigid.}
\label{fig:non_finite}
\end{figure}
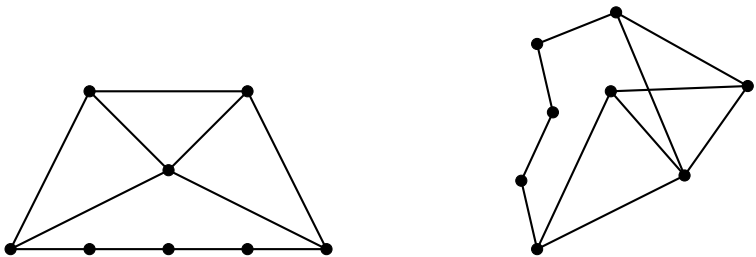

\subsection{Generic local $g$-rigidity and coordinate projections of affine varieties}
Let $G=(V,E)$ be a $k$-uniform  hypergraph,
and consider $\overline{{\rm im} f_{g,G}}$, that is,  the Zariski closure of the image of the $g$-measurement map $f_{g,G}$ in the affine space $\mathbb{F}^{E}$.
If $p$ is generic, then $f_{g,G}(p)$ is a generic point in $\overline{{\rm im} f_{g,G}}$, and hence
$\rank \J f_{g,G}(p)$ coincides with the dimension of $\overline{{\rm im} f_{g,G}}$.
Since the latter value is determined by $G$ (and independent of $p$), we can define the local $g$-rigidity as a property of hypergraphs (using Proposition \ref{prop:inf_rigid}).
This extends the remarkable idea of Asimov and Roth~\cite{asimow} that used this property as the definition of rigidity in the Euclidean distance case.

Specifically, a hypergraph $G$ with $n=|V|$ vertices is said to be {\em locally $g$-rigid} if 
$(G,p)$ is infinitesimally $g$-rigid for some/all generic $p$,
and equivalently
\[
\dim \overline{{\rm im} f_{g,G}}=dn-d_{\Gamma_g}
\quad\text{if\  }n\geq n_{\Gamma_g}.\]
%if $n\geq n_{\Gamma_g}$.

It is often convenient to analyse each hypergraph  as a subgraph of the complete $k$-uniform hypergraph $K_n^k$ with the same vertex set. For this,  suppose $G=([n],E)$ is a subgraph of $K_n^k$.
Let $\pi_G: \mathbb{F}^{\genfrac{\{}{\}}{0pt}{}{[n]}{k}}\rightarrow \mathbb{F}^{E}$ be the {\em coordinate projection} to the subspace indexed by the hyperedges in $G$.
Then, 
\[
f_{g,G}=\pi_G\circ f_{g, K_n^k}.
\]
Hence, checking the local $g$-rigidity of $G$ is equivalent to checking the equality:
\[
\dim \overline{\pi_G({\rm im} f_{g,K_n^k})}=dn-d_{\Gamma_g}.
\]
One notable advantage of this viewpoint is that 
$\overline{{\rm im} f_{g,K_n^k}}$ is often well-investigated in the literature, which allows us to leverage established tools.
For instance, in the case where $g$ represents the squared distance function, 
$\overline{{\rm im} f_{g,K_n^k}}$ is known as the Cayley-Menger variety  
\cite{Borcea02}.
Similarly, when $g$ is the product map, 
$\overline{{\rm im} f_{g,K_n^k}}$ coincides with the affine cone of the Veronese variety.
We will discuss these examples in more detail in Section~\ref{sec:examples2}.

\subsection{Secant, non-defectivity, and identifiability of varieties}
Let ${\cal V}$ be a homogeneous affine variety in $\mathbb{C}^m$.
Its {\em $r$-secant} is defined as 
\begin{equation}\label{eqn:secant}
{\rm Sect}_r({\cal V})=\overline{\bigcup_{x_1, \dots, x_r\in {\cal V}} \langle x_1, \dots, x_r\rangle}.
\end{equation}
In applications, it is often the case that a homogeneous polynomial map $g:(\mathbb{C}^d)^k\rightarrow \mathbb{C}$ is written as the sum of polynomial maps $h:(\mathbb{C}^s)^k\rightarrow \mathbb{C}$ for $s<d$. In such a case, the concept of $g$-rigidity is closely related to the non-defectivity and identifiability of varieties.

To formally establish this connection, let us consider a polynomial map $g:(\mathbb{C}^d)^k\rightarrow \mathbb{C}$.
The argument of $g$ is a tuple $(p_1, p_2, \dots, p_k)$ consisting of $k$ points $p_i$ in $\mathbb{C}^d$. We denote the coordinates of each $p_i$ by $p_i=(p_{i1}, p_{i2}, \dots, p_{id})^{\top}$.
Suppose $g$ is separable into coordinate-wise maps, that is, 
\[
g(p_1,\dots, p_k)=\sum_{i=1}^d h(p_{1i}, \dots, p_{ki})
\]
for some $h:(\mathbb{C}^1)^k \rightarrow \mathbb{C}$. In this case, it follows that $\overline{{\rm im} f_{g,G}}$ is the $d$-secant of $\overline{{\rm im} f_{h,G}}$.
We need slightly more involved notation to deal with the case when $g\in (\mathbb{C}^d)^k$ is decomposed into the sum of copies of $h\in (\mathbb{C}^s)^k$. (Such situations appear in skew-symmetric completions and Chow decompositions given in Section~\ref{sec:examples}.)
Specifically, consider the identification between  $(\mathbb{C}^d)^k$ and $((\mathbb{C}^s)^k)^t$ with $st=d$ such that 
$p\in (\mathbb{C}^d)^k$ is identified with a $t$-tuple $(q_1,\dots, q_t)$ of $q_i\in (\mathbb{C}^s)^k$.
If $g$ is written as 
\begin{equation}\label{eq:sum}
g(p)=\sum_{i=1}^t h(q_i)
\end{equation}
for some $h:(\mathbb{C}^s)^k \rightarrow \mathbb{C}$, then $\overline{{\rm im} f_{g,G}}$ is the $t$-secant of $\overline{{\rm im} f_{h,G}}$.
In such a case, we simply say that {\em $g$ is the sum of $t$ copies of $h$}.

Computing the dimension of secant varieties has been extensively studied in algebraic geometry (see, e.g.,~\cite{bernardi,Zak}). 
In particular, the algebraic notions of non-defectivity and identifiability of these varieties are  closely related to the notion of $g$-rigidity. 
Although these concepts are typically defined for projective varieties, we maintain our focus on affine varieties to preserve the connection with rigidity theory.

\paragraph{Non-defectivity and local $g$-rigidity.}
Given an affine variety ${\cal V}$ in $\mathbb{C}^m$, the $r$-th secant of ${\cal V}$ has dimension at most $\min\{r\dim {\cal V}, m\}$. The latter number is called the {\em expected dimension}, and ${\cal V}$ is said to be {\em $r$-defective} if $\dim {\rm Sect}_r({\cal V})$ is smaller than the expected dimension. 

Assume Setup~\ref{setup}, and suppose further that $g$ is the sum of $t$ copies of $h$.
If $\overline{{\rm im} f_{h,G}}$ is not $t$-defective, then $\dim(\overline{ {\rm im} f_{g,G}})=\min\{t\dim(\overline{{\rm im} f_{h,G}}), |E|\}$, 
and hence checking local $g$-rigidity is reduced to checking local $h$-rigidity. 
The defectivity of classical algebraic varieties such as Veronese, Segre, and Grassmannian varieties has been extensively studied (see, e.g.,~\cite{bernardi}). In the classical setting, the underlying hypergraphs of such varieties are always complete, and hence we can apply the existing results to prove the local $g$-rigidity of complete hypergraphs. 
In particular, when the underlying hypergraph is complete, then the notion of local  $g$-rigidity often coincides with the classical notion of non-defectivity of the corresponding secant varieties. See Section~\ref{sec:examples2} for concrete examples. 

However, for the analysis of general hypergraphs, we need to understand the influence of coordinate projections on defectivity. 
Proposition~\ref{prop:sparsity}, below, shows that the dimension for a general hypergraph can be much smaller than the expected dimension for combinatorial reasons. Therefore, we focus 
on the combinatorial properties of the underlying hypergraphs. 

It is also important to point out that our primal applications include the case when $g$ and $h$ are quadratic. In such cases, $\overline{{\rm im} f_{h,G}}$ will always be defective whereas $G$ can be still locally $g$-rigid. 
So the non-defectivity and local $g$-rigidity are different notions in general.
An example is described below.
%See the discussion on ordinary Euclidean rigidity in the examples below. 

\paragraph{Tangent spaces.}
Suppose $g$ is the sum of $t$ copies of $h$.
Then, by arranging the column ordering appropriately, we have
\begin{equation}\label{eq:Jacobian_block}
\J f_{g,G}(p)=\begin{pmatrix} \J f_{h,G}(q_1) \dots \J f_{h,G}(q_t)\end{pmatrix}.
\end{equation}
Hence the tangent space of $\overline{{\rm im} f_{g,G}}$ at a generic point $f_{g,G}(p)$ is the linear span of 
those of $\overline{{\rm im} f_{h,G}}$ at 
$f_{h,G}(q_1), \dots, f_{h,G}(q_t)$.
This fact corresponds to the Teraccini lemma in algebraic geometry (see, e.g.,~\cite{bernardi}).

\paragraph{Identifiability and global $g$-rigidity.}
Suppose $g$ is homogeneous.
For a variety ${\V}$ in $\mathbb{C}^m$, a generic point $y\in {\rm Sect}_r({\V})$ can be written as 
$y=\sum_{i=1}^r y_i$ for some $y_i\in {\V}$. 
The variety ${\mathcal{V}}$ is said to be {\it $r$-identifiable} if $y_1,\ldots, y_r$ are uniquely determined up to permutations of indices and scaling of each $y_i$.
As we will see in Section~\ref{sec:global}, identifiability is a useful property for checking global $g$-rigidity. For example, in the case of symmetric tensors, a complete hypergraph is globally $g$-rigid if and only if the corresponding secant variety is identifiable. However, in general, global $g$-rigidity and identifiability are not comparable. We now provide two examples to demonstrate this. %However, in general global $g$-rigidity can be stronger property than identifiability. 
%More precisely, even if the decomposition $y=\sum_{i=1}^r y_i$ is uniquely determined,  one needs to check the size of the fiber of the $h$-measurement map at each $y_i$ to guarantee global $g$-rigidity.
%More details will be explained in Section~\ref{sec:global}.
%See an example below.

\paragraph{Example.} To see the difference among local/global rigidity, defectiveness, and identifiability, 
we consider the symmetric rank-$r$ matrix completion of order $k$. In this case  
$g:(\mathbb{C}^r)^k\rightarrow \mathbb{C}$ is the sum of $r$ copies of the product map
$h_{\rm prod}: \mathbb{C}^k\rightarrow \mathbb{C}$.
Then 
$\overline{{\rm im} f_{g, G}}$ is the $r$-secant of $\overline{{\rm im} f_{h_{\rm prod}, G}}$.

\begin{itemize}
\item Suppose $k=2$.
This case  corresponds to the rank-$r$ matrix completion.
It is known (see, e.g., \cite{singer2010uniqueness}) that 
$\J f_{h_{\rm prod}, K_n}=n$ and $\J f_{g, K_n}=rn-{r\choose 2}$.
So $\overline{{\rm im} f_{h_{\rm prod}, K_n}}$
is $r$-defective and not $r$-identifiable.
However, since $d_{\Gamma_g}={r\choose 2}$, 
$K_n$ is locally $g$-rigid by Proposition~\ref{prop:inf_rigid}.
In fact, $K_n$ is known to be globally $g$-rigid, see, e.g.~\cite{singer2010uniqueness}.

\item Suppose $r=1$.
Then 
$\overline{{\rm im} f_{h_{\rm prod}, G}}$ is $1$-identifiable by definition.
However, the local rigidity and the global $g$-rigidity of a generic framework $(G,p)$  depends on the combinatorial structure of $G$. Indeed the global $g$-rigidity captures the uniqueness of symmetric rank-one tensor completions. So, if $G$ is very sparse, $(G,p)$ cannot be globally $g$-rigid.
\end{itemize}

\subsection{Generic $g$-rigidity matroids}
Our goal in this subsection is to establish a connection between $g$-rigidity and combinatorial properties of hypergraphs. For this purpose, we formulate our problem using the framework of algebraic matroids.

\begin{definition}%[Algebraic matroid]
For a $\mathbb{Q}$-polynomial map $f:\mathbb{F}^n \rightarrow \mathbb{F}^m$, consider the following collections of subsets~of~$[m]$:
 \begin{enumerate}
\item For a generic $p\in \mathbb{F}^{n}$, 
\[
\mathcal{I}_1= \{
S \subseteq [m]: \text{the rows of the Jacobian matrix } \J f(p) \text{ indexed by $S$ are independent}
\}.
\]
\item Let $V_f={\rm im} f$ and $I(V_f)$ be the associated ideal of the variety $\overline{{\rm im} f}$, and 
\[
\mathcal{I}_2 = \{
S \subseteq [m]: I(V_f)\cap \mathbb{F}[S] = \{0\}
\}.
\]
\item 
Let $\pi_S$ be the coordinate projection to the subspace indexed by $S\subseteq [m]$, and let
\[
\mathcal{I}_3 = \{
S \subseteq [m] : \overline{\pi_S(V_f)} = \mathbb{F}^S
\}.
\]
\end{enumerate}
The pair $([m], \mathcal{I}_1)$ is called the {\em algebraic matroid} of $f$. It is known that $\mathcal{I}_1=\mathcal{I}_2=\mathcal{I}_3$ (see, for example, \cite{rosen2020algebraic} which also provides a detailed exposition on algebraic matroids). 
It is also known that $\mathcal{I}_i$ is invariant over the choice of  $\mathbb{F}\in \{\mathbb{R}, \mathbb{C}\}$, which underlines the robustness of algebraic matroids in capturing relevant properties of the corresponding function $f$.

\end{definition}
\begin{definition}%[$g$-rigidity matroid]
For a $\mathbb{Q}$-polynomial map $g:(\mathbb{F}^d)^k\rightarrow \mathbb{F}$ and a positive integer $n$, the algebraic matroid of $f_{g, K_n^k}$ is called the {\em $g$-rigidity matroid} ${\cal M}_{g,n}$ (or simply ${\cal M}_g$ if $n$ is not important).
\end{definition}
The $g$-rigidity matroid is the matroid on the hyperedges of $K_n^k$.

\subsection{Combinatorial characterisation problem}
\label{subsec:combinatorial}

Given that generic $g$-rigidity is a property of hypergraphs, it is natural to ask whether this property can be characterised purely in terms of graph theory concepts. This pursuit of a combinatorial characterisation forms a central theme within the field of graph rigidity. 
The most famous example is the following Geiringer-Laman theorem \cite{Laman,HPG}.

\begin{theorem}\label{thm:Laman}
Suppose $g$ is the squared distance function in two-dimensional Euclidean space. Then $G$ is $g$-rigid if and only if $G$ has a subgraph $H$ satisfying the following conditions:
\begin{itemize}
\item $|E(H)|=2|V(G)|-3$, and
\item for any subgraph $H'=(V',E')$ of $H$ with $|E'|\geq 1$, $|E'|\leq 2|V'|-3$.
\end{itemize}
\end{theorem}

Inspired by the goal of extending the aforementioned result to encompass general polynomial maps $g$ and hypergraphs, we review a construction of combinatorial matroids on hypergraphs. 

\smallskip
Let $G=(\MV,E)$ be a $k$-uniform hypergraph.
For nonnegative integers $a$ and $b$, we say that $G$ is {\em $(a,b)$-sparse} if 
every sub-hypergraph $G'$ with $E(G')\neq \emptyset$ satisfies $|E(G')|\leq a|\MV(G')|-b$,
and an $(a,b)$-sparse $G$ is said to be {\em $(a,b)$-tight} if $|E|=a|\MV|-b$.
It is known that, if $0\leq b\leq ak-1$, the collection of edge sets of all $(a,b)$-sparse sub-hypergraphs of $K_n^k$ forms the independent set of a matroid on $E$ 
(see, e.g.,~\cite{whiteley1996some}),
which is called the {\em $(a,b)$-sparsity matroid} $\cS^{a,b}_n$.

Consider a polynomial map $g$ and a hypergraph $G=(V,E)$ with $G\subseteq K_n^k$.
By Proposition~\ref{prop:inf_rigid}, if $E$ is independent in the $g$-rigidity matroid ${\cal M}_g$, 
\[
|E'|\leq d|\MV'|-d_{\Gamma_g}
\]
for any subgraph $G'=(\MV',E')$ of $G$ with $|\MV'|\geq n_{\Gamma_g}$.
If $dk-d_{\Gamma_g}\geq 1$, the $(a,b)$-sparsity matroid $\cS^{a,b}_n$ is nontrivial, and its rank gives a {\it combinatorially expected rank} of $E$ in the $g$-rigidity matroid. It is known that  the rank of $E$ in $\cS^{a,b}_n$ is given by 
\begin{equation}\label{eq:sparsity_rank}
\min\left\{|F_0|+\sum_{i=1}^k \left(a\left|\bigcup_{e\in F_i} e\right|-b\right)\ \Big|\  \begin{array}{l}F_0\subseteq E\text{ and } \{F_1,\dots, F_k\}   \text{ is a partition of $E\setminus F_0$}\end{array}\right\},
\end{equation}
see, e.g.,~\cite{schrijver2003combinatorial}. 
Hence, we have the following ``combinatorially expected dimension" of $\overline{\pi_G({\rm im} f_{g, K_n^k})}$,
which can be strictly smaller than a trivial expected dimension in the context of secant varieties.
\begin{prop}\label{prop:sparsity}
Let $g:(\mathbb{F}^d)^k\rightarrow \mathbb{F}$ be a polynomial map, and let $r_{g,n}$ be the rank function of ${\cal M}_{g,n}$. 
Suppose that $k\geq n_{\Gamma_g}$ and $dk-d_{\Gamma_g}\geq 1$.
Then, for any $k$-uniform hypergraph $G=(V,E)$, 
\begin{equation}\label{eq:sparsity_count}
\dim(\overline{\pi_G({\rm im} f_{g,K_n^k}})) \nonumber =r_{g,n}(E)\leq 
\min\left\{|F_0|+\sum_{i=1}^k \left(d\left|\bigcup_{e\in F_i} e\right|-d_{\Gamma_g}\right)\ \Big|\  \begin{array}{l}F_0\subseteq E \text{ and } \{F_1,\dots, F_k\} \\  \text{ is a partition of $E\setminus F_0$}\end{array}\right\}. 
\end{equation}
\end{prop}

The structure of $(a,b)$-sparsity matroids are well-understood,
and there is an efficient algorithm for computing the right hand side of Equation (\ref{eq:sparsity_count}),  see \cite{lee2007graded, streinu2009sparse}.

A natural question arises concerning whether the inequality presented in Equation (\ref{eq:sparsity_count}) is actually an equality. While the Geiringer-Laman theorem, Theorem~\ref{thm:Laman}, establishes this equality for any graph $G$ when $g$ represents the squared distance function in two-dimensional Euclidean space, it is important to note that, in most naturally occurring examples in practical applications, there exists a hypergraph $G$ for which the inequality in Equation (\ref{eq:sparsity_count}) is strict.
Nevertheless, we conjecture that 
the equality always holds if $g$ is sufficiently generic.
%, that underlines the significance of genericity. 

\begin{conjecture}\label{con:genericform}
If $g$ is generic, 
then ${\cal M}_{g,n}={\cal S}_n^{d,0}$.
\end{conjecture}

\subsection{Matroid Union}
We now show that the decomposition property of $g$ is inherited by the corresponding matroid ${\cal M}_g$.

\begin{lemma}\label{lem:decomposition}
Suppose a polynomial map $g:(\mathbb{F}^d)^k\rightarrow \mathbb{F}$ is the sum of $t$ copies of $h:(\mathbb{F}^s)^k\rightarrow \mathbb{F}$ with $d=st$.
If a set $E$ of hyperedges is independent in ${\cal M}_{g,n}$, then
$E$ can be partitioned into $E_1, E_2, \dots, E_t$ such that 
each $E_i$ is independent in ${\cal M}_{h,n}$.
\end{lemma}
\begin{proof}
For the sake of simplicity, we will provide the proof specifically for the case when $d_{\Gamma_g}=0$. However, it is important to note that the proof can be readily adapted to the general case by considering a maximal non-singular submatrix of the Jacobian.
We may further assume that $E$ represents a basis of ${\cal M}_{g,n}$, as the property of decomposition holds for all bases, implying the same property for all independent sets. 

For a basis $E$ of the matroid ${\cal M}_{g,n}$, we denote $G[E]$ for the induced sub-hypergraph on the hyperedges $E$. Let $p$ be a generic point-configuration of $G[E]$.
Since $E$ is a basis and $d_{\Gamma_g}=0$, we have that $|E|=kn$, and so $\J f_{g,G[E]}(p)$ is a square matrix.
Since $g$ is the sum of $t$ copies of $h$, the structure of $\J f_{g,G[E]}(p)$ has the form of Equation (\ref{eq:Jacobian_block}).
Hence, the Laplace expansion tells us that 
\begin{equation}\label{eq:expansion2}
\det \J f_{g,G[E]}(p)=\sum_{\{E_1,\dots, E_t\}} \prod_{i=1}^t \det\J f_{h,G[E_i]}(q_i), 
\end{equation}
where the sum is taken over all partitions $\{E_1,\dots, E_t\}$ of $E$ with $|E_i|=sn\ (1\leq i\leq t)$.
Since $E$ is a basis of ${\cal M}_{g,n}$, $\det \J f_{g,G[E]}(p)\neq 0$.
Hence, there must be at least one non-zero term in the right hand side of Equation (\ref{eq:expansion2}), and the corresponding partition $\{E_1,\dots, E_t\}$ is a desired one.
\end{proof}

Given matroids ${\cal M}_1,\dots, {\cal M}_t$, their {\em union} is defined such that a set is independent if it can be partitioned into independent sets of each ${\cal M}_i$. The union of $t$ copies of ${\cal M}$ is denoted by $t{\cal M}$.
Lemma~\ref{lem:decomposition} implies that 
${\cal M}_{g,n} \preceq t{\cal M}_{h,n}$ in the weak order poset of matroids, i.e., any independent set of ${\cal M}_{g,n}$ is independent in $t{\cal M}_{h,n}$.
A natural question is whether they are indeed different.
This question is closely related to $t$-defectivity since 
${\cal M}_{g,n}=t{\cal M}_{h,n}$ holds if and only if 
$\overline{\pi_H({\rm im} f_{h,K_n^k})}$ is not $t$-defective for any $H\subseteq K_n^k$ having no coloop
 in $t{\cal M}_{h,n}$ (i.e., every element is covered by some circuit) by Edmonds' matroid union theorem, see, e.g.,~\cite{schrijver2003combinatorial}.

\section{Combinatorial Analysis in Our Example Models} \label{sec:examples2}
In this section, we will review the concepts and tools introduced in Section \ref{sec:tools} by applying them to concrete examples presented in Section \ref{sec:examples}. 

\paragraph{Euclidean rigidity.} 
In the ordinary Euclidean rigidity, $G$ is 2-uniform (i.e., a graph) and 
$g(x,y)=\sum_{i=1}^d (x_i-y_i)^2$ for $x=(x_1,\dots, x_d)^{\top}, y=(y_1,\dots, y_d)^{\top}\in \mathbb{R}^d$.
The stabiliser $\Gamma_g$ is the Euclidean group $E(d)$.
We have $d_{E(d)}={d+1\choose 2}$ and $n_{E(d)}=d$. Then
$(d,{d+1\choose 2})$-sparsity is a necessary condition for the independence in the $g$-rigidity matroid. 
When $d=2$, Theorem~\ref{thm:Laman} 
implies that the $g$-rigidity matroid is equal to the $(2,3)$-sparsity matroid. This is not true for $d\geq 3$ (see Figure \ref{fig:banana} for an example and \cite{GGJN} for a recent discussion) and determining which $(d,{d+1\choose 2})$-sparse graphs are independent is a fundamental open problem in rigidity theory.
See \cite{KJT} for a new candidate characterisation when $d=3$.

\begin{figure}[h]
\begin{center}
\begin{tikzpicture}[font=\small]
\foreach \i in {0,...,4}{
	\node[roundnode] at (36+72*\i:1cm) (\i) []{};
	\foreach \j in {0,...,\i}{
		\ifthenelse{\i=\j \OR \i=4 \AND \j=0}
			{}
			{\draw[] (\i)--(\j)};
	}
}
\begin{scope}[shift={(36:1cm)}]
\begin{scope}[shift={(-36:1cm)}]
\foreach \i in {5,...,7}{
	\node[roundnode] at ({144-72*(\i-4)}:1cm) (\i) []{};
	\draw[] (\i)--(0);
	\draw[] (\i)--(4);
	\foreach \j in {5,...,\i}{
		\ifthenelse{\i=\j \OR \i=4 \AND \j=0}
			{}
			{\draw[] (\i)--(\j)};
	}
}
\end{scope}
\end{scope}

\end{tikzpicture}
\end{center}
\caption{A $(3,6)$-sparse graph that is dependent in the $g$-rigidity matroid (when $g$ is Euclidean distance and $d=3$).}
\label{fig:banana}
\end{figure}
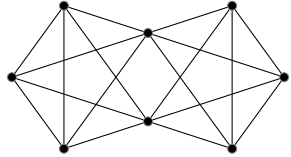

Let $h(a,b)=(a-b)^2$ for $a,b \in \mathbb{C}$.
Then $g(x,y)=\sum_{i=1}^d h(x_i,y_i)$. 
So the affine variety defined by $g$ is the $d$-secant of that of $h$.
The $h$-rigidity is easy to understand since it corresponds to 1-dimensional rigidity and it is characterised by graph connectivity. 
However, the affine variety defined by $h$ is defective, and currently there is no sufficient condition to guarantee the $d$-dimensional rigidity in terms of $1$-dimensional rigidity. 

Hendrickson~\cite{Hendrickson1992conditions} proved two natural necessary conditions for a graph to be globally $g$-rigid in dimension $d$. More precisely he proved that, when $g$ is the usual Euclidean distance, any globally $g$-rigid graph $G$ is $(d+1)$-connected and redundantly $g$-rigid (which means that the graph obtained from $G$ by deleting any single edge is $g$-rigid). Connelly~\cite{Con05} proved an algebraic sufficient condition based on stress matrices (weighted Laplacians). Using these two results, and specialised to the case when $d=2$, Jackson and Jord\'an~\cite{jackson2005connected} gave a combinatorial characterisation of generic global $g$-rigidity. While Hendrickson's necessary conditions are known to be insufficient when $d\geq 3$ \cite{Con93}, Gortler, Healy and Thurston~\cite{GHT10} proved that generic global $g$-rigidity does depend only on the underlying graph.

\paragraph{Pseudo-Euclidean rigidity.} Here the situation is similar, $G$ is 2-uniform, the stabiliser is still the Euclidean group, $d_{\Gamma_g}={d+1 \choose 2}$ and $n_{\Gamma_g}=d$. Thus $(d,{d+1\choose 2})$-sparsity is again a necessary condition for independence. In fact there is a direct linear algebraic transformation between independence in the Euclidean case and the pseudo-Euclidean case~\cite{SW07}. This translates, for example, the Geiringer-Laman theorem to pseudo-Euclidean planes.

Suppose $g$ is the usual Euclidean distance and $g'$ is a pseudo-Euclidean distance. Gortler and Thurston \cite{Gortler2014} proved that a generic framework $(G,p)$ is globally $g$-rigid if and only if it is globally $g'$-rigid. Along the way they extended Connelly's sufficient condition to pseudo-Euclidean spaces. However in the case of hyperbolic space, it appears to be open to decide if Hendrickson's redundant rigidity condition applies. Earlier Connelly and Whiteley \cite{Con10} proved that, generically, global rigidity in the Euclidean and spherical contexts coincide.

\paragraph{$\ell_p$-rigidity.} Suppose $p$ is a positive even integer not equal to 2.
Again $G$ is 2-uniform. The stabiliser is the $d$-dimensional group of translations. We have $d_{\Gamma_g}=d$, $n_{\Gamma_g}=1$ and $(d,d)$-sparsity is a necessary condition for independence. When $d=2$, an analogue of the Geiringer-Laman theorem was obtained in \cite{KP14} showing that the $g$-rigidity matroid is equivalent to the $(2,2)$-sparsity matroid on a simple graph. 
It is conjectured that the $d$-dimensional case is characterised by $(d,d)$-sparsity (for simple graphs). This would characterise $d$-dimensional rigidity in terms of 1-dimensional rigidity. However, such a combinatorial description is currently only known in certain very special cases \cite{Dewar2022}.

It is an open problem to analyse global $g$-rigidity in the $\ell_p$ case. Since $p$ is even, $\ell_p$-norms are analytic and hence the results of \cite{Dewar23} apply; they showed that the natural analogue of Hendrickson's necessary conditions apply in analytic normed spaces and gave a combinatorial characterisation analogous to Jackson-Jord\'an when the space is an analytic normed plane.

\paragraph{Volume constraint rigidity.} 
In the volume constraint rigidity, the map $g:(\mathbb{R}^d)^{d+1}\rightarrow \mathbb{R}$ is given  by 
\[
g(x_1, x_2, \dots, x_{d+1})=\det \begin{pmatrix} x_1 & x_2 & \dots & x_{d+1} \\
1 & 1 & \dots & 1\end{pmatrix}.
\]
for $x_1,\dots, x_{d+1}\in \mathbb{R}^{d+1}$.
The stabiliser $\Gamma_g$ is the special affine group $SA(d,\mathbb{R})$.
We have $d_{\Gamma_g}=d^2+d-1$ and $n_{\Gamma_g}= d+1$.
Then $(d,d^2+d-1)$-sparsity is a necessary condition for the independence in the $g$-rigidity matroid, but it is not sufficient in general, see, e.g.~\cite{bulavka2022volume}.
Very recently, Southgate \cite{Southgate} has proved that global $g$-rigidity, in this volume case, is not a generic property of the underlying hypergraph.

\paragraph{Positive semidefinite matrix completion problem.}
In this problem, $G$ is 2-uniform (i.e., a graph) and 
$g$ is the Euclidean inner product in $\mathbb{R}^d$.
The stabiliser $\Gamma_g$ is $O(d)$,
and $d_{E(d)}={d\choose 2}, n_{E(d)}=d-1$. Then
$(d,{d\choose 2})$-sparsity is a necessary condition for the independence in the $g$-rigidity matroid. 
This sparsity condition is sufficient for $d=1$,
but is not in general for $d\geq 2$~\cite{singer2010uniqueness}.
It is a challenging open problem to give a good characterisation of the independence in $d=2$, see~\cite{bernstein,JT21}.
The global $g$-rigidity is also characterised when $d=1$ in \cite{singer2010uniqueness}, but it is not well understood when $d\geq 2$. 
It is known that global $g$-rigidity is not a generic property of graphs~\cite{JJTunique}, see Section~\ref{subsec:global_graphs} for more details.

\paragraph{Symmetric tensor completion problem.}
Recall that $h_{\rm prod}:\mathbb{C}^k\rightarrow \mathbb{C}$ denotes the product map, i.e., $h(y_1, y_2,\dots, y_k)=y_1y_2\dots y_k$.
In the symmetric tensor completion with symmetric rank $d$,
the measurement $g: (\mathbb{C}^d)^k \rightarrow \mathbb{C}$ is given as the sum of $d$ copies of $h_{\rm prod}$.
Hence, we denote this $g$ by $dh_{\rm prod}$.

The $g$-rigidity property is very different between the cases of $k=2$ and $k\geq 3$.
When $k=2$, tensors are matrices and we have the same results as the positive semidefinite matrix completion problem discussed above. So in the subsequent discussion we focus on the case when $k\geq 3$.

The stabiliser $\Gamma$ of $dh_{\rm prod}$ is the set of diagonal matrices over $\mathbb{C}$ whose diagonal entries are  $k$-th roots of unity (up to permutation of indices). 
We have $d_{\Gamma}=0$ and $n_{\Gamma}=0$. 
The $(d,0)$-sparsity is a necessary condition for the independence in the $dh_{\rm prod}$-rigidity matroid, but it is not sufficient since no $k$-uniform $k$-partite hypergraph can be $dh_{\rm prod}$-rigid. We will discuss this point in Section~\ref{sec:variants}.

By definition, $G=(V,E)$ is locally $dh_{\rm prod}$-rigid if and only if $\rank \J f_{dh_{\rm prod},G}(p)=d|V|$ for a  generic $p\in (\mathbb{C}^r)^k$. 
Let $I_G$ be the edge-vertex incidence matrix of $G$, that is, the matrix of size $|E|\times |V|$ whose $(e,i)$-th entry is one if vertex $i$ is contained in hyperedge $e$, and zero otherwise. 
When $d=1$, it is easy %not difficult 
to check that $\J f_{dh_{\rm prod},G}(p)$ is equivalent to $I_G$ (in the sense of row and column operations). 
Hence, if $d=1$, $G$ is locally $dh_{\rm prod}$-rigid if and only if the edge-vertex incidence matrix is full rank.
The problem becomes much harder when $d\geq 2$, and checking the locally $dh_{\rm prod}$-rigidity of $K_n^k$ is already a nontrivial question.

Since $dh_{\rm prod}$ is the sum of $d$ copies of $h_{\rm prod}$,
 $\overline{{\rm im} f_{dh_{\rm prod},G}}$ is the $d$-secant of $\overline{{\rm im} f_{h_{\rm prod},G}}$.
When $G=K_n^k$, $f_{h_{\rm prod},K_n^k}$ is the Veronese embedding, and $\overline{{\rm im} f_{h^{\rm prod},K_n^k}}$ is the affine cone of the Veronese variety of degree $k$.
It is relatively recent that Alexander and Hirschowitz~\cite{alexander1995polynomial} gave a concrete list of defective cases of Veronese varieties. 
Adapting this result to our language,  we have the following.

\begin{theorem}\label{thm:ah}
Let $k, n, d$ be positive integers with $k\geq 3$,
and  $p:[n]\rightarrow \mathbb{C}^d$ be a generic point configuration. 
Then
\[
\rk \J f_{dh_{\rm prod},K_n^k}(p)=\min\{|E(K_n^k)|, dn\},
\]
except for $(k,n,d)$ in the set $\{(3,5,7), (4,3,5), (4,4,9), (4,5,14)\}$.
In other words, $K_n^k$ is locally $dh_{\rm prod}$-rigid if and only if $|{n+k-1 \choose k}|\geq dn$ 
and $(k,n,d)\notin \{(3,5,7), (4,3,5), (4,4,9), (4,5,14)\}$.
\end{theorem}

See, e.g.,~\cite[Theorem~2.12]{bernardi} for more details on the Alexander-Hirschowitz theorem.

The identifiability problem for Veronese varieties has been also solved in \cite{chiantini2017generic,galuppi2019identifiability}, which implies a characterisation of the globally $dh_{\rm prod}$-rigid complete hypergraphs in our terminology.
Notably, complete hypergraphs which are both minimally locally $dh_{\rm prod}$-rigid and globally $dh_{\rm prod}$-rigid are classified in \cite{galuppi2019identifiability}. 
Adapting these results to our language,  we have the following.
\begin{theorem}\label{thm:sym_tensor_global}
Let $k, n, d$ be positive integers with $k\geq 3$. 
\begin{itemize}
\item Suppose ${n+1+k\choose k}>dn$.
Then $K_n^k$ is globally $dh_{\rm prod}$-rigid if and only if $(k,n,d)\notin \{(6,3,9), (4,4,8), (3,6,9)\}$.

\item Suppose ${n+1+k\choose k}=dn$.
Then $K_n^k$ is globally $dh_{\rm prod}$-rigid if and only if $(k,n,d)\in \{(3,4,5), (5,3,7)\}\cup \{(2s-1,2,s): s\geq 2, s\in \mathbb{Z}\}$.
\end{itemize}
\end{theorem}

\paragraph{Skew-symmetric tensor completions.}
Let $h_{\rm det}:(\mathbb{C}^k)^k\rightarrow \mathbb{C}$ be the determinant as a $k$-form over $\mathbb{C}^k$.
In the skew-symmetric tensor completion with skew-symmetric rank $r$,
the measurement map $g: (\mathbb{C}^{rk})^k \rightarrow \mathbb{C}$ is given as the sum of $r$ copies of $h_{\rm det}$.
Hence, we denote this $g$ by $rh_{\rm det}$.

The stabilizer $\Gamma$ of $rh_{\rm det}$ is 
$\{(A_1\oplus \dots \oplus A_r)(\Sigma\otimes I_k)\mid  A_i\in {\rm SL}(k,\mathbb{C}), \Sigma: \text{ a permutation matrix of size $r$}\}$.
We have $d_{\Gamma}=r(k^2-1)$ and $n_{\Gamma}=k$. 
The $(rk,r(k^2-1))$-sparsity is a necessary condition for the independence in the $rh_{\rm det}$-rigidity matroid, but it is not sufficient even when $r=1$  
(see Example~\ref{ex:0extension_fails} below for details).

As in the symmetric tensor case, proving the local $rh_{\rm det}$-rigidity 
for simple complete hypergraphs $\tilde{K}_n^k$ is already a difficult question.
Since $rh_{\rm det}$ is the sum of $r$ copies of $h_{\rm det}$,
 $\overline{{\rm im} f_{rh_{\rm det},G}}$ is the $r$-secant of $\overline{{\rm im} f_{h_{\rm det},G}}$.
When $G=\tilde{K}_n^k$, $\overline{{\rm im} f_{h_{\rm det},\tilde{K}_n^k}}$ is the affine cone of the Grassmannian ${\rm Gr}(k,n)$.
Hence, the locally $rh_{\rm det}$-rigid of $\tilde{K}_n^k$ is equivalent to the non-$r$-defectivity of ${\rm Gr}(k,n)$.
Unlike the Veronese case, the non-defectivity problem and the identifiability problem of Grassmannian varieties have not been solved yet.
See~\cite{AOP,massarenti2016non,MMident}.  

\paragraph{Chow decompositions.}
The same story holds for Chow decompositions if we replace the determinant function in the above discussion on skew-symmetric tensors with the permanent function. 
There is no complete answer to non-defectivity or identifiability of the corresponding varieties, so the local/global rigidity of complete hypergraphs are also not known. 
See~\cite{torrance2021all,torrance2022almost} for recent developments.

\section{Sufficient Conditions for Local $g$-rigidity}
\label{sec:suff}

The goal of this section is to extend combinatorial techniques in graph rigidity theory to $k$-uniform hyper-framework $g$-rigidity. 
Since $g$-rigidity is a rather general concept, we shall restrict our attention to polynomial maps $g$ which are multilinear or multiaffine (defined below).
Based on this assumption, we shall analyse combinatorial patterns of the Jacobians of the measurement maps $g$.  

\subsection{Multiaffine forms and combinatorial patterns of Jacobians}

A map $g:(\mathbb{F}^d)^k\rightarrow \mathbb{F}$ is said to be a {\em multilinear (resp.~multiaffine) $k$-form on $\mathbb{F}^d$} if 
it is linear (resp.~affine) on each argument on the vector space $\mathbb{F}^d$.
As we remarked when defining $g$-rigidity, $g$ is always assumed to be symmetric or anti-symmetric with respect to permutations of indices.
A symmetric (resp.~anti-symmetric) multilinear $k$-form is nothing but a symmetric (resp.~skew-symmetric) tensor of $(\mathbb{F}^d)^k$. 

Suppose $g$ is a symmetric (or~anti-symmetric) multiaffine $k$-form.
By regarding $g(x_1,\dots, x_k)$ as an affine function in the variable $x_j$, 
each entry of the gradient of $g$ with respect to $x_j$ can be considered as a polynomial map in the remaining points $x_1, \dots, x_{j-1}, x_{j+1},\dots, x_k$.
Moreover, since $g$ is symmetric (or anti-symmetric), the resulting polynomial map is independent of the choice of $j$ (up to the sign if $g$ is anti-symmetric).
Hence,  we may consider the gradient of $g$ with respect to a point $x_j$ 
as a polynomial map from $(\mathbb{F}^d)^{k-1}$ to $\mathbb{F}^d$,
and denote it by $\nabla g$. 

Based on Proposition~\ref{prop:inf_rigid}, our analysis of the local $g$-rigidity of a hypergraph $G=(V,E)$ proceeds by 
analyzing the rank of the Jacobian $\J f_{g,G}(p)$ at a generic $p$.
This Jacobian has the following structure.
Given that the map $f_{g,G}$ is defined from $(\mathbb{F}^d)^{V}$ to $\mathbb{F}^E$, we have that 
each row of $\J f_{g,G}(p)$ is indexed by a hyperedge in $E$ and 
each consecutive $d$-tuple of columns of $\J f_{g,G}(p)$ are indexed by a vertex in $V$.
The definition of $f_{g,G}$ further tells us that 
the $1\times d$ block $b(e,v)$ associated with a pair $(e,v)$ for $e\in E$ and $v\in V$ 
is 
\[
b(e,v)=\begin{cases}
\pm m_e(v)\nabla g(p(e-v))^{\top} & \text{if $v\in e$} \\
0 &\text{otherwise}.
\end{cases}
\]
Here, by $p(e-v)$, we mean a tuple of $(k-1)$ points in $\{p(u): u\in e-v\}$,
and $m_e(v)$ denotes the multiplicity of $v$ in $e$.
The sign is always positive if $g$ is symmetric, and it depends on the ordering of $v$ in $e$ if $g$ is anti-symmetric.

\begin{example}\label{ex:symmetric_tensor}
A primary example arises in the symmetric tensor completion problem. 
As we explained in Section~\ref{sec:examples2},
the corresponding map $g$ is the sum $dh_{\rm prod}$ of $d$ copies of $h_{\rm prod}$.
Clearly $h_{\rm prod}$ is symmetric multilinear,
and hence so is $dh_{\rm prod}$.
The map $\nabla (dh_{\rm prod}):(\mathbb{F}^d)^{k-1}\rightarrow \mathbb{F}^d$ is given by 
\[
\nabla (dh_{\rm prod})(x_1,\dots, x_{k-1})=
\begin{pmatrix} \nabla h_{\rm prod}(x_{11}, \dots, x_{1k-1}) \\ \vdots \\ \nabla h_{\rm prod}(x_{d1},\dots, x_{dk-1})\end{pmatrix}
=
\begin{pmatrix} x_{11}\dots x_{1k-1} \\ \vdots \\x_{d1}\dots  x_{dk-1}\end{pmatrix}
\]
for $x_i=(x_{1i}, \dots, x_{di})^{\top}\in\mathbb{F}^{d}$ for $i=1,\dots, k-1$.
% For a $d$-dimensional $k$-uniform hyper-framework $(G,p)$ with $G=(V,E)$, the Jacobian matrix of $f_{dh_{\rm prod},G}$ at $p$ is of size $|E|\times d|\MV|$ representing the linear system in $\dot{p}\in (\mathbb{F}^d)^{V}$:
% \begin{equation*}%\label{eq:system2}
% \sum_{v\in e} m_e(v)\nabla dh_{\rm prod}(p(e-v))\cdot \dot{p}(v)

Suppose
$k=3$, $d=2$, $G=(\MV,E)$ with $\MV=\{a,b,c,d\}$ and $E=\{aaa,aab,abc,bcd\}$.
Then, the Jacobian is the matrix 
\[
\kbordermatrix{
 & (a,1) & (a,2) & (b,1) & (b,2) & (c,1) & (c,2) & (d,1) & (d,2) \\
aaa & 3x_{1a}^2 & 3x_{2a}^2 & 0 & 0 & 0 & 0 & 0 & 0 \\
aab & 2x_{1a}x_{1b} & 2x_{2a}x_{2b} & x_{1a}^2 & x_{2a}^2 & 0 & 0 & 0 & 0\\
abc & x_{1b}x_{1c} & x_{2b} x_{2c} & x_{1c}x_{1a} & x_{2c}x_{2a} & x_{1a}x_{1b} & x_{2a}x_{2b} & 0 & 0\\
bcd & 0  & 0  & x_{1c}x_{1d} & x_{2c}x_{2d} & x_{1d}x_{1b} & x_{2d}x_{2b} & x_{1b}x_{1c} & x_{2b}x_{2c}
},
\]
where each row is associated with a hyperedge $e\in E$ and each column is associated with a pair $(v, i)$ of a vertex $v$ and a coordinate axis $i$.  
\end{example}

\subsection{Extension operations}
A {\em $d$-valent extension} creates a new $k$-uniform hypergraph from a given $k$-uniform hypergraph by adding a new vertex with $d$ distinct new hyperedges containing the new vertex.
We say that an extension operation is {\em simple} if each new hyperedge is simple.
See Figure \ref{fig:extension} for examples.
It is one of the most basic and widely used theorems in rigidity theory that any $d$-valent simple extension preserves $d$-dimensional rigidity.
We shall extend this to the $g$-rigidity setting for $g:(\mathbb{F}^d)^k\rightarrow \mathbb{F}$.

\begin{figure}[h]
\begin{center}
\begin{tikzpicture}[scale=0.8]

\draw (0,0) circle (30pt);
\draw (5,0) circle (30pt);

\filldraw (.5,.6) circle (3pt);
\filldraw (.5,-.6) circle (3pt);
\filldraw (.3,0) circle (3pt);

\filldraw (5.5,.6) circle (3pt);
\filldraw (5.5,-.6) circle (3pt);
\filldraw (5.3,0) circle (3pt);
\filldraw (7,0) circle (3pt);

\draw[black,thick]
(5.5,.6) -- (7,0) -- (5.5,-.6);

\draw[black,thick]
(5.3,0) -- (7,0);

\draw[black,thick]
(2,0) -- (3,0);

\draw[black,thick]
(2.8,.2) -- (3,0) -- (2.8,-.2);

%%%%

\draw (10,0) circle (30pt);
\draw (15,0) circle (30pt);
\filldraw (10.3,0) circle (3pt);
\filldraw (10.5,.6) circle (3pt);
\filldraw (10.5,-.6) circle (3pt);

\filldraw (15.5,.6) circle (3pt);
\filldraw (15.5,-.6) circle (3pt);
\filldraw (15.3,0) circle (3pt);
\filldraw (17,0) circle (3pt);

\draw[thick] plot[smooth, tension=1] coordinates{(15.1,-.2)(15.4,.7) (17.3,-.1) (15.1,-.2)};

\draw[thick] plot[smooth, tension=1] coordinates{(15.1,.2)(15.4,-.7) (17.3,.1) (15.1,.2)};

\draw[black,thick]
(12,0) -- (13,0);

\draw[black,thick]
(12.8,.2) -- (13,0) -- (12.8,-.2);

\end{tikzpicture}
\end{center}
\caption{(Left) A 3-valent simple extension of a 2-uniform hypergraph. (Right) A 2-valent simple extension of a 3-uniform hypergraph.}
\label{fig:extension}
\end{figure}
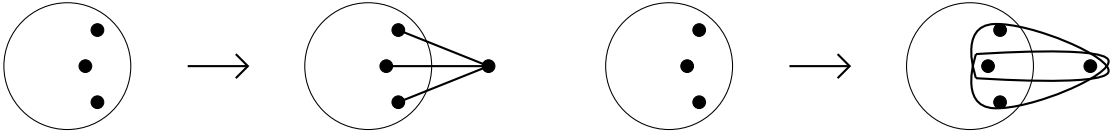

For a $d$-dimensional hyper-framework $(G,p)$ and a vertex $v\in V(G)$, consider the matrix $B$ obtained by aligning $\nabla g(p(e-v))^{\top}$ as row vectors over all hyperedges in $G$ incident to $v$, i.e., 
\begin{equation}\label{eq:B}
B=\begin{pmatrix}
    \nabla g(p(e_1-v))^{\top}\\
    \vdots \\
    \nabla g(p(e_j-v))^{\top}
\end{pmatrix},
\end{equation}
where $e_1,e_2, \dots, e_j$ are hyperedges in $G$ incident to $v$.
We say that $v$ is {\em stable} in $(G,p)$ if $B$ has full column rank, i.e., $\rank B=d$.

\begin{lemma}\label{lem:extension}
Assume Setup~\ref{setup}. Consider $g:(\mathbb{F}^d)^k\rightarrow \mathbb{F}$ and let $G=(V,E)$ be a $k$-uniform hypergraph.
\begin{itemize}
    \item If $(G,p)$ is infinitesimally $g$-rigid with $|V|\geq n_{\Gamma_g}+1$, then every vertex is stable in $(G,p)$.
    \item If $G$ is obtained from $H$ by 0-extension by adding a new vertex $v$ and $|V(H)|\geq n_{\Gamma_g}$, then $(G,p)$ is infinitesimally $g$-rigid if and only if 
    $(H,p_{|V(H)})$ is infinitesimally $g$-rigid and $v$ is stable in $(G,p)$.
\end{itemize}
 \end{lemma}
\begin{proof}
Suppose that $H$ is obtained from $G$ by removing a vertex $v$. 
Let  $e_1,\dots, e_j$ be the hyperedges in $G$ incident to $v$.
%Let $H$ be the subgraph of $G'$ consisting of $e_1,\dots, e_d$.
Then, $\J f_{g,G'}$  is partitioned into submatrices as follows:
\[
\J f_{g,G}=\kbordermatrix{
 & v & V(H) \\
\{e_1,\dots, e_d\} & B & \ast \\
E(H) & 0 & \J f_{g,H}
}.
\]
Hence, since $|V(G)|\geq n_{\Gamma_g}+1$, 
we may use Proposition \ref{prop:inf_rigid} to deduce that every vertex is stable in $(G,p)$. 

The second statement follows similarly. If $|V(H)|\geq n_{\Gamma_g}$, $(H,p_{|V(H)})$ is infinitesimally $g$-rigid and $v$ is stable in $(G,p)$, then the block matrix form shown above allows us to deduce, via Proposition \ref{prop:inf_rigid}, that $(G,p)$ is infinitesimally $g$-rigid. Conversely, if $(G,p)$ is infinitesimally $g$-rigid then $v$ is stable in $(G,p)$ by the first statement and $(H,p_{|V(H)})$ is infinitesimally $g$-rigid.
\end{proof}

When $g$ is symmetric, we say that $g$-rigidity has {\em the extension property} if 
any $d$-valent extension of a $g$-rigid graph $G=(V,E)$ with $|V|\geq n_{\Gamma_g}$ is still $g$-rigid.
When $g$ is anti-symmetric, the extension property is defined by checking only simple extensions since the value of $g$ is zero for a non-simple hyperedge.

We shall now prove the extension property of $g$-rigidity for a certain class of multiaffine $k$-forms $g$. To see this, we need the following notion.
A variety is called {\em degenerate} if it is contained in a hyperplane.
We say that $g:(\mathbb{F}^d)^k\rightarrow \mathbb{F}$ is {\em non-degenerate} if $\overline{{\rm im}f_{g, K_n^k}}$ is non-degenerate for any $n$.
For example, $h_{\rm prod}$ is non-degenerate since 
$\overline{{\rm im}f_{h_{\rm prod}, K_n^k}}$ is the affine cone of the Veronese variety of degree $k$, which is known to be non-degenerate.

\begin{lemma}\label{lem:extension_tensor}
Consider $g:(\mathbb{F}^d)^k\rightarrow \mathbb{F}$ written as 
the sum of $d$ copies of a non-zero multiaffine $h:(\mathbb{F}^1)^k\rightarrow \mathbb{F}$.
Suppose that $\nabla h$  is non-degenerate.
Then $g$-rigidity has the extension property.
\end{lemma}
\begin{proof}
    By Lemma~\ref{lem:extension}, it suffices to check that the matrix
    $B$ defined by Equation (\ref{eq:B}) is non-singular for any vertex $v$, $d$ distinct hyperedges $e_1,\dots, e_d$ incident to $v$, and a generic $d$-dimensional point-configuration $B$.
    %Since $h:(\mathbb{F}^1)^k\rightarrow \mathbb{F}$,
    Note that $\nabla h$ is a polynomial map from $(\mathbb{F}^1)^{k-1}$ to $\mathbb{F}$.
    In particular, $\nabla h$ is a $(k-1)$-form over $\mathbb{F}$,
    so the $\nabla h$-measurement map is defined for each $(k-1)$-uniform hypergraph.
    Let $H$ be the hypergraph consisting of $(k-1)$-uniform hyperedges $e_1-v, e_2-v, \dots, e_d-v$.
    Then the columns of $B$ are $f_{\nabla h, H}(x_1), f_{\nabla h, H}(x_2), \dots, f_{\nabla h, H}(x_d)$ for some $1$-dimensional generic point-configurations $x_1, \dots, x_d$.
   
    We now show that $f_{\nabla h, H}(x_1), \dots, f_{\nabla h, H}(x_d)$ are linearly independent.
    Since $\nabla h$ is non-degenerate, $\overline{{\rm im}f_{\nabla h, K_n^{k-1}}}$ is non-degenerate. 
    Since the non-degeneracy is preserved by projection, this in turn implies the non-degeneracy of $\overline{{\rm im}f_{\nabla h, H}}$. Therefore, any generic $d$ points on $\overline{{\rm im}f_{\nabla h, H}}$ spans the ambient space $\mathbb{F}^{E(H)}\simeq \mathbb{F}^d$. In particular, for generic $x_1, \dots, x_d$, we have that
    $f_{\nabla h, H}(x_1),\dots, f_{\nabla h, H}(x_d)$ are linearly independent. This implies that $B$ is non-singular, and hence $g$-rigidity has the extension property.  
\end{proof}

An important special case is when $g$ is the map for symmetric tensor completions.
\begin{corollary}\label{cor:extension}
    Suppose $g$ is the sum of $d$ copies of $h_{\rm prod}$.
    Then $g$-rigidity has the extension property.
\end{corollary}
\begin{proof}
Observe that $\nabla h_{\rm prod}$ is again the product map of $(k-1)$ variables. 
So $\overline{{\rm im}f_{\nabla h_{\rm prod}, K_n^k}}$ is the affine cone of the Veronese variety of degree $k-1$, which is non-degenerate.
Hence, $\nabla h_{\rm prod}$ is non-degenerate and Lemma~\ref{lem:extension_tensor} can be applied.
\end{proof}

Corollary~\ref{cor:extension} is false in general. The following example shows that the failure of the extension property can actually occur in real applications.

\begin{example}\label{ex:0extension_fails}
Let $h_{\rm det}:(\mathbb{F}^k)^k\rightarrow \mathbb{F}$ be the determinant as a multilinear $k$-form.
For $h_{\rm det}$-rigidity to have the extension property, 
it is necessary that the matrix $B$ in (\ref{eq:B}) is non-singular for any vertex $v$
and  any $d$ distinct hyperedges $e_1,\dots, e_d$ incident to $v$. 
Denote $e_i-v=\{u_{i1}, \dots, u_{ik-1}\}$.
Then $\nabla h_{\rm det}(p(e_i-v))=\pm p(u_{i1})\wedge p(u_{i2}) \wedge  \dots\wedge p(u_{ik-1})$ by using the standard basis in the exterior algebra of $\mathbb{F}$.
Then $B$ is singular if and only if 
$\{p(u_{i1})\wedge p(u_{i2}) \wedge  \dots\wedge p(u_{ik-1}): 1\leq i\leq k\}$ is linearly dependent.
The latter condition holds if and only if $\langle p(u_{i1}), p(u_{i2}), \dots, p(u_{ik-1})\rangle$ has a common nonzero subspace for $i=1,\dots, k$. If all hyperedges $e_1,\dots, e_k$ contain a common vertex other than $v$, then the condition fails. \qed
\end{example}

The precise characterisation of when $g$-rigidity exhibits the extension or simple extension property remains an open question. In the case where $g$ corresponds to the Euclidean distance, several recursive graph operations have been identified that preserve $g$-rigidity and have been used to characterise classes of $g$-rigid graphs (see, for example, \cite{HPG,whiteley1996some}). One of the most commonly used operations is the 1-extension operation, which operates on a graph by removing an edge ${v_1, v_2}$ and introducing a new vertex $v$ connected to $d+1$ vertices of $G$, including $v_1$ and $v_2$.
While this 1-extension operation can be extended to hypergraphs, it does not generally preserve $g$-rigidity. A simple example illustrating this is the 1-dimensional matrix completion problem, where the 1-extension operation corresponds to edge subdivision. In this case, it is well-known that odd cycles are independent in the $g$-rigidity matroid, whereas even cycles form circuits \cite{singer2010uniqueness}.
Therefore, the simple extension property is not universal for all types of polynomial maps $g$. Therefore, a deeper understanding of the properties and restrictions imposed by different types of $g$-rigidity is crucial in order to characterise the conditions under which the extension property holds.

\subsection{Packing-type sufficient condition}

In this section, we aim to establish a packing-type condition that is sufficient for the local rigidity of hypergraphs $G$. Let us assume that $g$ can be expressed as the sum of $t$ copies of $h$. In Lemma~\ref{lem:decomposition}, we demonstrated that any independent set in ${\cal M}_{g,n}$ can be decomposed into $t$ independent sets in ${\cal M}_{h,n}$.
While the converse direction does not hold in general, we will prove in this subsection that a decomposition with an additional property guarantees the validity of the converse direction.

Recall that, for a hyperedge $e\in  \genfrac{\{}{\}}{0pt}{}{[n]}{k}$ and $u\in e$,  $e-u$ (resp., $e+u$) denotes the multiset obtained from $e$ by reducing (resp., increasing) the multiplicity of $u$ by one.
We denote by ${\rm supp}(e)$ the set obtained from $e$ by ignoring the multiplicity of each element.

\medskip

Let $G=(V,E)$ be a $k$-uniform hypergraph.
For $X\subseteq V$, 
let $E_G[X]$ be the set of hyperedges  $e$ of $G$ with ${\rm supp}(e)\subseteq X$,
and  $G[X]$ be the subgraph of $G$ induced by $X$, i.e., $G[X]=(X,E_G[X])$.
 For $E'\subseteq E$, we define the closed neighbour set
 $N_G(E')$ as 
\[
N_G(E') = \{e-u+v : e\in E', u\in e, v\in [n]\} \cap E.
\]
Note that $E'\subseteq N_G(E')$.

\begin{theorem}\label{thm:packing}
Consider $g:(\mathbb{F}^d)^k\rightarrow \mathbb{F}$ written as 
the sum of $t$ copies of a non-zero multilinear $k$-form $h:(\mathbb{F}^s)^k\rightarrow \mathbb{F}$, where $st=d$.
Let   $G=([n], E)$ be a $k$-uniform hypergraph on $[n]$, 
and 
${\cal X}=\{X_1, X_2,\dots, X_t\}$ be a family of subsets $X_i$ of $[n]$ with $|X_i|\geq n_{\Gamma_h}$.
Denote  $F_i=N_G(E_G[X_i])$ for each $i=1,\dots, t$.
Suppose that:
\begin{itemize}
\item[{\rm (P1)}]\label{P1} for every  $i$, 
the hypergraph $([n],F_i)$ is locally $h$-rigid as a hypergraph on $[n]$;
% \Jim{I think it would be clearer to say that "$([n],F_i)$ is locally $h$-rigid" to emphasise that this hypergraph is a spanning subgraph of $G$.}
\item[{\rm (P2)}]\label{P2} for every $i$, $G[X_i]$ is locally $h$-rigid;   and  
\item[{\rm (P3)}]\label{P3} for any $i,j$ with $i\neq j$, and for any  $e\in F_i$ and $v\in e$,
  ${\rm supp}(e-v)\not \subseteq X_j$.
\end{itemize}
Then $G$ is locally $g$-rigid.
\end{theorem}

\begin{proof}
For simplicity of notation, we present the proof in the case when $s=1$ and $t=d$, but the extension to the general case is straightforward.
Let $G'=([n], \bigcup_i F_i)$. It is sufficient to show that $G'$ is locally $g$-rigid in $\mathbb{R}^d$ as $G'$ is a subgraph of $G$.

We first observe that $F_i\cap F_j=\emptyset$. Indeed, if $e\in F_i\cap F_j$,
then the definition of the operator $N_G$ implies that  there are some $u\in e$ and $v\in [n]$ such that $e-u+v\in E_G[X_j]$.
This, in particular, implies that ${\rm supp}(e-u)\subseteq X_j$, contradicting (P3).
Hence, ${F_1,\dots, F_d}$ is a family of disjoint edge sets. We also note that $E_G[X_i]\subseteq F_i$ by definition.

For each $i$ with $1\leq i\leq d$, we define $x_i\in \mathbb{F}^{n}$ by
\begin{equation}\label{eq:pi}
x_i(v)=\begin{cases}
x_{i,v} & \text{ if $v\in X_i$} \\
0 & \text{ otherwise},
\end{cases}
\end{equation}
where $x_{i,v}$ denotes a generic number.
Let $p\in (\mathbb{F}^d)^{n}$ be the point configuration obtained by stacking $x_i$ for all $1\leq i\leq d$ as row vectors,
i.e., $x_i$ is the $i$-th coordinate vector of $p$.
Consider $\J f_{g, G'}(p)$. Each column of $\J f_{g, G'}(p)$ is indexed by $(v,i)\in [n]\times [d]$ and each row is indexed by $e\in E(G')$.
We use $b(e,v,i)$ to denote the entry of $\J f_{g, G'}(p)$ indexed by $e$ and $(v,i)$.

\begin{claim}\label{claim:decomp}
Let $i,j\in [d]$,  $v\in [n]$, and $e\in F_i$.
%Then $r(e,v,j)\in \{0,1\}$.
\begin{itemize}
\item If $e\in E_G[X_i]$, then $b(e,v, j)\neq 0$ if and only if $i=j$ and $v\in e$.

\item If $e\in F_i\setminus E_G[X_i]$, then $b(e,v, j)\neq 0$ if and only if $i=j$ and $v\in {\rm supp}(e)\setminus X_i$. 
%Moreover, if $b(e,v, j)\neq 0$, then $b(e,v, j)=\prod_{u\in e-v} x_{j,u}$.
\end{itemize}
\end{claim}
\begin{proof}
Since $g=\sum_{i=1}^d h(x_i)$, we have
\begin{equation}\label{eq:claim}
b(e,v,j)=\begin{cases}
m_e(v) \nabla h(p_j(e-v)) &\text{ if $v\in e$}\\
0 & \text{ otherwise}.
\end{cases}
\end{equation}
%Since $p_j(u)\in \{0,1\}$,  we get $r(e,v,j)\in \{0,1\}$.
Since $h$ is nonzero and multilinear, $\nabla h$ is also nonzero and multilinear.

We first consider a hyperedge $e\in E_G[X_i]$.
Suppose $i=j$ and $v\in e$. 
Then, by Equation (\ref{eq:pi}),  $p_j(e-v)$ is a $(k-1)$-tuple of  nonzero numbers, so $\nabla h(p_j(e-v))$ is nonzero.
Hence, $b(e,v,j)\neq 0$ by Equation (\ref{eq:claim}).

To see the converse direction, suppose next $b(e,v,j)\neq 0$.
Then, by Equation (\ref{eq:claim}), we have $v\in e$.
The multilinearity of $\nabla h$ and Equation (\ref{eq:pi}) further imply  that,
if ${\rm supp}(e-v)\not\subseteq X_j$, then 
$\nabla h(p_j(e-v))=0$.
So, by Equation (\ref{eq:claim}), ${\rm supp}(e-v)\subseteq X_j$ should hold.
Hence, $i=j$ follows from (P3). 

For the second statement, we consider a hyperedge $e\in F_i\setminus E_G[X_i]$.
Suppose $b(e,v,j)\neq 0$.
Then by Equation (\ref{eq:pi}) $i=j$ clearly holds. 
We further have $v\in {\rm supp}(e)\setminus X_i$,
since otherwise $e$ would contain some vertex $w\in {\rm supp}(e)\setminus X_i$ other than $v$
and $p_i(w)=0$ would follow by $w\notin X_i$, which would imply $b(e,v,j)=0$ by the multilinearity of $\nabla h$.
This completes the proof of the necessity for $b(e,v,j)\neq 0$.
The sufficiency can be checked easily from Equations (\ref{eq:pi}) and (\ref{eq:claim}).
\end{proof}

Let $G_i=([n],F_i)$.
Claim~\ref{claim:decomp} implies that $\J f_{g,G'}(p)$ is in the following block-diagonalized form:
\begin{equation}\label{eq:block}
\J f_{g,G'}(p)=\kbordermatrix{
 & I_1 & I_2 & \dots & I_d\\
F_1 & \J f_{h,G_1}(x_1) & 0 & \dots & 0 \\ 
F_2 & 0 & \J f_{h,G_2}(x_2) & \dots & 0 \\ 
\vdots & \vdots & 0 & \ddots & \vdots \\
F_d & 0 & \dots & 0 & \J f_{h,G_d}(x_d) 
}
\end{equation}
where $I_i:=\{(v,i): v\in [n]\}$ for $i=1,\ldots,d$. 
Let $\Gamma_h$ and $\Gamma_g$ be the stabilizers of $h$ and $g$, respectively.
\begin{claim}\label{claim:block}
We have $\rank \J f_{h,G_i}(x_i)=n-d_{\Gamma_h}$  for all $1\leq i \leq d$.
\end{claim}
\begin{proof}
(For completeness we give a proof of the claim for general $s$. 
We show $\rank \J f_{h,G_i}(x_i)=sn-d_{\Gamma_h}$.)
Let $F_{i,v}$ be the set of hyperedges incident to $v$ in $G_i$,
and consider the submatrix $B_v$ of $\J f_{h,G_i}(x_i)$ induced by the rows indexed by $F_{i,v}$ and the corresponding $s$ columns of $v$. This submatrix $B_v$ corresponds to the matrix $B$ in Equation~(\ref{eq:B}) for $v$ and $F_{i,v}=\{e_1,\dots, e_j\}$.
By Claim~\ref{claim:decomp}, $\J f_{h,G_i}(x_i)$ has the form 
\begin{equation}\label{eq:block2}
\J f_{h,G_i}(p)=\kbordermatrix{
 & X_i & v_1 & \dots & v_l\\
F_i\setminus \bigcup_{v\in [n]\setminus X_i} F_{i,v} & \J f_{h,G[X_i]}(x_i) & 0 & \dots & 0 \\ 
F_{i,v_1} & 0 & B_{v_1} & \dots & 0 \\ 
\vdots & \vdots & 0 & \ddots & \vdots \\
F_{i,v_l} & 0 & \dots & 0 & B_{v_l} 
}
\end{equation}
where we denote $[n]\setminus X_i=\{v_1,\dots, v_l\}$.
By (P2), $G[X_i]$ is locally $h$-rigid, and hence $\rank \J f_{h,G[X_i]}(x_i)=sn-d_{\Gamma_h}$.
So it suffices to show that $\rank B_{v_j}=s$ for each $v_j\in [n]\setminus X_i$, i.e., $v_j$ is stable in $(G_i,x_i)$.

By (P1) and Lemma~\ref{lem:extension}, every vertex $v_j$ is stable in a generic framework of $G_i$ (but the current $x_i$ is not generic).
Recall that every edge in $F_{i,v_j}$ is written in the form $f+v_j$ for some $(k-1)$-multiset $f$ with ${\rm supp}(f)\subseteq X_i$,
and so the entries of $B_{v_j}$ are determined by the coordinates of $X_i$.
Since the coordinates of $X_i$ in $x_i$ are generic, $v_j$ is still stable in $(G_i,x_i)$.
This completes the proof of Claim~\ref{claim:block}.
\end{proof}
Hence,
\begin{align*}
    dn-d_{\Gamma_g}&\geq \rank \J f_{g,G'}(p) \qquad \text{(by Proposition~\ref{prop:inf_rigid})} \\ 
    &=\sum_{i=1}^d \rank \J f_{h,G_i}(x_i) \qquad \text{(by Equation (\ref{eq:block}))}\\
    &=d(n-d_{\Gamma_h}) \qquad \text{(by Claim~\ref{claim:block})} \\
    &\geq dn-d_{\Gamma_g} \qquad \text{(by the definition of $\Gamma_h$ and $\Gamma_g$)},
\end{align*}
implying the local $g$-rigidity of $(G',p_i)$ by Proposition~\ref{prop:inf_rigid}.
This completes the proof.
\end{proof}

We give two applications of Theorem~\ref{thm:packing}. 

\begin{example}
Consider the product map $h_{\rm prod}$ and the sum $dh_{\rm prod}$ of $d$ copies of $h_{\rm prod}$. This corresponds to the case of symmetric tensor completions.
Consider $k\geq 3$ and a $k$-uniform hypergraph $G=([n],E)$ whose edge set is given by 
\[
E=\left\{e \in  \genfrac{\{}{\}}{0pt}{}{[n]}{k}:  m_e(v)\geq k-1 \text{ for some $v\in [n]$}\right\}.
\]
This hypergraph is the underlying graph of a tridiagonal symmetric tensor.
Suppose $d\leq n$.
We choose $\cX=\{X_1,\dots, X_d\}$ with $X_v=\{v\}$ for each $v\in [d]$.
We show $\cX$ satisfies the hypotheses (P1)-(P3) of Theorem~\ref{thm:packing}.
We use $v_1^{i_1}\dots v_s^{i_s}$ to denote the hyperedge consisting of $v_1,\dots, v_s$ with $m_e(v_j)=i_j$.
Then, $E=\{v^{k-1}w: v, w\in [n]\}$.
For each $v\in [d]$,  $E_G[X_v]=\{v^k\}$, and $F_v=\{v^{k-1}w: w\in [n]\}$.
The hypergraph $G[X_v]$ is locally $h_{\rm prod}$-rigid since $G[X_v]$ is a single-vertex graph with one edge. So (P2) holds.
Note that $F_v$ is obtained from $G[X_v]$ by a sequence of $1$-valent extension, implying (P1) by Corollary~\ref{cor:extension}.
To see (P3), consider any $e\in F_v$ and $u\in e$.
Since $e=v^{k-1}w$ for some $w\in [n]$, $e-u=v^{k-2}w$ if $v=u$ and $e-u=v^{k-1}$ if $u=w$.
Hence, by $k\geq 3$, ${\rm supp}(e-u)\not\subset X_{v'}=\{v'\}$ holds for any $v'$ with $v\neq v'$, and (P3) follows.
By Theorem~\ref{thm:packing}, we conclude that $G$ is locally $dh_{\rm prod}$-rigid if $d\leq n$.
 \end{example}

 \begin{example}
For another (more exciting) application of Theorem~\ref{thm:packing}, we need the following notation.
For positive integers $\alpha, \beta$, an $\alpha$-uniform family $\cX\subseteq {[n]\choose \alpha}$  is called {\em $\beta$-sparse} if 
$|X\cap Y|<\beta$ for any $X, Y\in \cX$.
We define the packing number by 
\[
{\rm packing}(n,\alpha,\beta)=\max\left\{ |\cX|: \cX\subseteq {[n]\choose \alpha}, \cX \text{ is $\beta$-sparse}\right\}.
\]
An asymptotic bound by R{\"o}dl \cite{rodl1985packing} gives
$
{\rm packing}(n,\alpha,\beta)=(1-o(1)) {n\choose \beta}/{\alpha\choose \beta}.
$

\smallskip
Recall that  $\tilde{K}_n^k$ denotes the simple complete $k$-uniform hypergraph on $n$ vertices.

\begin{corollary}\label{cor:packing_Kn}
Let $n, k\geq 3$ %$k$ and $n$ 
be positive integers % with $n, k\geq 3$, 
and consider $g:(\mathbb{F}^d)^k\rightarrow \mathbb{F}$ written as 
the sum of $t$ copies of a non-zero multilinear $k$-form $h:(\mathbb{F}^s)^k\rightarrow \mathbb{F}$ with $d=st$.
Suppose there is a positive integer $a$ such that 
\begin{itemize}
    \item $\tilde{K}_a^k$ is locally $h$-rigid,
    \item $a\geq n_{\Gamma_h}+1$, and
    \item $t\leq {\rm packing}(n,a,k-2)$.
\end{itemize}
Then $\tilde{K}_n^k$ is locally $g$-rigid.
\end{corollary}
\begin{proof}
We may assume $t={\rm packing}(n,a,k-2)$. 
Let $\cX=\{X_1,\dots, X_t\} \subseteq {[n]\choose a}$ be a $(k-2)$-sparse family achieving $p(n,a,k-2)$.
By Theorem~\ref{thm:packing}, it suffices to show that $\cX$ satisfies (P1)-(P3).

Choose $X_i\in \cX$.
Since $|X_i|=a$, the subgraph $\tilde{K}_n^k[X_i]$ is isomorphic to $\tilde{K}_a^k$, so
(P2) follows from the local $h$-rigidity of $\tilde{K}_a^k$.

To see (P1), consider $F_i$ as defined in the statement of Theorem~\ref{thm:packing} in the case when $G=\tilde{K}_n^k$, and let $G_i=([n],F_i)$.
Pick any vertex $u\in [n]\setminus X_i$.
Since $a\geq n_{\Gamma_h}+1$ and $K_a^k$ is locally $h$-rigid,
Lemma~\ref{lem:extension} implies that each vertex $v$ in $X_i$ is stable in a generic framework of $\tilde{K}_n^k[X_i]$.
Hence each vertex $v$ in $X_i$ remains stable in a generic framework of $\tilde{K}_n^k[X_i\cup \{u\}]$. 
Since $\tilde{K}_n^k[X_i\cup \{u\}]$ is a simple complete hypergraph, the symmetry further implies that $u$ is also stable in a generic framework  of $\tilde{K}_n^k[X_i\cup \{u\}]$.
Hence, by Lemma~\ref{lem:extension}, a subgraph of $\tilde{K}_n^k[X_i\cup \{u\}]$ is obtained from 
$\tilde{K}_n^k[X_i]$ by an $s$-valent extension which preserves local $g$-rigidity.
Note that $\tilde{K}_n^k[X_i\cup \{u\}]$ is a subgraph of $G_i$.
Applying the same argument to each vertex $u\in [n]\setminus X_i$,
we can further construct 
a spanning subgraph of $G_i$ from $\tilde{K}_n^k[X_i]$ by a sequence of $s$-valent extensions keeping local $g$-rigidity.
Thus $G_i$ is locally $g$-rigid.

Finally, suppose, for a contradiction, that (P3) does not hold. 
Then ${\rm supp}(e-v)\subseteq X_j$ for some $e\in F_i$, $v\in e$, and $X_j$ with $i\neq j$.
Since $e\in F_i$, there are some $u,w\in [n]$ such that $e-u+w\in E_{\tilde{K}_n^k}[X_i]$, so ${\rm supp}(e-u-v)\subseteq X_i\cap X_j$.
Since $|{\rm supp}(e-u-v)|=k-2$, this contradicts the fact that $\cX$ is $(k-2)$-sparse.
Thus (P3) holds, and the $g$-rigidity of $G$ follows from Theorem~\ref{thm:packing}.
\end{proof}
\end{example}

We note that when $h$ is the determinant function, Corollary~\ref{cor:packing_Kn} aligns with an observation in \cite{AOP}.

\section{Conditions for Global $g$-rigidity}\label{sec:global}

In this section we analyse global $g$-rigidity.

\subsection{Generic global $g$-rigidity}\label{subsec:global_graphs}
In Section~\ref{sec:tools}, we observed that %As we have seen in Section~\ref{sec:tools}, 
local $g$-rigidity is a generic property of the underlying hypergraphs. This notion naturally leads to the concept of local $g$-rigidity of hypergraphs. By applying Chevalley's theorem, similar to the approach in Proposition \ref{prop:matroid}(iv), we can extend this result to global $g$-rigidity when considering constraints over the field of complex numbers.

\begin{prop}
Global $g$-rigidity is a generic property of the underlying hypergraphs if the underlying field is complex. 
\end{prop}

Additionally, in the case of Euclidean distance where $g$ is the squared Euclidean distance function, there are notable results in the real case by Gortler, Healy, and Thurston \cite{Gortler2014, Jackson19}. They show that global $g$-rigidity is a generic property in the real setting. However, it is important to note that this result does not hold for other maps. For example, if $g$ represents the Euclidean inner product of two points over $\mathbb{R}$ \cite{JJTunique} or if $g$ is the map for volume rigidity \cite{Southgate}, then global $g$-rigidity is not a generic property over reals.

While global $g$-rigidity may not be a generic property in general, it would be useful if we can guarantee the global $g$-rigidity for any generic realization of a given hypergraph. Therefore, we define a hypergraph $G$ to be {\em globally $g$-rigid} (over $\mathbb{F}$) if $(G,p)$ is globally $g$-rigid for any generic point-configuration $p$ over $\mathbb{F}$. 

\smallskip
The following proposition provides motivation for focusing on the complex case when establishing a sufficient condition for global rigidity.
\begin{prop}\label{prop:complex_to_real}
If a hypergraph $G$ is globally $g$-rigid over $\mathbb{C}$, then it is globally $g$-rigid over $\mathbb{R}$.
\end{prop}
\begin{proof}
Since $G$ is globally $g$-rigid over $\mathbb{C}$,
any generic hyperframework $(G,p)$ of $G$ over $\mathbb{C}$ is globally rigid.
Since the genericity is defined in terms of algebraic independence over $\mathbb{Q}$, any generic configuration $q$ over $\mathbb{R}$ is also generic over $\mathbb{C}$.
So for any generic real configuration $q$,  $(G,q)$ is a generic hyperframework over $\mathbb{C}$, and it is globally rigid.
\end{proof}
\subsection{Identifiability and global $g$-rigidity}
\label{subsec:identifiability}
The following proposition highlights the usefulness of identifiability in verifying global rigidity.
%The following proposition implies how the identifiability is useful to check the global rigidity.
\begin{prop}\label{prop:identifiability_global}
Suppose  $g:(\mathbb{F}^d)^k\rightarrow \mathbb{F}$ is the sum of $t$ copies of a homogeneous  $h:(\mathbb{F}^s)^k\rightarrow \mathbb{F}$.
Suppose further that $\overline{{\rm im} f_{h,G}}$ is $t$-identifiable.
Then $G$ is globally $g$-rigid if and only if $G$ is globally $h$-rigid.
\end{prop}
\begin{proof}
Clearly global $h$-rigidity is necessary for global $g$-rigidity.
To see the converse direction, we take any generic $p\in (\mathbb{F}^d)^{V}$ 
and show that, if $G$ is globally $h$-rigid then $(G,p)$ is globally $g$-rigid.

Since $g$ is the sum of $t$ copies of $h$, 
$p\in (\mathbb{F}^d)^{V}$ can be decomposed into $(q_1,\dots, q_t)$ with $q_i\in (\mathbb{F}^s)^{V}$ such that 
$f_{g,G}(p)=\sum_{i=1}^t f_{h,G}(q_i)$.
Since $p$ is generic, each $q_i$ is generic.
So $(G,q_i)$ is globally $h$-rigid.
Hence, the $t$-identifiability of $\overline{{\rm im} f_{h,G}}$ implies the global $g$-rigidity of $(G,p)$.    
\end{proof}

We next review some sufficient conditions for checking the identifiability of varieties.
A key notion is that of tangential weak defectiveness due to Chiantini and Ottaviani~\cite{CO2012}.
Let ${\cal V}$ be a variety over $\mathbb{C}$, 
and denote by $T_x{\cal V}$ the tangent space at $x\in {\cal V}$.
The {\em $t$-tangential contact locus} ${\cal C}$ at  generic $t$ points $x_1,\dots, x_t$ is defined as 
$
{\cal C}=\overline{\{y\in {\cal V}: T_y{\cal V}\subseteq \langle T_{x_1}{\cal V},\dots,T_{x_t}{\cal V}\rangle\}}.
$
A variety ${\cal V}$ is said to be {\em $t$-tangentially weakly defective} 
if an irreducible component of ${\cal C}$ that contains $x_i$ has dimension at least two for some $i$.
(Since we are working in affine space, the tangential contact locus always contains a one-dimensional subspace of scaling.) 

It was shown that $t$-tangential weak non-defectiveness is a sufficient condition for $t$-identifiability.
\begin{theorem}[Chiantini and Ottaviani~\cite{CO2012}]\label{thm:twd}
Suppose an affine variety ${\cal V}$ over $\mathbb{C}$ is $t$-tangentially weakly nondefective.
Then ${\cal V}$ is $t$-identifiable.
\end{theorem}

A combination of Proposition~\ref{prop:identifiability_global} and Theorem~\ref{thm:twd} leads to the following criterion for global $g$-rigidity.

\begin{corollary}\label{cor:cov}
Suppose a homogeneous  $g:(\mathbb{C}^d)^k\rightarrow \mathbb{C}$ is the sum of $t$ copies of $h:(\mathbb{C}^s)^k\rightarrow \mathbb{C}$.
Then $G$ is globally $g$-rigid if 
\begin{itemize}
    \item $\overline{{\rm im} f_{h,G}}$ is $t$-tangentially weakly nondefective, and
    \item $G$ is globally $h$-rigid.
\end{itemize}
\end{corollary}

The tangential contact locus can be characterised algebraically, allowing for the development of numerical algorithms to test for $t$-tangentially weak defectiveness. Chiantini, Ottaviani, and Vannieuwenhoven \cite{COV2014} proposed an algorithm based on this description, which we will briefly review as it is closely related to a well-known method in rigidity.

Let $G$ be a $k$-uniform hypergraph with $n$ vertices and $m$ edges,
and suppose $g:(\mathbb{C}^d)^k\rightarrow \mathbb{C}$ is the sum of $t$ copies of $h:(\mathbb{C}^s)^k\rightarrow \mathbb{C}$. 
In view of Corollary~\ref{cor:cov}, we want to check 
the $t$-tangentially weakly defectiveness of 
$\overline{{\rm im} f_{h,G}}$.
Let $p\in (\mathbb{C}^d)^{n}$ be a generic point-configuration,
and let $\omega_1,\dots, \omega_k\in \mathbb{C}^m$ be a basis of the orthogonal complement of the tangent space of $\overline{{\rm im} f_{g,G}}$ at $f_{g,G}(p)$.
Since the tangent space of $\overline{{\rm im} f_{h,G}}$ at $f_{h,G}(q)$ is the image of $\J f_{h,G}(q)$, 
the $t$-tangential contact locus ${\cal C}$ is the closure of 
$\{q\in (\mathbb{C}^s)^{n}: %\in(\mathbb{F}^s)^n
\omega_i \in \ker \J f_{h,G}(q)^{\top} \ (i=1,\dots, k)\}$.
Note that $\omega_i \in \ker \J f_{h,G}(q)^{\top}$  forms a polynomial system in $q$ (where $\omega_i$ is fixed). Therefore, there is a polynomial map $\alpha_i:(\mathbb{C}^s)^{n} 
\rightarrow (\mathbb{C}^{s})^{n}$
 such that 
$\omega_i \in \ker \J f_{h,G}(q)^{\top}$ if and only if $\alpha_i(q)=0$.
So, the tangential contact locus is written as the closure of 
$\{q\in(\mathbb{C}^s)^n: \alpha_i(q)=0 \ (i=1,\dots, k)\}$.

It was shown in \cite{COV2014} that, if an affine variety ${\cal V}$ is $t$-tangentially weakly defective, then the $t$-tangential contact locus formed by generic points $x_1,\dots,x_t$ contains a smooth path connecting any $x_i$ to the locus.
%then the $t$-tangential contact locus  at generic points $x_1,\dots, x_t$ contains a smooth path to any $x_i$. 
Hence, in our terminology above, this implies that one can determine $t$-tangential weak non-defectiveness by examining the rank of the Jacobian matrix of $\alpha_i$.

\begin{remark}
It is worth mentioning that the resulting Jacobian matrix of $\alpha_i$ is the Laplacian matrix of $G$ weighted by $\omega_i$, in the case that $g$ is the squared Euclidean distance. So, the numerical condition in \cite{COV2014} has a similarity to the stress matrix condition  for global rigidity due to Connelly~\cite{Con05}
(or more generally to the shared stress kernel condition due to Gortler-Healy-Thurston~\cite{Gortler2014}). It should be noted, however, that in the case of the squared $d$-dimensional Euclidean distance the underlying variety is always $d$-tangentially weakly defective if $d>1$. 

\end{remark}
 
A recent result of Massarenti and Mella \cite{MMident} gives a powerful criterion for checking $t$-identifiability. 

\begin{theorem}[\cite{MMident}]\label{thm:meka}
Let ${\cal V}$ be a non-degenerate affine variety in $\mathbb{C}^m$.
Suppose that $(t+1)\dim {\cal V}\leq m$ and ${\cal V}$ is $(t+1)$-nondefective and $1$-tangentially weakly nondefective.
Then ${\cal V}$ is $t$-identifiable.
\end{theorem}

Combining Proposition~\ref{prop:identifiability_global} and Theorem~\ref{thm:meka}, we have the following criterion for global $g$-rigidity.

\begin{corollary}\label{cor:global1}
For a homogeneous 
$h:(\mathbb{C}^s)^k\rightarrow \mathbb{C}$ and a positive integer $d$, let $dh$ be the sum of $d$ copies of $h$, and suppose $d_{\Gamma_{(d+1)h}}=(d+1) d_{\Gamma_h}$.
Then, for a positive integer $d$, $G$ is globally $dh$-rigid if 
\begin{itemize}
    \item $G$ is locally $(d+1)h$-rigid, 
    \item $\overline{{\rm im} f_{h,G}}$ is 1-tangentially weakly nondefective, and
    \item $G$ is globally $h$-rigid.
\end{itemize}
\end{corollary}

The $1$-tangential contact locus is just an ordinary contact locus, so $1$-tangential weak defectiveness is equivalent to the degeneracy of the Gauss map.

\subsection{Sufficient conditions}

Corollary~\ref{cor:cov} or Corollary~\ref{cor:global1} provides a  way to reduce the global $g$-rigidity problem to the global $h$-rigidity problem, however it is a non-trivial problem to establish global $h$-rigidity.  
In order to accomplish this, we will extend Connelly's sufficient condition \cite{Con05} from graph rigidity theory.

The result is described in terms of the {\em adjacency matrix} of an edge weighted hypergraph.
Suppose $G=(V,E)$ is a $k$-uniform hypergraph and $w\in \mathbb{F}^{E}$ is a vector representing the weight of each edge. 
Consider the collection 
$E^{(k-1)}:=\{\sigma\in \genfrac{\{}{\}}{0pt}{}{V}{k-1}: \sigma\subseteq e\in E\}$ of multisets of size $(k-1)$ contained in some hyperedge in $G$.
We define the adjacency matrix $A_{G,w}$ to be 
an $\mathbb{F}$-matrix of size 
$|V|\times |E^{(k-1)}|$ such that:
\[
A_{G,w}[v,\sigma]=\begin{cases}
    m_e(v) w(e) & (\text{if $e=\sigma+v$ is in $G$})\\
    0 & (\text{otherwise}),
\end{cases}
\]
where each row is indexed by a vertex $v\in V$ and each column is indexed by $\sigma\in E^{(k-1)}$.
We will often look at the case when the edge weight $w$ is in the left kernel of $\J f_{g,G}(p)$.
%See Example~\ref{eq:global_tensor} for an example.
See Example~\ref{eq:global_tensor_new} for an example.

In order to deal with the case of anti-symmetric functions, we  consider the signed variant of the adjacency matrix.
For this, we assume that the elements of $V$ are totally ordered, and consider~$A^{s}_{G,w}$~with:
\[
A^{s}_{G,w}[v,\sigma]=\begin{cases}
    {\rm sign}(e,v)w(e)m_e(v) & (\text{if $e=\sigma+v$ is in $G$})\\
    0 & (\text{otherwise})
\end{cases}
\]
for each $v\in V$ and $\sigma\in E^{(k-1)}$,
where ${\rm sign}(e,v)$ denotes the standard sign function of permutations, which is positive (resp. negative) if the ordering of $v$ in $e$ is odd (resp. even). 

%\begin{example}\label{eq:global_tensor}
\begin{example}\label{eq:global_tensor_new}
Consider a $4$-uniform hypergraph $G_1=(\{a,b\}, \{aaaa, abbb,bbbb\})$ with edge-weight $\omega_1$ and 
a $4$-uniform hypergraph  $G_2=(\{a,b\}, \{aaaa,aabb,bbbb\})$ with edge-weight $\omega_2$.
Then
\begin{align*}
A_{G_1,\omega_1}&=
\kbordermatrix{
 & aaa & abb & bbb \\
 a & 4\omega_1(aaaa) & 0 & \omega_1(abbb) \\
 b & 0 & 3\omega_1(abbb) & 4\omega_1(bbbb)
}
\\
A_{G_2,\omega_2}&=
\kbordermatrix{
 & aaa & aab & abb & bbb \\
 a & 4\omega_2(aaaa) & 0 & 2\omega_2(aabb) & 0 \\
 b & 0 & 2\omega_2(aabb) & 0 & 4\omega_2(bbbb)
}.
\end{align*}
\end{example}

\begin{prop}\label{prop:G1}
Suppose $g:(\mathbb{F}^d)^k\rightarrow \mathbb{F}$ is a  multilinear $k$-form over $\mathbb{F}^d$, and $(G,p)$ is a $k$-uniform hyper-framework.
Let $P$ be a matrix of size $d\times |E^{(k-1)}|$ obtained by aligning $\nabla g(p(\sigma))$ as a column vector for each $\sigma\in E^{(k-1)}$. 
\begin{itemize}
 \item[{\rm (i)}] If $g$ is symmetric, then $A_{G,w}P^{\top}=w^{\top} \J f_{g,G}(p)$ for any edge weight $w$ of $G$.
 \item[{\rm (ii)}] If $g$ is anti-symmetric, then $A^{s}_{G,w}P^{\top}=w^{\top} \J f_{g,G}(p)$ for any edge weight $w$ of $G$.
 \item[{\rm (iii)}] If $(G,p)$ is infinitesimally $g$-rigid, then $\rank P=d$.
\end{itemize}
\end{prop}
\begin{proof}
    Assume $g$ is symmetric. By definition, the $v$-th entry of $w^{\top} \J f_{g,G}(p)$
    is 
    \[
    \sum_{\sigma\in E^{(k-1)}: \sigma+v\in E} w(\sigma+v) m_{\sigma+v}(v) \nabla g(p(\sigma)).
    \]
    Reshaping this in terms of $\nabla g(p(\sigma))$, 
    we obtain $A_{G,w}P^{\top}.$ The same argument also applies to the case when $g$ is anti-symmetric. We thus obtain (i) and (ii).

To see (iii), assume $(G,p)$ is infinitesimally $g$-rigid.
By Lemma~\ref{lem:extension},
 every vertex is stable in $(G,p)$.
 So the columns of $P$ indexed by $d$ hyperedges incident to a vertex forms an independent set, which means that $P$ has full row-rank.
\end{proof}

\begin{prop}\label{prop:G2}
Suppose $g:(\mathbb{F}^d)^k\rightarrow \mathbb{F}$ is a  multilinear $k$-form and $(G,p)$ is a generic locally $g$-rigid $k$-uniform hyper-framework.
Suppose 
\begin{itemize}
    \item $\dim\left( \bigcap_{\omega\in \ker \J f_{G,g}(p)^{\top}} \ker A_{G,\omega}\right)=d$ if $f$ is symmetric, and
\item $\dim\left( \bigcap_{\omega\in \ker  \J f_{G,g}(p)^{\top}} \ker A^{s}_{G,\omega}\right)=d$ if $f$ is anti-symmetric.
\end{itemize}
Then for each $q\in f_{g,G}^{-1}(f_{g,G}(p))$ there exists $T\in \mathbb{F}^{d\times d}$ such that 
\[
\nabla g(q(\sigma))=T\nabla g(p(f)) \qquad \text{for all $\sigma\in E^{(k-1)}$}.
\]
\end{prop}
\begin{proof}
We give the proof only in the case when $f$ is symmetric
since the anti-symmetric case is identical.

By Proposition~\ref{prop:matroid}(iv) and the generic assumption on $p$, $f_{g,G}(p)$ is a regular value of $f_{g,G}$, so for any $q\in f_{g,G}^{-1}(f_{g,G}(p))$ 
the image of $ \J f_{g,G}(q)$ is the tangent space of $\overline{{\rm im} f_{g,G}}$ at $f_{g,G}(q)$. We have that: 
\begin{equation}\label{eq:G2-1}
\ker \J f_{g,G}(q)^{\top}=\ker \J f_{g,G}(p)^{\top},
\end{equation}
since $f_{g,G}(p)=f_{g,G}(q)$. For any point-configuration $q$ of $G$, let 
$Q$ be a matrix of size $d\times |E^{(k-1)}|$ obtained by aligning $\nabla g(q(\sigma))$ as the columns for all $\sigma\in E^{(k-1)}$. 
By Proposition~\ref{prop:G1} and Equation (\ref{eq:G2-1}),
\begin{equation}\label{eq:G2-2}
A_{G,w} Q^{\top}=0\quad \text{for any } \omega\in \ker \J f_{G,g}(p)^{\top}. 
\end{equation}
In particular, $A_{G,w}P^{\top}=0$.
Hence each row of $P$ belongs to $\bigcap_{\omega\in \ker  \J f_{G,g}(p)^{\top}} \ker A_{G,\omega}$.
 
Note that we have assumed %Equation (\ref{eq:G2-0}) implies that 
the dimension of $\bigcap_{\omega\in \ker  \J f_{G,g}(p)^{\top}} \ker A_{G,\omega}$ is $d$.
Moreover, we have $\rank P=d$ by Proposition~\ref{prop:G1}.
By combining those two facts with $A_{G,w}P^{\top}=0$,
we may conclude that 
the rows of $P$ form a basis of 
$\bigcap_{\omega\in \ker  \J f_{G,g}(p)^{\top}} \ker A_{G,\omega}$.
Since each row of $Q$ belongs to this space by Equation (\ref{eq:G2-2}),
there must be $T\in \mathbb{F}^{d\times d}$ such that 
$Q=TP$.
This is equivalent to the statement.
\end{proof}

\begin{prop}\label{prop:G3}
    Suppose $g:(\mathbb{F}^d)^k\rightarrow \mathbb{F}$ is a  multilinear $k$-form, $(G,p)$ is a generic locally $g$-rigid $k$-uniform hyper-framework, and $q\in f_{g,G}^{-1}(f_{g,G}(p))$.
Suppose there is $T\in \mathbb{F}^{d\times d}$ such that 
\[
\nabla g(q(\sigma))=T\nabla g(p(\sigma)) \qquad (\sigma\in E^{(k-1)}).
\]
Then 
\[
p(v)=T^{\top} q(v) \qquad (v\in V). 
\]
\end{prop}
\begin{proof}
Since $(G,p)$ is generic and locally $g$-rigid, 
every vertex is stable by Lemma~\ref{lem:extension}.
So there are $d$ hyperedges $e_1,\dots, e_d$ incident to $v$ in $G$
such that $\{\nabla g(p(e_i-v)): i=1,\dots, d\}$ is a basis of $\mathbb{F}^d$.
Since $q\in f_{g,G}^{-1}(f_{g,G}(p))$, we have 
$f_{g,G}(q)=f_{g,G}(p)$. By the multilinearity of $g$, this implies that
\[
\langle q(v), \nabla g(q(e_i-v))\rangle= \langle p(v), \nabla g(p(e_i-v))\rangle,
\]
and hence 
\[
\langle T^{\top} q(v)-p(v), \nabla g(p(e_i-v))\rangle =0 \quad (i=1,\dots, d). 
\] 
The fact that $\{\nabla g(p(e_i-v)): i=1,\dots, d\}$ is a basis now gives 
$p(v)=T^{\top} q(v)$ for all $v$.
\end{proof}

When $g$ is the determinant, a combination of Propositions~\ref{prop:G2} and~\ref{prop:G3} gives the following sufficient condition for global $g$-rigidity.

\begin{theorem}\label{thm:global_determinant}
Suppose $g:(\mathbb{F}^d)^d\rightarrow \mathbb{F}$ is the determinant as a multilinear $d$-form over $\mathbb{F}^d$.
Then $G=(V,E)$ is globally $g$-rigid if 
there exists a point-configuration $p\in (\mathbb{F}^d)^V\rightarrow \mathbb{F}$ such that
\begin{itemize}
\item $(G,p)$ is infinitesimally $g$-rigid, and 
\item $\dim \left(\bigcap_{w\in \ker  \J f_{g,G}(p)^{\top}} \ker A^{s}_{G,w}\right)=d$.
\end{itemize}
\end{theorem}
\begin{proof}
The infinitesimal rigidity of $(G,p)$ implies that the rank of $\J f_{g,G}(p)$ takes the maximum possible value over all point-configurations.
Note that $\ker \J f_{g,G}(p)$ changes continuously if we perturb the entries of $p$ continuously.
So any sufficiently small continuous perturbation of the entries of $p$ does not change $\dim \left(\bigcap_{w\in \ker  \J f_{g,G}(p)^{\top}} \ker A_{G,w}\right)$. 
Hence, we may assume that $p$ is generic.

To see the global $g$-rigidity of $(G,p)$, pick any $q\in f_{g,G}^{-1}(f_{g,G}(p))$. By Propositions~\ref{prop:G2} and~\ref{prop:G3},
there is $T\in \mathbb{F}^{d\times d}$ such that 
$q(v)=Tp(v)$ for all $v\in V(G)$.
Pick any hyperedge $e$ in $G$, and assume $e$ consists of $k$ vertices $v_1,\dots, v_k$.
Since $q\in f_{g,G}^{-1}(f_{g,G}(p))$, we have 
\[
\det \begin{pmatrix} p(v_1) & \dots & p(v_k)\end{pmatrix}=\det \begin{pmatrix} q(v_1) & \dots & q(v_k)\end{pmatrix}=\det T \det \begin{pmatrix} p(v_1) & \dots & p(v_k)\end{pmatrix},
\]
so $\det T=1$.
This implies that $T$ is in the stabilizer of $g$, and the global $g$-rigidity of $(G,p)$ follows.
\end{proof}

By combining Theorem~\ref{thm:global_determinant} with either Corollary~\ref{cor:cov} or Corollary~\ref{cor:global1}, we can establish a sufficient condition for global rigidity when $g$ is the sum of copies of the determinant map. However, when $g$ is the sum of copies of the product map (as in the case of symmetric tensor completion), we can provide a direct sufficient condition without the need for identifiability tests.
A key ingredient in establishing this condition is the following consequence of a generalised trisecant lemma given in \cite{CC2001}.

\begin{prop}\label{lem:prod_trisecant}
Let $h:=h_{\rm prod}:\mathbb{C}^k\rightarrow \mathbb{C}$ be the product map, $G=(V,E)$ a hypergraph,
and $x_1, x_2, \dots, x_t$ some generic points in $\overline{{\rm im}f_{h,G}}$.
If $t\leq |E|-\rank \J f_{h,G}$, then any point in the intersection of $\overline{{\rm im}f_{h,G}}$ and ${\rm span}\{x_1, x_2,\dots, x_t\}$ is a scalar multiple of some $x_i$. 
\end{prop}
\begin{proof}
A generalised trisecant lemma~\cite{CC2001} states the following.
Suppose ${\cal X}$ is an irreducible, nondegenerate, $d$-dimensional projective variety in $\mathbb{P}^n$.
Then for any set $\{y_1, y_2, \dots, y_t\}$ of generic points in ${\cal X}$, we have ${\rm span}\{y_1, y_2,\dots, y_t\}\cap {\cal X}=\{y_1, y_2,\dots, y_t\}$, provided $t\leq n-d$.
We apply this lemma to $\overline{{\rm im}f_{h,G}}$.
Note that $\overline{{\rm im}f_{h,G}}$ is an affine variety in $\mathbb{F}^{|E|}$, whose dimension is equal to $\rank \J f_{h,G}(z)$ at a generic $z$.
So it remains to check that $\overline{{\rm im}f_{h,G}}$ is irreducible and nondegenerate.
The irreducibility of $\overline{{\rm im}f_{h,G}}$ follows from the fact that it is the image of a polynomial map.
To see the nondegeneracy, recall that $\overline{{\rm im}f_{h,K_n^k}}$ is the affine cone of the Veronese variety of degree $k$, which is known to be nondegenerate (and is easily checked).
Since $\overline{{\rm im}f_{h,G}}$ is a projection of $\overline{{\rm im}f_{h,K_n^k}}$ and a projection preserves nondegeneracy, $\overline{{\rm im}f_{h,G}}$ is nondegenerate.
\end{proof}

\begin{theorem}\label{thm:global_tensor}
Suppose that $g$ is the sum of $d$ copies of $h_{\rm prod}:\mathbb{F}^k\rightarrow \mathbb{F}$.
Then $G=(V,E)$ is globally $g$-rigid if 
there exists a point-configuration $p\in (\mathbb{F}^d)^V\rightarrow \mathbb{F}$  such that 
\begin{itemize}
\item $(G,p)$ is infinitesimally $g$-rigid, 
\item $|E^{(k-1)}|\geq |V|+d$, 
and 
\item $\dim \left(\bigcap_{w\in \ker \J f_{g,G}(p)^{\top}} \ker A_{G,w}\right)=d$.
\end{itemize}
\end{theorem}
\begin{proof}
In the same manner as in the proof of Theorem~\ref{thm:global_determinant}, we may suppose, by a sufficiently small perturbation, that $p$ is generic.
Also, by Proposition~\ref{prop:complex_to_real}, we may assume $\mathbb{F}=\mathbb{C}$.
By Proposition~\ref{prop:G2}, there is $T\in \mathbb{F}^{d\times d}$ such that 
$\nabla g(q(\sigma))=T\nabla g(p(\sigma))$ for every $\sigma\in E^{(k-1)}$.
%and $q(v)=T^{\top} p(v)$ for every $v\in V(G)$.
Let $P$ be a matrix of size $d\times |E^{(k-1)}|$ obtained by aligning $\nabla g(p(\sigma))$ for all $\sigma\in E^{(k-1)}$,
and let $Q$ be the corresponding matrix for $q$.
Then $Q=TP$.

Let us denote the row vectors of $P$  by $x_1,\dots, x_d$ and those of $Q$  by $y_1,\dots, y_d$.
Also, let $H$ be a $(k-1)$-uniform hypergraph on $V$ with the edge set $E^{(k-1)}$.
Since $g$ is the sum of $d$ copies of the product map of $k$ variables, each coordinate of $\nabla g(p(\sigma))$ is also an image of the product map of $(k-1)$ variables. 
So, each $x_i$ is the image of $f_{h,H}$ at a generic point configuration, where $h$ denotes the  product map of $(k-1)$ variables.
Since the dimension of  $\overline{{\rm im}f_{h,H}}$ is upper bounded by $n$ (as it is a projection of the affine cone of a Veronese variety), we can apply Proposition~\ref{lem:prod_trisecant} to deduce that 
any point in the intersection of $\overline{{\rm im}f_{h,H}}$ and 
$\Span \{x_1,\dots, x_d\}$ is a scalar multiple of some $x_i$.
Since each $y_j$ belongs to $\overline{{\rm im}f_{h,H}}$ and $Q=TP$,
\begin{equation}\label{eq:global_tensor1}
\text{each $y_j$ is a scalar multiple of some $x_i$.}
\end{equation}
Also, by Proposition~\ref{prop:matroid}, $q$ is also regular, so $(G,q)$ is infinitesimally $g$-rigid. 
Hence, by Proposition~\ref{prop:G1}, $\rank Q=d$.
So by $Q=TP$ and (\ref{eq:global_tensor1}), 
$T$ is written as $T=\Sigma D$ for some permutation matrix $\Sigma$ and a diagonal matrix $D$. 
By reordering the coordinates of $q$ if necessary, we may assume $\Sigma$ is the identity matrix, so $T=D$.

Let $t_i$ be the $i$-th diagonal entry $T$.
By Proposition~\ref{prop:G3}, $q(v)=Tp(v)$ for all $v\in V(G)$.
Since $f_{g,G}(p)=f_{g,G}(q)$, comparing the $e$-th coordinate for each $e=\{i_1, i_2, \dots, i_k\}\in E$, we obtain 
\begin{equation}\label{eq:global_tensor2}
\left\langle (t_1^k\ t_2^k\ \dots t_k^k), \bigodot_{j=1}^k p(i_j)\right\rangle = \bigodot_{j=1}^k p(i_j),\end{equation}
where $(t_1^k\ t_2^k\ \dots t_k^k)$ is the $k$-dimensional vector whose $i$-th coordinate is $t_i^k$.
Using the same trick as the second paragraph of this proof,
one can check from Proposition~\ref{lem:prod_trisecant} that 
there are $d$ linearly independent vectors among 
$\{\bigodot_{j=1}^k p(i_j): e=\{i_1, i_2, \dots, i_k\}\in E\}$. 
So Equation (\ref{eq:global_tensor2}) implies that 
$t_j^k=1$ for every $j=1,\dots, k$. In other words, each diagonal entry of $T$ is the $k$-th root of unity. 
So $T$ belongs to the stabilizer of $g$, and the global $g$-rigidity of $G$ follows.
 \end{proof}
Note that the theorem is analogous to \cite[Theorem 6.2]{JJT14} which gives a sufficient condition in the matrix completion case. However the proof technique does not apply in the matrix completion case. For example, the second bullet point in the hypotheses always fails in the matrix completion context.

\begin{example}\label{eq:global_tensor}
    Consider the symmetric tensor completion problem of symmetric rank one and order four, that is, $k=4$ and $g=h_{\rm prod}$.
    We analyse two slightly different $4$-uniform hyper-frameworks
     $(G_1,p)$ and $(G_2, p)$, where 
    $G_1=(\{a,b\}, \{aaaa, abbb,bbbb\})$, $G_2=(\{a,b\}, \{aaaa, aabb,bbbb\})$ and $p: a\mapsto x_a, b\mapsto x_b$.
    Then,
    \[
\J f_{g,G_1}(p)=
\kbordermatrix{
 & a & b \\
 aaaa & 4x_a^3 & 0 \\
 abbb & x_b^3 & 3x_ax_b^2 \\
 bbbb & 0 & 4x_b^3
    },
\qquad 
\J f_{g,G_2}(p)=
\kbordermatrix{
 & a & b \\
 aaaa & 4x_a^3 & 0 \\
 aabb & 2x_ax_b^2 & 2x_a^2x_b \\
 bbbb & 0 & 4x_b^3
    }.
    \]
    Since $\rank \J f_{g,G_i}(p)=2$, each $(G_i, p)$ is infinitesimally $g$-rigid.
    
    For $i=1,2$, $\omega_i\in \mathbb{C}^E$ is in the left kernel of $\J f_{g,G_i}(p)$ if 
    \begin{align*}
    &\omega_1(aaaa)=\frac{1}{x_a^4}, \omega_1(abbb)=-\frac{4}{x_ax_b^3}, \omega_1(bbbb)=\frac{3}{x_b^4}, \\
    &\omega_2(aaaa)=\frac{1}{x_a^4}, \omega_2(aabb)=-\frac{2}{x_a^2x_b^2}, \omega_2(bbbb)=\frac{1}{x_b^4}.
    \end{align*}
The corresponding weighted adjacency matrices are 
\begin{align*}
A_{G_1,\omega_1}&=
\kbordermatrix{
 & aaa & abb & bbb \\
 a & 4\omega_1(aaaa) & 0 & \omega_1(abbb) \\
 b & 0 & 3\omega_1(abbb) & 4\omega_1(bbbb)
}
=\kbordermatrix{
 & aaa & abb & bbb \\
 a & \frac{4}{x_a^4} & 0 & -\frac{4}{x_ax_b^3} \\
 b & 0 & -\frac{12}{x_ax_b^3} & \frac{12}{x_b^4}
}\\
A_{G_2,\omega_2}&=
\kbordermatrix{
 & aaa & aab & abb & bbb \\
 a & 4\omega_2(aaaa) & 0 & 2\omega_2(aabb) & 0 \\
 b & 0 & 2\omega_2(aabb) & 0 & 4\omega_2(bbbb)
}
=
\kbordermatrix{
 & aaa & aab & abb & bbb \\
 a & \frac{4}{x_a^4} & 0 & -\frac{4}{x_a^2x_b^2} & 0 \\
 b & 0 & -\frac{4}{x_a^2x_b^2} & 0 & \frac{4}{x_b^4}
}.
\end{align*}
Since $\dim \ker A_{G_1,\omega_1}=1$, we can conclude from Theorem \ref{thm:global_tensor} that $G_1$ is globally $g$-rigid,
whereas we cannot apply Theorem \ref{thm:global_tensor} to $G_2$ since $\dim \ker A_{G_2, \omega_2}>1$.
Indeed $(G_2,p)$ cannot be globally $g$-rigid
since the polynomial system $x_a^4=t_1, x_a^2x_b^2=t_2, x_b^4=t_3$ with variables $x_a, x_b$ and  fixed scalars $t_1, t_2, t_3$ only determines the value of $(x_a^2, x_b^2)$.
(On the other hand for $(G_1, p)$ the corresponding polynomial system $x_a^4=t_1, x_ax_b^3=t_2, x_b^4=t_3$ with variables $x_a, x_b$ and  fixed scalars $t_1, t_2, t_3$ does determine the value of $(x_a, x_b)$
up to simultaneous multiplication by a 4-th root of unity).

We also compute the matrix representation $H_i$ of the Gauss map of 
$\overline{{\rm im} f_{g,G_i}}$ according to the technique by Chiantini, Ottaviani, and Vannieuwenhoven~\cite{COV2014} for each $i=1,2$:
\begin{align*}
H_1&=
\kbordermatrix{
 & a & b \\
 a & 12 x_a^2 \omega_1(aaaa) &  3x_b^2\omega_1(abbb)\\
 b & 3x_b^2 \omega_1(abbb) & 6x_ax_b\omega_1(abbb)+12x_b^2\omega_1(bbbb) 
}=
\kbordermatrix{
 & a & b \\
 a & \frac{12}{x_a^2} &  -\frac{12}{x_ax_b} \\
 b & -\frac{12}{x_ax_b} & \frac{12}{x_b^2}
}\\
H_2&=
\kbordermatrix{
 & a & b \\
 a & 12 x_a^2 \omega_2(aaaa)+2x_a^2\omega_{2}(aabb) &  4x_ax_b\omega_2(aabb)\\
 b & 4x_ax_b \omega_2(aabb) & 12x_b^2\omega_2(bbbb) 
}=
\kbordermatrix{
 & a & b \\
 a & \frac{8}{x_a^2} &  -\frac{8}{x_ax_b} \\
 b & -\frac{8}{x_ax_b} & \frac{8}{x_b^2}
}.
\end{align*}
Hence $H_1$ and $H_2$ are  equal up to scalar multiplication.
\end{example}
\subsection{Necessary conditions}\label{subsec:necessity}
In the context of Euclidean rigidity theory, Hendrickson \cite{Hendrickson1992conditions} established two necessary conditions for a graph to be globally $g$-rigid in dimension $d$. However, we will now provide examples that demonstrate that both of these conditions fail to be necessary for global $g$-rigidity. 

As described in \cite{JJTunique} Hendrickson's argument does not extend to the context of matrix completion, and hence does not extend to an arbitrary polynomial map $g$. Hendrickson's approach involves considering a generic globally $g$-rigid hyper-framework $(G,p)$ and assuming that removing an edge $e \in E(G)$ leads to a configuration $(G-e,p)$ that is not $g$-rigid. The key step in Hendrickson's strategy is to prove that the configuration space $f_{g,G}^{-1}(f_{g,G}(p))/\Gamma$ is compact. However, it turns out that relatively often, this configuration space can be unbounded. While it is possible that some configuration spaces might satisfy the required property to continue Hendrickson's strategy, this is not the case in several natural examples. This indicates that redundant rigidity is not necessarily a prerequisite for global $g$-rigidity in the multilinear case, as demonstrated by the following lemma, providing an infinite family of examples. 
\begin{lemma}\label{lem:0extension_global}
Assume Setup~\ref{setup}, 
and suppose $g$ is multilinear.
Then any simple $d$-valent extension that preserves local $g$-rigidity preserves global $g$-rigidity.
In other words, if $G$ is globally $g$-rigid 
and a simple $d$-valent extension $H$ of $G$ is locally $g$-rigid, then $H$ is globally $g$-rigid.
\end{lemma}
\begin{proof}
Let $G$ be a globally $g$-rigid graph
and suppose $H$ is obtained from $G$ by a simple $d$-valent extension by adding a new vertex $v$.
Pick any generic $p\in (\mathbb{F}^d)^{V(H)}$. We will show that $(H,p)$ is globally $g$-rigid.

For this, choose any $q\in f_{g,H}^{-1}(f_{g,H}(p))$.
Since $G$ is globally $g$-rigid, we may assume 
$q(u)=p(u)$ for all $u\in V(G)$.
Since $g$ is multilinear, $f_{g,H}(p)=f_{g,G}(q)$ leads to the equality:
\[
\langle q(v) \nabla g(q(e_i-v))\rangle= \langle p(v), \nabla g(p(e_i-v))\rangle,
\]
for all hyperedges $e_1,\dots, e_d$ incident to $v$.
Since the $d$-valent extension is simple, $e_i-v$ does not contain $v$, we have that $q(u)=p(u)$ for any $u\in V(G)$ and so:
\[
\langle (p(v)-q(v)), \nabla g(p(e_i-v))\rangle=0,
\]
for $i=1,\dots, d$.
From the fact that the extension preserves local $g$-rigidity and Lemma~\ref{lem:extension}, 
$\{\nabla g(p(e_i-v)): i=1,\dots, d\}$ is a basis of $\mathbb{F}^d$.
So we get $p(v)=q(v)$, implying that $p=q$ and $(G,p)$ is globally $g$-rigid.
\end{proof}

By repeatedly applying the lemma, we can construct an infinite family of hypergraphs that exhibit both minimally local $g$-rigidity and global $g$-rigidity when $g$ is a multilinear polynomial map. 
There are also infinite families of hypergraphs which are locally $g$-rigid but not globally $g$-rigid. 

We next consider Hendrickson's second necessary condition which is in terms of graph connectivity. We will present an analogous condition where the level of connectivity depends strongly on the map $g$. 

Let $G=(V,E)$ be an $k$-uniform hypergraph.
A subset $X\subset V$ is called a {\em separator} if 
$G-X$ is disconnected.
We say that $G$ is $r$-connected if $|V(G)|\geq r+1$ and it has no separator of size less than $r$.

The following lemma extends the well-known fact that, in the $d$-dimensional Euclidean rigidity, $d$-connectivity is necessary for local rigidity.
\begin{lemma}\label{lem:connectivity}
Assume Setup~\ref{setup}, and suppose $|V|\geq n_{\Gamma_g}+1$.
If $G$ is locally $g$-rigid, then $G$ is $n_{\Gamma_g}$-connected.
\end{lemma}
\begin{proof}
We first introduce some terminology for the proof.
Recall that $\mathfrak{g}$ denotes the Lie algebra of $\Gamma_g$.
For a point configuration $q\in (\mathbb{F}^d)^{n}$, 
let ${\rm stab}_\mathfrak{g}(q)=\{\gamma\in \mathfrak{g}: \gamma\cdot q=0\}$.
Clearly ${\rm stab}_\mathfrak{g}(q)$ is a linear subspace of $\mathfrak{g}$.
By definition of $n_{\Gamma_g}$, 
${\rm stab}_\mathfrak{g}(q)=\{0\}$ if and only if $n\geq n_{\Gamma_g}$.

We now move to the proof.
Let $(G,p)$ be a generic framework of $G$,
and suppose that $G$ has a separator  of size less than $n_{\Gamma_g}$.
Since $|V(G)|\geq n_{\Gamma}+1$, $G$ has a separator $X$ of size equal to $n_{\Gamma_g}-1$.
Then we can take $V_1, V_2\subseteq V(G)$ such that 
$V_i\setminus X\neq \emptyset (i=1,2), V(G)=V_1\cup V_2, X=V_1\cap V_2$, and there is no hyperedge intersecting $V_1\setminus X$ and $V_2\setminus X$.
Moreover, since $|X|\geq n_{\Gamma_g}-1$, we have that $|V_i|\geq n_{\Gamma_g}$.

Since $|X|<n_{\Gamma_g}$, 
${\rm stab}_\mathfrak{g}(p_{|X})\neq \{0\}$.
Also, since $p$ is generic and $V_2\setminus X\neq \emptyset$, we have ${\rm stab}_\mathfrak{g}(p_{|V_2})\subsetneq {\rm stab}_\mathfrak{g}(p_{|X})$.
Pick any $\gamma\in {\rm stab}_\mathfrak{g}(p_{|X})\setminus {\rm stab}_\mathfrak{g}(p_{|V_2})$,
and define $\dot{p}\in (\mathbb{F}^d)^{V(G)}$ by 
\[
\dot{p}(v)=\begin{cases}
    0 & (v\in V_1\setminus X) \\
    0=\gamma\cdot p(v) & (v\in  X) \\
    \gamma\cdot p(v) & (v\in V_2\setminus X).
\end{cases}
\]
From the fact that $\gamma\in \mathfrak{g}$ and there is no hyperedge intersecting $V_1\setminus X$ and $V_2\setminus X$, $\dot{p}$ is an infinitesimal $g$-motion of $(G,p)$.
Also, $\dot{p}$ is nonzero since $\gamma\notin {\rm stab}_\mathfrak{g}(p_{|V_2})$.
Moreover, $\dot{p}$ is not trivial.
Indeed, if $\dot{p}$ is trivial,
then there must be some $\gamma'\in \mathfrak{g}$ such that $\dot{p}(v)=\gamma'\cdot p(v)$ for all $v\in V(G)$.
Then, for all $v\in V_1$, $\gamma'\cdot p(v)=\dot{p}(v)=0$, implying 
$\gamma'\in {\rm stab}_\mathfrak{g}(p_{|V_1})$.
Since $|V_1|\geq n_{\Gamma_g}$ and hence
${\rm stab}_\mathfrak{g}(p_{|V_1})=\{0\}$,  this further implies 
$\gamma'=0$.
On the other hand, for any $v\notin V_1$,
$\dot{p}(v)=\gamma' \cdot p(v)=0$, contradicting that $\dot{p}$ is nonzero.
Thus, $(G,p)$ admits a non-trivial infinitesimal $g$-motion,
and $G$ is not locally $g$-rigid.
\end{proof}

Since local $g$-rigidity is a necessary condition for global $g$-rigidity, the same necessary connectivity condition holds for global $g$-rigidity. The rank 1 tensor completion problem shows that this cannot be improved. However in many applications, such as Euclidean rigidity and higher rank tensor completions, one may add one to the connectivity condition; that is for many choices of $g$, global $g$-rigidity implies $(n_{\Gamma_g}+1)$-connectivity.

\section{Multipartite Rigidity Model}\label{sec:variants}

In certain applications, we will encounter the identifiability question  defined using a structure of $k$-partiteness of $k$-partite hypergraphs. We now introduce the corresponding rigidity concept and analyse the unique tensor completion problem as an illustrative example.

\subsection{Birigidity}\label{subsec:rigidity_partite}

We say that a hypergraph $G=(V,E)$ is {\em $k$-partite} if 
there is a partition $\{V_1,\dots, V_k\}$ of $V$ into $k$ nonempty subsets $V_i$ such that $|V_i\cap e|=1$ for any $e\in E$ and $i=1,\dots, k$.
For example, when $k=2$, a $k$-partite hypergraph is an ordinary bipartite graph.

Suppose $(G,p)$ is a $k$-partite hyper-framework.
When we denote a hyperedge $e$ as $e=(v_1,\dots, v_k)\in E$, we mean that $v_i\in V_i$ for $i=1,\ldots,k$.
As in Setup~\ref{setup}, 
for a polynomial map $g:(\mathbb{F}^d)^k\rightarrow \mathbb{F}$,
we define 
the $g$-measurement map $f_{g,G}:(\mathbb{F}^d)^{V}\rightarrow \mathbb{F}^E$.
Since $G$ is $k$-partite, $g$ may not be symmetric or anti-symmetric.

The main difference from the rigidity model given in Section~\ref{sec:rigid} occurs in the action of the affine group to $(\mathbb{F}^d)^k$.
Suppose the general affine group ${\rm Aff}(d,\mathbb{F})$ acts on $\mathbb{F}^d$ by $\gamma\cdot x=Ax+t$ for $\gamma=(A,t)$, as before. 
In the multipartite model, we consider the direct product ${\rm Aff}(d,\mathbb{F})^k$ of $k$ copies of ${\rm Aff}(d,\mathbb{F})$,
and consider the component-wise action of ${\rm Aff}(d,\mathbb{F})^k$ to $(\mathbb{F}^d)^k$.
The induced action on $g:(\mathbb{F}^d)^k\rightarrow \mathbb{F}$ is given by $\gamma\cdot g(x_1,\dots, x_k)=g(\gamma_1^{-1} \cdot x_1,\dots, \gamma_k^{-1} x_k)$ for $x_1,\dots, x_k\in \mathbb{F}^d$ and $\gamma=(\gamma_1,\dots, \gamma_k)\in {\rm Aff}(d,\mathbb{F})^k$.
The stabilizer of $g$ is denoted by $\Gamma_g^{\times}$.

Since in this $k$-partite model the underlying vertex set $V$ is assumed to be partitioned into $\{V_1,\dots, V_k\}$, we can also define the action of ${\rm Aff}(d,\mathbb{F})^k$ to $(\mathbb{F}^d)^V$ such that 
$(\gamma\cdot p)(v)=\gamma_i \cdot p(v)$ for each $v\in V_i$ for
$p\in (\mathbb{F}^d)^V$ and $\gamma=(\gamma_1,\dots, \gamma_k)\in {\rm Aff}(d,\mathbb{F})^k$.
\begin{definition}
    We say that $(G,p)$ is {\em globally $g$-birigid}  if 
for any $q\in f^{-1}_{g,G}(f_{g,G}(p))$ there is $\gamma \in \Gamma_G^{\times}$ such that 
$q=\gamma \cdot p$.
We say that $(G,p)$ is {\em locally $g$-birigid} if there is an open neighbourhood $N$ of $p$ in $(\mathbb{F}^d)^{V}$ (in the Euclidean topology) such that
for any $q\in f^{-1}_{g,G}(f_{g,G}(p))\cap N$ there is $\gamma \in \Gamma_G^{\times}$ such that $q=\gamma \cdot p$.
\end{definition}

\begin{example}\label{eq:rectangular_matrix}
Consider the case when $G$ is 2-uniform (i.e., a graph) and $g(x,y)=\langle x,y\rangle$ for $x,y\in \mathbb{F}^d$.
In the model in Section~\ref{sec:rigid}, 
the stabilizer of $g$ is $O(d)$ and its Lie algebra is ${d\choose 2}$-dimensional.
On the other hand, in the $k$-partite model, 
the stabilizer of $g$ is $\{(A, A^{-\top}): A\in {\rm GL}(d,\mathbb{F})\}$,
and its Lie algebra is $d^2$-dimensional.
\end{example}

We define the infinitesimal $g$-birigidity in the same manner as described in Section~\ref{subsec:infinitesimal_rigidity}. To do so, we denote $d_{\Gamma_g^{\times}}$ for the dimension of the variety $\Gamma_g^{\times}$,
and let  ${\rm triv}_{\Gamma_g^{\times}}(p)$ be the space of trivial infinitesimal motions of $(G,p)$. 
If $|V_i|$ is sufficiently large for all $i$, then $\dim {\rm triv}_{\Gamma_g^{\times}}(p)=d_{\Gamma_g^{\times}}$.
The minimum size of $V_i$ that guarantees this property is denoted by $n_{\Gamma_g^{\times}}$.
Specifically, let $$n_{\Gamma_g^{\times}}:=\min\{|V_i|: \dim {\rm triv}_{\Gamma}(p)=d_{\Gamma_g^{\times}} \text{ for some } p: V\rightarrow\mathbb{F}^d\}.$$

All the materials in Section~\ref{sec:tools} can be extended to the multipartite rigidity model.
For example, the following proposition can be shown easily.

\begin{prop}\label{prop:g-birigidity}
Let $g:(\mathbb{F}^d)^k\rightarrow \mathbb{F}$ be a polynomial map and  $G=(V,E)$ be a $k$-partite hypergraph with 
$|V_i|\geq n_{\Gamma_g^{\times}}$.
Then the following are equivalent:
\begin{itemize}
    \item $(G,p)$ is locally $g$-birigid for some generic $p$.
    \item $(G,p)$ is infinitesimally $g$-birigid for some generic $p$.
    \item The rank of $\J f_{g,G}(p)$ is equal to $d|\MV|-d_{\Gamma_g^{\times}}$ for some generic $p$.
    % \item The rank of the generic $g$-rigidity matroid $\cM_g(G)$ is equal to $d|\MV|-d_{\Gamma_{\rm partite}}$.
    \item $\dim(\overline{{\rm im}f_{g,G}(p)})=d|V|-d_{\Gamma_g^{\times}}$. 
\end{itemize}
\end{prop}

Let $K_{n_1,\dots, n_k}^k$ be the $k$-uniform complete hypergraph with $|V_i|=n_i$.
The generic $g$-birigidity matroid ${\cal M}_{g,n_1,\dots, n_k}^{\times}$ is defined on $E(K_{n_1,\dots, n_k}^k)$ whose independence is defined by the row independence in $\J f_{g,K_{n_1,\dots, n_k}^r}(p)$ at a generic $p$.
If $g$ is symmetric or anti-symmetric, ${\cal M}_{g,n_1,\dots, n_k}^{\times}$ is the restriction of the $g$-rigidity matroid ${\cal M}_{g,n}$ of $K_{n}^r$ to the edge set of $K_{n_1,\dots, n_k}^k$ by regarding 
$K_{n_1,\dots, n_k}^k$ as a subgraph of $K_n^k$ with $n=n_1+\dots+n_k$.

\subsection{Tensor Completion}
\label{subsec:tensor}
A primary example of $g$-birigidity occurs in the rectangular tensor completion problem.
\medskip

Let $W_1,\dots,W_k$ be a collection of finite dimensional vector spaces over $\mathbb{C}$. We denote $W_1\otimes \dots \otimes W_k$ for the set of all order $k$ tensors of dimensions $n_1=\dim(W_1),\dots,n_k=\dim(W_k)$.
We fix a basis for each $W_i$, and assume that each $T\in W_1\otimes \dots \otimes W_k$ is represented by a $k$-dimensional array of numbers in $\mathbb{C}$.
In particular, any tensor can be written as 
\begin{equation}\label{eq:tensor}
T =\sum_{i=1}^d  x_i^1\otimes x_i^2\otimes \dots \otimes x_i^k
\end{equation}
for some  vectors $x_i^j \in W_j$ and $\lambda_1,\dots, \lambda_r\in \mathbb{F}$.
The smallest possible $d$ for which we can write $T$ in the form of Equation (\ref{eq:tensor}) is called the {\it (CP) rank} of $T$. 

As in the case of matrix completion or symmetric tensor completion, we can ask the low rank tensor completion problem as the problem of filling the missing entries of a given partially-filled matrix. Let $V_i=[n_i]$ for $1\leq i\leq k$ and $V$ be the disjoint union of all $V_i$.
We can use a subset $E$ of $V_1 \times \dots \times V_k$ to represent the known entries in the tensor completion problem. In this manner, we encode the underlying combinatorics of each instance of the completion problem using a $k$-partite hypergraph $(V, E)$.

As in the symmetric tensor completion, the decomposition in Equation (\ref{eq:tensor}) can be converted to a form of an algebraic relation among points in $\mathbb{C}^d$.
This gives the following equivalent formulation of the tensor completion problem: Given a $k$-partite hypergraph $G=(V,E)$ with $E\subseteq V_1 \times \dots \times V_k$ and  $a_e\in \mathbb{C}$ for $e\in E$, find $p:V\rightarrow \mathbb{C}^r$ such that 
\begin{equation}\label{eq:system1_partite}
\mathbf{1}\cdot \bigodot_{v\in e} p(v)=a_e \qquad \text{for $e\in E$}.
\end{equation}
Observe that this equation is the same as Equation (\ref{eq:system1}), and we can recast the rectangular tensor completion problem as a special case of the symmetric tensor completion when the underlying hypergraph is  $k$-partite.

The unique completion problem can be formulated as the $g$-birigidity problem using the same polynomial map $g$ as that in the symmetric tensor case, that is, the sum of $d$ copies of the product function $h_{\rm prod}$. The key difference from the symmetric tensor case is the stabilizer $\Gamma_g^{\times}$.
When $k\geq 3$, 
$\Gamma_g^{\times}$ is the set of $k$-tuples $(D_1\Sigma , D_2 \Sigma, \dots, D_k \Sigma)$ such that each $D_i$ is a diagonal matrix of size $d$ with $\prod_{i=1}^k D_i=I_d$
and $\Sigma$ is a permutation matrix.
Then $d_{\Gamma_g^{\times}}=d(k-1)$ and 
$n_{\Gamma_g^{\times}}=1$.

The fact that $d_{\Gamma_g^{\times}}=d(k-1)$, in particular, implies that the rank of the $g$-birigidity matroid of~$K_{n_1,\dots, n_k}^r$ is at most $d(n_1+\dots +n_k)-d(r-1)$.
An obvious question is whether the equality always holds (equivalently, whether $K_{n_1,\dots, n_k}^r$ is locally $g$-birigid). 
This question is equivalent to determining the non-defectivity of Segre varieties.

Indeed, since $g=dh_{\rm prod}$ is the sum of $d$ copies of $h_{\rm prod}$,
 $\overline{{\rm im} f_{dh_{\rm prod},G}}$ is the $d$-secant of $\overline{{\rm im} f_{h_{\rm prod},G}}$.
When $G=K_{n_1,\dots, n_k}^k$,  $\overline{{\rm im} f_{h_{\rm prod},K_{n_1,\dots, n_k}^k}}$ is the affine cone of the Segre variety with parameter $(n_1,\dots, n_k)$ whose dimension is $n_1+\dots +n_k-(k-1)$.
Therefore, $K_{n_1,\dots, n_k}^k$ is locally $g$-birigid
if and only if $\overline{{\rm im} f_{h^{\rm prod},K_{n_1,\dots, n_k}^k}}$ is not $d$-defective.

Unlike the Veronese case, the non-defectivity problem of Segre varieties has not been fully resolved, and it remains an open question in algebraic geometry. Although there is a long-standing conjecture that aims to provide a complete answer, a conclusive proof is still lacking, see, e.g.~\cite{bernardi}.  

\section{Applications}
\label{sec:apps}

Our formulation of $g$-rigidity clarifies a connection between graph rigidity theory and the analysis of secant varieties. We believe this connection is fruitful, and we can obtain further results if we explore structures of specific varieties. As motivating evidence, we briefly provide two examples.

\subsection{Random projection of Secant varieties along coordinate axis}\label{subsec:random}
Let ${\cal V}$ be an affine variety of dimension $m$.
The Noether normalization theorem says that 
if ${\cal V}$ is projected to a random $t$-dimensional space
then with high probability the Zariski closure of the resulting set has dimension equal to $\min\{m,t\}$.
Hence $t=m$ is the threshold to preserve the dimension by a random projection.
As we have seen in Section~\ref{subsec:combinatorial}, this fact is no longer true if the projection is an orthogonal projection along the coordinate axis. 
The question we want to address here is to estimate the threshold dimension $t$ such that a random axis-aligned projection to a $t$-dimensional subspace preserves the dimension. 
We give an asymptotic answer to this question.

Although our strategy works in the full generality of $g$-rigidity, for simplicity of the description we shall focus on the case when ${\cal V}$ is the $d$-secant of the balanced Segre variety.
Indeed this is one of the most important cases in both theory and applications since it answers a question about the sampling probability of entries in the low-rank tensor completion problem to guarantee the unique recovery. 
It is an important problem in machine learning to identify the sampling complexity for the unique recovery of low-rank tensors, see, e.g.~\cite{barak,song}.
%\textcolor{red}{Suggest we add a couple of references - e.g. https://doi.org/10.1016/j.acha.2023.03.007 and https://doi.org/10.1214/18-AOS1694}
%(TODO: add literature on sampling complexity.)

Let us recall the setting of Section~\ref{subsec:tensor}.
Let $K_{n,\dots, n}^k$ be the complete $k$-partite hypergraph with 
$|V_i|=n$,
and let $h_{\rm prod}:\mathbb{C}^k\rightarrow \mathbb{C}$ be the product function. Then ${\cal S}_n:=\overline{{\rm im} f_{h_{\rm prod},K_{n,\dots, n}^k}}$ is the affine cone of the Segre variety with balanced parameter $(n,\dots, n)$ whose dimension is $kn-(k-1)$.

We are interested in the dimension of a random axis-parallel projection of ${\rm Sect}_d({\cal S}_n)$ (recalling the definition of $d$-secant given in Equation (\ref{eqn:secant})).
Namely, we analyse the dimension of $\pi_G({\rm Sect}_d({\cal S}_n))$ for a random subgraph $G$ of $K_{n,\dots, n}^k$. 
We consider the Erdős-Rényi type random subgraph model $\mathbb{G}_{n,t}$, which is 
a probability distribution over the subgraphs of $K_{n,\dots, n}^r$
defined by selecting each  hyperedge with probability $t$ independently. 

Our main result in this subsection is the following.
\begin{theorem}\label{thm:random}
Let $n,k,d$ be positive integers with $n\geq d$, and $c$ be any number with $c>1$.
Let $\mathbb{G}_{n,t}$ be the random subgraph model of $K_{n,\dots, n}^k$, and $G\sim \mathbb{G}_{n,t}$.
If 
\[
1\geq t\geq \frac{kd^{k-1}\log (cdk^2 n)}{n^{k-1}},
\]
then with probability at least $1-\frac{1}{c}$, 
$G$ is locally $dh_{prod}$-birigid, i.e., 
the projection $\pi_G$ preserves the dimension of ${\rm Sect}_d({\cal S}_n)$.
\end{theorem}
Theorem~\ref{thm:random} implies that,
among $n^k$ hypereges in $K_{n,\dots, n}^k$, only $O(n\log n )$ hyperedges are enough to make a random sub-hypergraph  $dh_{prod}$-birigid.
This bound is asymptotically tight (in $n$)
since $G$ has an isolated vertex with high probability if 
$t< \frac{\log n}{n^{k-1}}$.

We prove Theorem~\ref{thm:random} by adapting the argument in \cite{JT22} for proving a corresponding statement for Euclidean rigidity.
The argument in \cite{JT22} is based on the following Matrix Chernoff bound.
For a non-zero positive semidefinite matrix $X$, $\lambda_{\max}(X)$ denotes the largest eigenvalue and $\lambda_{\min}(X)$ denotes the smallest non-zero eigenvalue.

\begin{theorem}\label{thm:chernoff}
Let $E$ be a finite set, $c\geq 1$, and
$X_e$ be a positive semidefinite $m\times m$ matrix indexed by $e$ for each $e\in E$.
Let $r = \rank \sum_{e\in E} X_e$
and $h = \max_{e\in E}{\lambda_{\max}(X_e)}$.
If 
\[
1\geq t\geq  \frac{h\log (rc)}{\lambda_{\min} (\sum_{e\in E} X_e)},
\]
then  for  a subset $E(t)$ of $E$ obtained by taking each $e\in E$ independently with probability $t$, 
we have 
$\rank \sum_{e\in E(t)} X_e = r$ with probability at least $1-\frac{1}{c}$.
\end{theorem}

Theorem~\ref{thm:chernoff} is an adaptation of the Matrix Chernoff bound (see, e.g.,~\cite{tropp}), and a detailed exposition of this form can be seen in \cite{JT22}.

Theorem~\ref{thm:random} is a direct corollary of Theorem~\ref{thm:chernoff} and the following lemma concerning the smallest nonzero singular value of the Jacobian of the measurement map.
The following lemma is based on a discussion with Kota Nakagawa.
The proof can be considered as a quantitative extension of the proof of Theorem~\ref{thm:packing}.
\begin{lemma}\label{lem:singular}
Let $n, d, k$ be positive integers with $k\geq 3$ and $n\geq d$,
and denote $V=V(K_{n,\dots, n}^k)$
and $f=f_{dh_{\rm prod}, K_{n,\dots, n}^k}$.
Then there exists $p:V\rightarrow \mathbb{R}^d$ such that
\begin{itemize}%{description}
\item[{\rm (i)}] $\rank \J f(p)=d(kn-(k-1))$,
\item[{\rm (ii)}] $\lambda_{\min}(\J f(p)^{\top} \J f(p))\geq \frac{n^{k-1}}{d^{k-1}}$, and
\item[{\rm (iii)}] each row of $\J f(p)$ has Euclidean norm at most $\sqrt{k}$ as a vector.
\end{itemize}
\end{lemma}
\begin{proof}
For simplicity of exposition, we assume that $n$ is an integer multiple of $d$. 
% (SHIN: Maybe we have a slightly worse bound in (ii) when $n$ is not an integer multiple of $d$. What would be the best presentation?)
Since $K_{n,\dots, n}^k$ consists of $k$ disjoint vertex sets of size $n$, we denote them by $V_1, V_2,\dots, V_k$.
For $i\in [k]$, we partition $V_i$ into $d$ sets $V_{i,j} \ (j=1,\dots, d)$ of size $n/d$.
We define $p:V\rightarrow \mathbb{R}^d$ by,
%for $v\in V$,
\[
p(v)=
e_j \qquad \text{for  $v\in V_{i,j}$ with $i\in [k]$ and $j\in [n]$},
%\mathbf{1} & \text{if $v\in V_i$ for some $i\geq 4$},
\]
where $e_j$ is the unit vector in $\mathbb{R}^d$ whose $i$-th entry is one and other entries are zero.
%and $\mathbf{1}$ denotes the all-one vector in $\mathbb{R}^d$.
We show the properties of the statement for $\J f(p)$.

The idea of the proof is the same as that of Theorem~\ref{thm:packing}.
Since the definition of $p$ above is very special, the entries of $\J f(p)$ is well structured as in Claim~\ref{claim:decomp} and 
we will see a block-diagonalized form in $\J f(p)^{\top}\J f(p)$ using the structures of the entries.
Then we will analyse each block by applying a known fact from spectral graph theory.

For such an analysis, we first partition a subset of $E(K_{n\dots, n}^k)$ into $2d$ subsets as follows: for $j=1,\dots, d$,
\begin{align*}
A_j&:=\{e=(v_1,v_2,\dots, v_k)\in E(K_{n\dots, n}^k): v_i\in V_{i,j}\ \text{for every $i\in [k]$}\}\\
B_j&:=\{e=(v_1,v_2,\dots, v_k)\in E(K_{n\dots, n}^k): v_i\in V_{i,j}\ \text{for exactly $k-1$ indices $i$ in $[k]$}\}.
\end{align*}
Note that all sets $A_j$ and $B_j$ are disjoint. %Note that all those $A_j$ and $B_j$ are disjoint.
%and they partition a subset of$E(K_{n\dots, n}^k)$.
%Let $H$ be the subgraph of $K_{n\dots, n}^k$ consisting the hyperedges of $\bigcup_{j\in [d]} A_j\cup B_j$,
%and let $f'=f_{dh_{\rm prod},H}$. 

As in the proof of Theorem~\ref{thm:packing}, we use $b(e,v,k)$ to denote the entry of $\J f(p)$ indexed by a hyperedge $e$, vertex $v\in V$, and a coordinate $k\in [d]$.
\begin{claim}\label{claim:decomp2}
Each entry of $\J f(\bp)$ is either $0$ or $1$.
For $e\in E(K_{n\dots,n}^k)$, $v\in V$, and $l\in [d]$,
$b(e,v,l)=1$ holds if and only if either one of the followings holds:
\begin{itemize}
\item[{\rm (i)}] $e\in A_j$ for some $j\in [d]$, $v\in e\cap V_{i,j}$ for some $i\in [k]$, and $l=j$;
%\item[{\rm (ii)}] $e\in A_j$ for some $j\in [d]$, $v\in e\cap V_{i}$ for some $i\geq 4$, and $k=j$;
\item[{\rm (ii)}] $e\in B_j$ for some $j\in [d]$, $v\in e\cap V_{i,j'}$ for some $i\in [k]$ and $j'\in [d]\setminus \{j\}$, and $l=j$.
\end{itemize}
\end{claim}
\begin{proof}
By definition of $\J f(p)$, 
the entry $b(e,v,l)$ %indexed by $e, v, k$ 
is 
the product of the $l$-th coordinates of $p(u)$ over all $u\in e\setminus \{v\}$ if $v\in e$,
and zero otherwise.
Since each coordinate of $p(v)$ is either one or zero, 
$b(e,v,l)$ is either one or zero.

The latter claim can be checked similarly by using the definition of $p$.
\end{proof}

Recall that each row of $\J f(p)$ is indexed by a hyperedge
and each column of $\J f(p)$ is indexed by a pair $(v,j)$
of $v\in V$ and $j\in [d]$,
and hence each entry of $\J f(p)^{\top} \J f(p)$ is indexed by a pair $((u,j),(v,j'))$ of such pairs.
For $j=1,\dots, d$, let 
\begin{align}
I_j&=\{(v,j): v\in V_{i,j} \text{ for some } i\in [k]\}, \mbox{ and} \\
I_j'&=\{(v,j): v\in V_i\setminus V_{i,j} \text{ for some } i\in [k]\}.
\end{align}
Note that $I_j, I_j'\ (j=1,\dots, d)$ are mutually disjoint and they partition $V\times [d]$.
Also, $|I_j|=k\frac{n}{d}$
and $|I_j'|=k(n-\frac{n}{d})$.

\begin{claim}\label{claim:entries2}
$\J f(p)^{\top} \J f(p)$ has the following block-diagonalized form:
\[
\J f(p)^{\top} \J f(p)=
\kbordermatrix{
& I_1 & I_1'  & \dots & I_k & I_k' \\
 I_1 & Q_1 & 0 & \dots   & \dots & 0 \\ 
 I_1' & 0 & R_1 &  0 & \dots  & 0 \\
 \vdots & \vdots & & \ddots &  & \vdots  \\
 I_k & 0 & & &  Q_k & 0 \\
 I_k' & 0 & \dots & &  0 & R_k 
}.
\]
Moreover, in $Q_j$ (resp. $R_j$), the value of the entry indexed by $((u,j), (v,j))$ is equal to the number of hyperedges in $A_j$ (resp. $B_j$) that contain $u$ and $v$.
\end{claim}
\begin{proof}
For each $e\in E(K_{n,\dots n}^k)$, let $R_e$ be the row vector of $\J f(p)$ associated with $e$,
and let $X_e=R_e^{\top} R_e$ by regarding $R_e$ as a matrix of size $1\times dn$.
Note that $\J f(p)^{\top} \J f(p)=\sum_{e\in E(K_{n,\dots n}^k)} X_e$.
Hence, it is enough to check the contribution of each $X_e$ to each entry of $\J f(p)^{\top} \J f(p)$.

Claim~\ref{claim:decomp2} implies that $X_e$ is a $0-1$ matrix.
Claim~\ref{claim:decomp2}(i) further implies that, if $e=(v_1, v_2,\dots, v_k)\in A_j$, 
then $R_e$ has nonzero entries only at the coordinates indexed by $(v_1,j), (v_2, j), \dots, (v_k,j)$.
Hence, $X_e$ contributes to increase 
the entry indexed by $((u,j),(v,j))$ for each $u,v\in e$ by one.
This is the reason why, in $Q_j$ the value of the entry indexed by $((u,j), (v,j))$ is equal to the number of hyperedges in $A_j$ that contain both $u$ and $v$.

A similar argument using Claim~\ref{claim:decomp2}(ii) also provides the entries of $R_j$.
% Similarly, Claim~\ref{claim}(iii) imply that, if $e=(v_1, v_2,\dots, v_k)\in B_j$, 
% then $R_e$ have only one nonzero entry at the coordinate indexed by $(v_i,j)$ (where $i$ is the unique $i$ satisfying $v_i\notin V_{i,j}$ in the definition of $B_j$).
\end{proof}

Observe that, for $(u,j), (v,j)\in I'_j$, 
the number of hyperedges in $B_j$ that contain $u$ and $v$ 
is $\left(\frac{n}{d}\right)^{k-1}$ if $u=v$,
and it is zero otherwise (i.e., when $u\neq v$).
So Claim~\ref{claim:entries2} implies that
\[
R_j = \frac{n^{k-1}}{d^{k-1}} I
\]
where $I$ denotes the identity matrix (of an appropriate size).
Hence, by the block-diagonal form in Claim~\ref{claim:entries2}, in order to show that $\rank \J f(p)=d(kn-(k-1))$ and 
$\lambda_{\min}(\J f(p)^{\top} \J f(p))\geq \frac{n^{k-1}}{d^{k-1}}$,
it suffices to show that 
$\dim \ker  Q_j= k-1$ and 
$\lambda_{\min}(Q_j)\geq \frac{n^{k-1}}{d^{k-1}}$ for each $j=1,\dots, d$.

Now we analyse $Q_j$ based on Claim~\ref{claim:entries2}.
By counting the number of hyperedges in $A_j$ that contain two vertices $u$ and $v$, we see that 
$Q_j$ is equal to the adjacency matrix of the edge-weighted complete $k$-partite graph on $V_{1,j}\cup V_{2,j}\cup \dots \cup V_{k,j}$ with a loop at each vertex, where the edge-weight $w$ is defined by
\[
w(uv) = 
\begin{cases}
\left(\frac{n}{d}\right)^{k-1} & (u=v \in V_{i,j} \text{ for some } i\in [k]) \\
\left(\frac{n}{d}\right)^{k-2} &  (u\in V_{i,j}, v\in V_{i',j} \text{ for some } i, i'\in [k], i\neq j) 
\end{cases}
\]
for each edge $uv$.
In other words,
\[
Q_j=\left(\frac{n}{d}\right)^{k-2} A_{K_{n/d,\dots, n/d}}+ \left(\frac{n}{d}\right)^{k-1} I
\]
where $A_{K_{n,\dots, n}}$ is the adjacency matrix of 
the (non-weighted) complete $k$-partite graph with $n/d$ vertices on each side and $I$ denotes the identity matrix of size $\frac{kn}{d}$.
It is a well-known fact~\cite{brouwer} from spectral graph theory that the spectrum of $A_{K_{n/d,\dots, n/d}}$
is $[(\frac{kn}{d}-\frac{n}{d})^1, 0^{k(\frac{n}{d}-1)}, (-\frac{n}{d})^{k-1}]$, where
$a^b$ means $a$ is an eigenvalue of multiplicity $b$.
So the spectrum of $Q_j$ is 
$[\left(k\frac{n^{k-1}}{d^{k-1}}\right)^1, \left(\frac{n^{k-1}}{d^{k-1}}\right)^{k(n-1)}, (0)^{k-1}]$.
Thus, $\dim \ker  Q_j=k-1$ and $\lambda_{\min}(Q_j)=\frac{n^{k-1}}{d^{k-1}}$.
This completes the proof.
\end{proof}

\begin{proof}[\bf{Proof of Theorem~\ref{thm:random}}]
Pick $p:V\rightarrow \mathbb{R}^d$ given in Lemma~\ref{lem:singular},
and for simplicity let us denote $R=\J f_{dh_{\rm prod}, K_{n,\dots, n}^k}(p)$.
Recall that
each row in $R$ is associated with an edge in $K_{n,\dots, n}^k$.
For each $e\in E(K_{n,\dots,n}^k)$, let 
$R_e$ be the row vector of $R$ associated with $e$.
By regarding each $R_e$ as a matrix of size $1\times dn$, let $X_e=R_e^{\top} R_e$.
Then $R^{\top}R=\sum_{e\in E(K_{n,\dots,n}^k)} X_e$.

Let $G\sim \mathbb{G}_{n,t}$ be a random subgraph,
and denote $R_G=\J f_{dh_{\rm prod}, G}(p)$.
Then $R_G^{\top}R_G=\sum_{e\in E(G)} X_e$,
and one can apply Theorem~\ref{thm:chernoff} to get a lower bound of $t$ for $G\sim \mathbb{G}_{n,t}$ to have $\rank R_G^{\top} R_G=\rank R^{\top}R$.
Specifically, since $\lambda_{\min}(R^{\top} R)\geq \frac{n^{k-1}}{d^{k-1}}$
by Lemma~\ref{lem:singular}(ii)
and $\lambda_{\max}(X_e)=k$ by Lemma~\ref{lem:singular}(iii),
Theorem~\ref{thm:chernoff} ensures that $\rank R_G= \rank R$ holds
when
\[
1\geq t\geq \frac{kd^{k-1} \log (cdk(kn-(k-1)))}{n^{k-1}}.
\]
Since $\rank R_G = \rank R = d(kn-(k-1))$,
the local  $dh_{\rm prod}$-birigidity of $G$ follows.
\end{proof}

It is an open problem to establish the global rigidity counterpart of Theorem~\ref{thm:random}.
For Euclidean rigidity, there are various methods available to establish a connection between local rigidity and global rigidity. Extending these methods would be a fruitful direction for further research.

\subsection{Identifiability of projections of $p$-Cayley-Menger varieties and $\ell_p$-norm global rigidity}
In Section~\ref{sec:examples}, we briefly looked at graph rigidity in $\ell_p$-space, where 
instead of conventional Euclidean distances 
we are concerned with $\ell_p$-distances between points in a real vector space.
An extensive study has been done for $\ell_p$ local rigidity,
but $\ell_p$ global rigidity is not yet well understood~\cite{Dewar23,Dewar22gen}.
A major open problem in this context is whether global rigidity in $\ell_p$-space is a generic property of graphs as in the case of Euclidean global rigidity by Gortler-Hearly-Thurston~\cite{GHT10},
see a discussion in Section~\ref{sec:global}.

The difficulty to extend techniques in Euclidean global rigidity is the lack of an obvious counterpart definition of stress matrices.
The theory of stress matrices, developed by Connelly~\cite{Con93,Con05}, is currently the central tool in Euclidean global rigidity analysis, and it heavily relies on the fact that the squared Euclidean distance is quadratic.
Interestingly, $t$-tangentially weakly non-defectiveness discussed in Section~\ref{subsec:identifiability} provides a counterpart tool.

To see this, we shall introduce the $\ell_p$-analogue of the Cayley-Menger variety.
For a positive integer $p$, 
define $h_p:\mathbb{C}^2\rightarrow \mathbb{C}$ by 
\[
h_p(x_1,x_2)=(x_1-x_2)^p.
\]
When $p=2$ and the domain is real, $f_{h_2,K_n}(x)$ is
the list of square Euclidean distances among $n$ points $x=(x_1,x_2, \dots, x_n)$ on a line,
and the image of $f_{h_2, K_n}$ is called the {\em Cayley-Menger variety}.
Motivated by this, we call the image of $f_{h_p, K_n}$ the {\em $p$-Cayley-Menger variety} ${\cal CM}_n^p$.

Let $g_p$ be the sum of $d$ copies of $h_p$.
If $p$ is a positive even integer and the domain is restricted to real numbers, then $f_{g_p, K_n}$ is the measurement map that outputs the list of $p$-th powered $\ell_p$-distances in $d$-dimensional real vector space.
Hence, the global $g_p$-rigidity of a $d$-dimensional framework $(G,p)$ in our terminology is equivalent to  the global rigidity of $(G,p)$ in $d$-dimensional $\ell_p$-space.

Since $g_p$ is the sum of $d$ copies of $h_p$,
one may try to apply Corollary~\ref{cor:cov} to show the global $g_p$-rigidity of generic frameworks.
To apply Corollary~\ref{cor:cov}, we need to understand when 
$\overline{\pi_G({\cal CM}_n^p)}$ is $d$-tangentially weakly nondefective
and when $G$ is globally $h_p$-rigid.
It is a folklore fact that global rigidity in 1-dimensional $\ell_p$-space is characterised by the 2-connectivity of the underlying graph,
which in turn implies that $G$ is globally $h_p$-rigid if and only if $G$ is 2-connected.
In \cite{sugiyama}, Sugiyama and the fourth author 
further gave a graph theoretical characterisation of the $2$-tangentially weak non-defectiveness of $\overline{\pi_G({\cal CM}_n^p)}$.

\begin{theorem}[Sugiyama and Tanigawa~\cite{sugiyama}]\label{thm:Lp}
Let $p$ be an even positive integer with $p\neq 2$, $n\geq 3$ a positive integer, and $G$ a connected graph with $n$ vertices.
Then the following are equivalent:
\begin{itemize}
\item[{\rm (i)}] Some/every generic $2$-dimensional framework of $G$ is globally rigid in the $\ell_p$-plane.
\item[{\rm (ii)}] $\overline{\pi_G({\cal CM}^p_n)}$ is $2$-identifiable.
\item[{\rm (iii)}] $\overline{\pi_G({\cal CM}^p_n)}$ is $2$-tangentially weakly nondefective.
\item[{\rm (iv)}] $G$ is $2$-connected and $G-e$ contains two edge-disjoint spanning trees for every $e\in E(G)$.
\item[{\rm (v)}] $G$ is $2$-connected and $G-e$ is locally rigid in the $\ell_p$-plane for every $e\in E(G)$.
\end{itemize}
\end{theorem}

To the best of our knowledge, such a purely graph-theoretical characterisation of the identifiability of an orthogonal projection of a variety is new.
Also, the theorem  solves the generic global rigidity problem in $\ell_p$-planes.

We should remark that the proof of Theorem~\ref{thm:Lp} relies on sophisticated graph theoretical techniques from graph rigidity theory.
In particular, the proof of (iv) $\Rightarrow$ (iii) uses the inductive construction due to Dewar, Hewetson and Nixon~\cite{Dewar2022}, who showed that a graph satisfying the combinatorial condition (iv) can be built up from a small graph by recursively applying two local graph-operations, called $K_4^-$-extension and generalised vertex splitting.
Sugiyama and Tanigawa~\cite{sugiyama} have confirmed that each local graph-operation preserves the 2-tangentially weakly nondefective of $\overline{\pi_G({\cal CM}^p_n)}$, which implies (iv) $\Rightarrow$ (iii).
The proof of (iii) $\Rightarrow$ (ii) $\Rightarrow$ (i)
is done by applying a general tool from algebraic geometry explained in Section~\ref{subsec:identifiability}.
Finally, (i) $\Rightarrow$ (v) and (v) $\Rightarrow$ (iv)
have been already shown in~\cite{Dewar2022} and~\cite{KP14}, respectively.

%follows from the characterisation of local rigidity.
%Finally,  has been already shown in~\cite{Dewar2022} based on Hendrickson's argument (cf.~Section~\ref{subsec:necessity}). 

We believe Theorem~\ref{thm:Lp} is valid in general dimension.

\begin{conjecture}
Let $p$ be an even positive integer with $p\neq 2$, $n\geq 3$ be a positive integer, and $G$ be a connected graph with $n$ vertices.
Then the following are equivalent:
\begin{itemize}
\item[{\rm (i)}] Some/every generic $d$-dimensional framework of $G$ is globally rigid in $d$-dimensional $\ell_p$-space.
\item[{\rm (ii)}] $\overline{\pi_G({\cal CM}^p_n)}$ is $d$-identifiable.
\item[{\rm (iii)}] $\overline{\pi_G({\cal CM}^p_n)}$ is $d$-tangentially weakly nondefective.
\item[{\rm (iv)}] $G$ is 2-connected and $G-e$ contains $d$ edge-disjoint spanning trees for every $e\in E(G)$.
\item[{\rm (v)}] $G$ is 2-connected and $G-e$ is locally rigid in the $\ell_p$-plane for every $e\in E(G)$.
\end{itemize}
\end{conjecture}

By the theory in Section~\ref{sec:global},
the proofs of (iii) $\Rightarrow$ (ii) $\Rightarrow$ (i) $\Rightarrow$ (v) in Theorem~\ref{thm:Lp} still work in general dimension.
(v) $\Rightarrow$ (iv) also follows from the well-known Maxwell count argument.
Currently, the proof for (iv) $\Rightarrow$ (iii) is missing.

\section{Conclusion}
\label{sec:concluding}

% \Tony{We discussed adding a paragraph here that makes it easy for readers to see clearly the differences between global rigidity, identifiability, local rigidity and non-defectivity. This could be specific additional examples or pointers to the relevant parts of the paper}

We have introduced a far reaching generalisation of graph rigidity theory to the general setting of identifiability of points, which encompasses diverse applications.  Throughout the paper, we have provided analogues of several fundamental theorems in this generalised framework. %, and given analogues of a number of cornerstone theorems in this generality.
We now conclude the paper by describing a number of open problems and potential extensions of our work.

From a rigidity theoretic perspective, a natural question arises regarding the search for combinatorial characterisations of (global) $g$-rigidity. A notable instance is presented in Conjecture \ref{con:genericform}, where we propose a potential characterisation of $g$-rigidity in the case when $g$ is of a generic form. Establishing such combinatorial characterisations would provide insights into the fundamental properties of $g$-rigidity. 

As mentioned earlier, we expect that Theorem~\ref{thm:global_determinant} holds for a broader class of polynomial maps $g$, extending beyond the specific cases explored in the paper. A particularly interesting to resolve would be the case when $g$ is the permanent function. Exploring the rigidity properties associated with the permanent could shed light on its connection to other combinatorial structures.

In Section~\ref{sec:global} we have shown a sufficient condition for global $g$-rigidity using weighted adjacent matrices,
which is an analogue of Connelly's stress-matrix condition. An important remaining question is to establish a graph theoretical property to guarantee the applicability of the theorem.

Another interesting direction for future research would be to develop sufficient conditions for (global) $g$-rigidity in the multipartite model discussed in Section \ref{sec:variants}.

We would like to add that Corollary~\ref{cor:packing_Kn} also has significant application in statistics, in particular in the identifiability problem of Gaussian mixture models, which aims to determine the minimum number of parameters necessary to uniquely determine the statistical model given $n$ sample points. The goal is to find parameters whose corresponding mixture function approximates the sample points. The method of moments is a commonly used approach for parameter recovery in Gaussian mixture models (see, e.g.~\cite{%ceron2017algebraic,
agostini2021moment, amendola2018algebraic}). In the method of moments, the identifiability problem is reduced to determining the non-defectivity %computing the dimension 
of secants of the corresponding varieties. However, this is only done for the specific cases, e.g. when the covariance matrix is zero, as in these cases the moment variety corresponds to the classical Veronese varieties, whose secants' non-defectivity are determined by the Alexander-Hirschowitz Theorem. 
Therefore, we hope that our arguments will shed light on the more general case of the problem. In particular, it would be interesting to capture the identifiability of Gaussian mixture models as specific instances of the $g$-rigidity model.

%\Shinichi{Combinatorial sufficient conditions for global rigidity}
%A fundamental operation in graph rigidity theory is the coning operation. This operation adds one vertex $v$ to a graph $G$ and forms $G'$ by connecting $v$ to every vertex of $G$. Whiteley \cite{Wcone} proved that a graph is $g$-rigid in $\mathbb{R}^d$ (where $g$ is the Euclidean distance) if and only if $G'$ is $g$-rigid in $\mathbb{R}^{d+1}$. For which other $g$-rigidity settings can such a coning theorem be proved? In particular is it true in the tensor completion setting?

%\Tony{Some probabilistic conjecture? e.g. a random XXX hypergraph will almost surely be $g$-rigid?}

%\Tony{It is worth trying to check if vertex splitting preserves $g$-rigidity...} 

%\Jim{In the case of rigidity the 1-d rigidity matroid coincides with the graphic matroid. Is there a similar result for k-uniform hypergraphs. For example, does the k-uniform simplicial matroid arise as the $g$-rigidity matroid for some natural $g$?}
%\Shinichi{The simplicial matroid is defined based on an adjacency relation between k-sets and (k-1)-sets. On the other hand, in the $g$-rigidity framework, we are looking at a "relation" between k-sets and 1-sets (vertices). So I think it is unlikely that the simplicial matroid arises as an example of $g$-rigidity. In fact the rank of the simplicial matroid is non-linear with respect to the number of vertices whereas the rank of the $g$-rigidity matroid is linear.} 

\subsection*{Acknowledgements}

This project grew from discussions at the Fields Institute Thematic Program on Geometric Constraint Systems, Framework Rigidity, and Distance Geometry, and we are grateful to the organizers for bringing us together. 
F.M. was partially supported by the KU Leuven grant iBOF/23/064, the UiT Aurora project MASCOT, and the FWO grants
G0F5921N (Odysseus) and G023721N. %grants from the Research Foundation - Flanders (FWO).
A.N. was partially supported by EPSRC grant numbers EP/W019698/1 and EP/X036723/1. 
S.T. was partially supported by JST ERATO
Grant Number JPMJER1903, JST PRESTO Grant Number JPMJPR2126, and JSPS KAKENHI Grant Number 20H05961.
S.T.~would like to thank Bill Jackson for pointing out a simpler proof of Proposition~\ref{prop:complex_to_real}
and Kota Nakagawa for stimulating discussion on the $g$-rigidity of random hypergraphs.

\smallskip
\bibliographystyle{abbrv} 
\bibliography{Rigidity.bib}

\bigskip 

\noindent
\footnotesize {\bf Authors' addresses:}

\medskip

\noindent{School of Mathematical and Statistical Sciences,
University of Galway,
Ireland
\hfill {\tt james.cruickshank@universityofgalway.ie}
}

\noindent{Departments of Mathematics and Computer Science, KU Leuven, Belgium 
%\\
%Department of Mathematics \& Statistics,  University of Troms\o, Norway} 
\hfill {\tt fatemeh.mohammadi@kuleuven.be}

\noindent{Mathematics and Statistics, 
Lancaster
University,
%LA1 4YF, 
UK
\hfill {\tt a.nixon@lancaster.ac.uk}
}

\noindent{Department of Mathematical Informatics, University of Tokyo, Japan}  \hfill {\tt tanigawa@mist.i.u-tokyo.ac.jp}

\end{document}